\documentclass[10pt,twoside,a4paper,english,final]{article}

\title{Solitary wave solutions to a class of modified Green-Naghdi systems}
\author{Vincent Duch\^{e}ne \and Dag Nilsson \and Erik Wahl\'{e}n}
\date{\today}

\usepackage{babel} 
\usepackage[utf8]{inputenc}
\usepackage[T1]{fontenc}
\usepackage{amsmath, amsthm,amsfonts,amssymb}
\usepackage{float} 
\usepackage[pdftex]{graphicx}
\usepackage[hang,centerlast]{subfigure}
\usepackage{url,hyperref}
\usepackage{mathtools}

\makeatletter
\let\Title\@title
\let\Author\@author
\makeatother

 \usepackage{fancyhdr}     
\fancypagestyle{mystyle}{%
	\fancyfoot{}
	\fancyhead[CO]{\Title} 
	\fancyhead[CE]{\scriptsize Vincent Duch\^{e}ne, Dag Nilsson and Erik Wahl\'{e}n} 
	\fancyhead[RE]{\scriptsize\today}  
	\fancyhead[LO]{}  
	\fancyhead[LE,RO]{\scriptsize\thepage}
	}
\fancypagestyle{plain}{\fancyhead{} } 

\newcommand{\RR}{\mathbb{R}}

\newcommand{\ZZ}{\mathbb{Z}}
\newcommand{\NN}{\mathbb{N}}

\newcommand{\A}{\mathcal{A}}

\newcommand{\Q}{\mathcal{Q}}
\newcommand{\I}{\mathcal{I}}
\newcommand{\C}{\mathcal{C}}
\newcommand{\R}{\mathcal{R}}
\renewcommand{\H}{\mathcal{H}}
\renewcommand{\O}{\mathcal{O}}
\newcommand{\E}{\mathcal{E}}

\newcommand{\dd}{{\rm d}}
\newcommand{\eqdef}{\stackrel{\rm def}{=}}

\newcommand{\F}{{\sf F}}
\DeclareMathOperator{\sech}{sech}

\DeclareMathOperator*{\argmin}{arg\, min}
\DeclareMathOperator*{\supp}{supp}
\DeclareMathOperator*{\dist}{dist}

\newcommand{\ie}{{\em i.e.}~}
\newcommand{\eg}{{\em e.g.}~}

\newcommand{\dsp}{\displaystyle}

\DeclarePairedDelimiter\abs{\lvert}{\rvert}
\DeclarePairedDelimiter\Norm{\big\lVert}{\big\rVert}

\newtheorem{Theorem}{Theorem}[section]
\newtheorem{Definition}[Theorem]{Definition}
\newtheorem{Proposition}[Theorem]{Proposition}
\newtheorem{Corollary}[Theorem]{Corollary}
\newtheorem{Lemma}[Theorem]{Lemma}
\newtheorem{Remark}[Theorem]{Remark}
\newtheorem{Assumption}[Theorem]{Assumption}

\numberwithin{equation}{section}

\begin{document}
\maketitle

\begin{abstract}
We provide the existence and asymptotic description of solitary wave solutions to a class of modified Green-Naghdi systems, modeling the propagation of long surface or internal waves. This class was recently proposed by Duch\^ene, Israwi and Talhouk~\cite{DucheneIsrawiTalhouk16} in order to improve the frequency dispersion of the original Green-Naghdi system while maintaining the same precision. The solitary waves are constructed from the solutions of a constrained minimization problem. The main difficulties stem from the fact that the functional at stake involves low order non-local operators, intertwining multiplications and convolutions through Fourier multipliers.
\end{abstract}

\section{Introduction}

\subsection{Motivation}
In this work, we study solitary traveling waves for a class of long-wave models for the propagation of surface and internal waves. Starting with the serendipitous discovery and experimental investigation by John Scott Russell, the study of solitary waves at the surface of a thin layer of water in a canal has a rich history~\cite{Darrigol03}. In particular, it is well-known that the most widely used nonlinear and dispersive models for the propagation of surface gravity waves, such as the Korteweg-de Vries equation or the Boussinesq and Green-Naghdi systems, admit explicit families of solitary waves~\cite{Boussinesq72,Rayleigh76,KortewegDe95,Serre53a,Chen98a}. These equations can be derived as asymptotic models for the so-called water waves system, describing the motion of a two-dimensional layer of ideal, incompressible, homogeneous, irrotational fluid with a free surface and a flat impermeable bottom; we let the reader refer to~\cite{Lannes} and references therein for a detailed account of the rigorous justification of these models. Among them, the Green-Naghdi model is the most precise, in the sense that it does not assume that the surface deformation is small. However, the validity of all these models relies on the hypothesis that the depth of the layer is thin compared with the horizontal wavelength of the flow and, as expected, the models do not describe the system accurately (for instance the dispersion relation of infinitesimally small waves) in a deep water situation. In order to tackle this issue, one of the authors has recently proposed in~\cite{DucheneIsrawiTalhouk16} a new family of models:
\begin{equation}\label{GN-1layer}
\left\{ \begin{array}{l}
\displaystyle\partial_{ t}{\zeta} \ + \ \partial_x w \ =\ 0,  \\ \\
\displaystyle
\partial_{ t} \left( h^{-1}w +\Q^{\F}[h](h^{-1}w)\right)\ + \ g\partial_x{\zeta} \ + \ \frac{1}{2} \partial_x\Big(( h^{-1} w)^2\Big)\ = \partial_x\big( \R^{\F}[h,h^{-1} w]\big),
\end{array}
\right.
\end{equation}
where
\begin{align*}
\dsp\Q^{\F}[h]\overline{u}&\eqdef -\frac13 h^{-1}\partial_x \F\big\{h^3 \partial_x \F\{\overline{u}\}\big\},  \\ 
\dsp \R^{\F}[h,\overline{u}]&\eqdef \frac13 \overline{u} h^{-1}\partial_x \F\big\{h^3\partial_x \F \{\overline{u}\}\big\}+\frac12 \big(h\partial_x \F\{\overline{u}\}\big)^2.
\end{align*}
Here, $\zeta$ is the surface deformation, $h=d+\zeta$ the total depth (where $d$ is the depth of the layer at rest), $\overline{u}$ the layer-averaged horizontal velocity,  $w=h \overline{u}$ the horizontal momentum and $g$ the gravitational acceleration; see Figure~\ref{F.sketch}.
Finally,  $\F\eqdef  \F( D)$ is a Fourier multiplier, \ie
\[ \widehat{\F \{\varphi\}}(k) =\F( k ) \widehat{\varphi}(k).\]
The original Green-Naghdi model is recovered when setting $\F(k)\equiv 1$. Any other choice satisfying $\F(k)=1+\O(k^2)$ enjoys the same precision (in the sense of consistency) in the shallow-water regime and the specific choice of $\F(k)=\sqrt{\frac{3}{d|k|\tanh(d|k|)}-\frac3{ d^2|k|^2}}$ allows to obtain a model whose linearization around constant states fits exactly with the one of the water waves system.
Compared with other strategies such as the Benjamin-Bona-Mahony trick~\cite{BenjaminBonaMahony72,MadsenMurraySorensen91} or using different choices of velocity unknowns~\cite{BonaSmith76,Nwogu93}, an additional advantage of~\eqref{GN-1layer} is that it preserves the Hamiltonian structure of the model, which turns out to play a key role since the existence of solitary waves will be deduced from a variational principle.

\begin{figure}[htb]
\includegraphics[width=.5\textwidth]{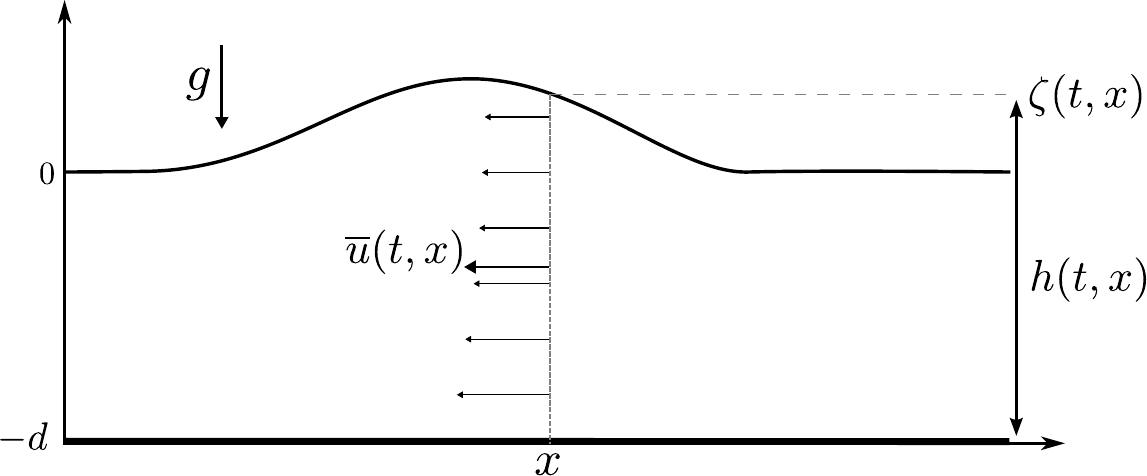}%
\includegraphics[width=.5\textwidth]{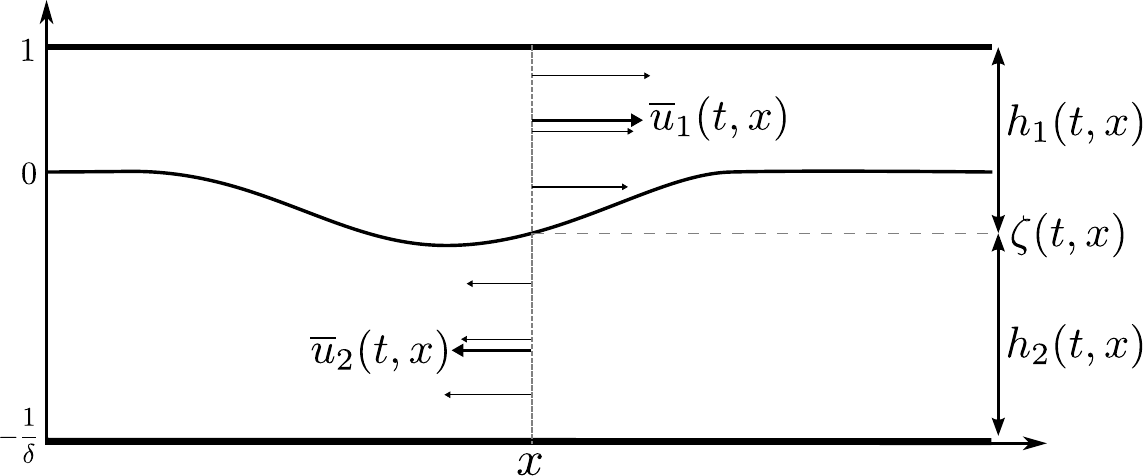}
\caption{Sketch of the domain and notations in the one-layer and bilayer situations.}
\label{F.sketch}
\end{figure}

The study of~\cite{DucheneIsrawiTalhouk16} is not restricted to surface propagation, but is rather dedicated to the propagation of internal waves at the interface between two immiscible fluids, confined above and below by rigid, impermeable and flat boundaries. Such a configuration appears naturally as a model for the ocean, as salinity and temperature may induce sharp density stratification, so that internal solitary waves are observed in many places~\cite{OstrovskyStepanyants89,Jackson04,HelfrichMelville06}. Due to the weak density contrast, the observed solitary waves typically have much larger amplitude than their surface counterpart, hence the bilayer extension of the Green-Naghdi system introduced by~\cite{Miyata87,Malcprimetseva89,ChoiCamassa99}, often called Miyata-Choi-Camassa model, is a very natural choice. It however suffers from strong Kelvin-Helmholtz instabilities --- in fact stronger than the ones of the water waves system for large frequencies --- and the work in~\cite{DucheneIsrawiTalhouk16} was motivated by taming these instabilities. 
The modified bilayer system reads
\begin{equation}\label{GN}
\left\{ \begin{array}{l}
\displaystyle\partial_{ t}{\zeta} \ + \ \partial_x w \ =\ 0,  \\ \\
\displaystyle
\partial_{ t} \left( \frac{h_1+\gamma h_2}{h_1 h_2}w + \Q^{\F}_{\gamma,\delta}[\zeta]w\right)\ + \ (\gamma+\delta)\partial_x{\zeta} \ + \ \frac{1}{2} \partial_x\Big(\frac{h_1^2 -\gamma h_2^2 }{(h_1 h_2)^2}w^2\Big)  = \partial_x\big( \R^{\F}_{\gamma,\delta}[\zeta,w]\big)
\end{array}
\right.
\end{equation}
where we denote $h_1=1-\zeta$, $h_2=\delta^{-1}+\zeta$, $
\Q^{\F}_{\gamma,\delta}[\zeta]w\eqdef \Q_2^{\F}[h_2](h_2^{-1}w)+\gamma \Q_1^{\F}[h_1](h_1^{-1} w) $ and $
 \R^{\F}_{\gamma,\delta}[\zeta,w]\eqdef \R_2^{\F}[h_2,h_2^{-1}w]-\gamma \R_1^{\F}[h_1,h_1^{-1}w]$, with
\begin{align*}
\dsp\Q_i^{\F}[h_i]\overline{u}_i&\eqdef -\frac13 h_i^{-1}\partial_x \F_i\big\{h_i^3 \partial_x \F_i\{\overline{u}_i\}\big\},  \\ 
\dsp \R_i^{\F}[h_i,\overline{u}_i]&\eqdef \frac13 \overline{u}_i h_i^{-1}\partial_x \F_i\big\{h_i^3\partial_x \F_i \{\overline{u}_i\}\big\}+\frac12 \big(h_i\partial_x \F_i\{\overline{u}_i\}\big)^2.
\end{align*}
 Here, $\zeta$ represents the deformation of the interface, 
 $h_1$ (resp.~$h_2$) is the depth of the upper (resp.~lower) layer,
 $\overline{u}_1$ (resp.~$\overline{u}_2$) is the layer-averaged horizontal velocity of the upper (resp.~lower) layer and  $w=h_1h_2(\overline{u}_2-\gamma \overline{u}_1)/(h_1+\gamma h_2)$ is the shear momentum. In this formulation we have used dimensionless variables, so that the depth at rest of the upper layer is scaled to $1$, whereas the one of the lower layer is $\delta^{-1}$, in which $\delta$ is the ratio of the depth at rest of the upper layer to the depth at rest of the lower layer (see Figure~\ref{F.sketch}). 
Similarly, $\gamma$ is the ratio of the upper layer over the lower layer densities. 
As a consequence of our scaling, the celerity of infinitesimally small and long waves is $c_0=1$. Once again, $\F_i$ ($i=1,2$) are Fourier multipliers. The choice $\F_i^{\rm id}(k)\equiv 1$ yields the Miyata-Choi-Camassa model while
 \[\F_i^{\rm imp}(k)=\sqrt{\frac{3}{\delta_i^{-1}|k|\tanh(\delta_i^{-1}|k|)}-\frac3{\delta_i^{-2} |k|^2}},\]
 with convention $\delta_1=1,\delta_2=\delta$, fits the behavior of the full bilayer Euler system around constant states, and thus gives hope for an improved precision when weak nonlinearities are involved.
Note that compared to equations (7)--(9) in~\cite{DucheneIsrawiTalhouk16} we have scaled the variables so that the shallowness parameter $\mu$ and amplitude parameter $\epsilon$ do not appear in the equations. This is for notational convenience since the parameters do not play a direct role in our results. On the other hand, we only expect the above model to be relevant for describing water waves in the regime $\mu \ll 1$ and the solutions that we construct in the end are found in the long-wave regime $\epsilon, \mu \ll 1$.

 In the following, we study solitary waves for the bilayer system~\eqref{GN}, noting that setting ${\gamma=0}$ immediately yields the corresponding result for the one-layer situation, namely system~\eqref{GN-1layer}. Our results are valid for a large class of parameters $\gamma,\delta$ and Fourier multipliers $\F_1,\F_2$, described hereafter. Our results are twofold:
 \begin{enumerate}
 \item We prove the existence of a family of solitary wave solutions for system~\eqref{GN};
 \item We provide an asymptotic description for this family in the long-wave regime.
 \end{enumerate}
These solitary waves are constructed from the Euler-Lagrange equation associated with a constrained minimization problem, as made possible by the Hamiltonian structure of system~\eqref{GN}. There are however several difficulties compared with standard works in the literature following a similar strategy (see e.g.~\cite{Angulo-Pava09} and references therein). Our functional cannot be written as the sum of the linear dispersive contribution and the nonlinear pointwise contribution: Fourier multipliers and nonlinearities are entangled. What is more, the operators involved are typically of low order ($\F$ is a smoothing operator). In order to deal with this situation, we  follow a strategy based on penalization and concentration-compactness used in a number of recent papers on the water waves problem (see e.g.~\cite{Buffoni04a, BuffoniGrovesSunWahlen13, GrovesWahlen15} and references therein) and in particular, in a recent work by one of the authors on nonlocal model equations with weak dispersion,~\cite{EhrnstromGrovesWahlen12}. Thus we show that the strategy therein may be favorably applied to bidirectional systems of equations in addition to unidirectional scalar equations such as the Whitham equation. 
Roughly speaking, the strategy is the following. The minimization problem is first solved in periodic domains using a penalization argument do deal with the fact that the energy functional is not coercive.
This allows to construct a {\em special minimizing sequence} for the real line problem by letting the period tend to infinity, which is essential to rule-out the dichotomy scenario in Lions' concentration-compactness principle. 
The long-wave description follows from precise asymptotic estimates and standard properties of the limiting (Korteweg-de Vries) model.
When the Fourier multipliers $\F_i$ have sufficiently high order, we can in fact avoid the penalization argument and consider the minimization problem on the real line directly, since any minimizing sequence is then also a special minimizing sequence. In particular, this is the case for the original Miyata-Choi-Camassa model (and of course  also the Green-Naghdi system).

Our existence proof unfortunately gives no information about stability, since our variational formulation does not involve conserved functionals; see the discussion in Section~\ref{S.S.preliminaries}. 
If sufficiently strong surface tension is included in the model, we expect that a different variational formulation could be used which also yields a conditional stability result (see~\cite{Buffoni04a, BuffoniGrovesSunWahlen13, GrovesWahlen15}).  A similar situation appears e.g.~in the study of Boussinesq systems~\cite{ChenNguyenSun10,ChenNguyenSun11}.

\subsection{The minimization problem}\label{S.S.preliminaries}

We now set up the minimization problem which allows to obtain solitary waves of system~\eqref{GN}. 
We seek traveling waves of~\eqref{GN}, namely solutions of the form (abusing notation) $\zeta(t,x)=\zeta(x-ct)$, $w(t,x)=w(x-ct)$; from which we deduce
 \[  -c\partial_x \zeta + \partial_x w \ = \ 0 \quad ;\quad -c\partial_x \big(\A_{\gamma, \delta}^\F[\zeta] w\big) +   (\gamma+\delta)\partial_x\zeta  +  \frac{1}{2} \partial_x\big(\frac{h_1^2-\gamma h_2^2}{h_1^2h_2^2} w^2\big) -\partial_x\big( \R^{\F}_{\gamma,\delta}[\zeta,w ] \big) \ = \ 0,\]
 where we denote
 \[
 \A_{\gamma, \delta}^\F[\zeta] w\eqdef \A_2^{\F}[h_2](h_2^{-1}w)+\gamma\A_1^{\F}[h_1](h_1^{-1}w), \quad \A_i^{\F}[h_i]\overline{u}_i\eqdef \overline{u}_i+\Q_i^{\F}[h_i]\overline{u}_i.
 \]
 Integrating these equations and using the assumption (since we restrict ourselves to solitary waves) $\lim_{\abs{x}\rightarrow \infty}\zeta(x)= \lim_{\abs{x}\rightarrow \infty}w(x)= 0$ yields the system of equations
\begin{equation}\label{GN-sol}
\left\{ \begin{array}{l}
-c{\zeta} \ + \  w \ =\ 0 ,  \\ \\
\displaystyle
-c \A_{\gamma, \delta}^\F[\zeta]w\ + \ (\gamma+\delta){\zeta} \ + \ \frac{1}{2} \frac{h_1^2 -\gamma h_2^2 }{(h_1 h_2)^2} w^2 \ -\  \R^{\F}_{\gamma,\delta}[\zeta,w]\ = \ 0.
\end{array}
\right.
\end{equation}

We now observe that system~\eqref{GN} enjoys a Hamiltonian structure. Indeed, define the functional
\[ \H(\zeta,w)\eqdef \frac12\int_{-\infty}^\infty (\gamma+\delta)\zeta^2+ w\A_{\gamma, \delta}^\F[\zeta]w \, \dd x.
\]
Under reasonable assumptions on $\F_1,\F_2$, and for sufficiently regular $\zeta$, $\A_{\gamma, \delta}^\F[\zeta]$ defines a well-defined, symmetric, positive definite operator~\cite{DucheneIsrawiTalhouk16}. We may thus introduce the variable
\begin{equation}\label{def-v}
v \eqdef \A_{\gamma, \delta}^\F[\zeta] w,
\end{equation}
and write
\[
 \H(\zeta,v) \eqdef \frac12\int_{-\infty}^\infty (\gamma+\delta)\zeta^2+ v(\A_{\gamma, \delta}^\F[\zeta])^{-1}v \, \dd x.
\]
It is now straightforward to check that~\eqref{GN} can be written in terms of functional derivatives of $\H$:
\begin{equation}\label{GN-Hamilton}
 \partial_t \zeta =-\partial_x \left(\dd_v \H\right)\qquad  ; \qquad  \partial_t v =-\partial_x \left(\dd_\zeta \H\right).\end{equation}

 It is therefore tempting to seek traveling waves through the time-independent quantities 
 \[\H(\zeta,v)\ = \ \frac12\int_{-\infty}^\infty (\gamma+\delta)\zeta^2+ v\A_{\gamma, \delta}^\F[\zeta]^{-1}v\, \dd x
\quad  \text{ and }\quad 
\I(\zeta,v) \ \eqdef \  \int_{-\infty}^\infty\zeta v\ \dd x.\] 
Indeed any solution to~\eqref{GN-sol} is a critical point of the functional $\H-c\I$:
 \[ \dd \H(\zeta,v)-c \dd  \I(\zeta,v)=0\qquad  \Longleftrightarrow \qquad  \dd_v\H -c\zeta =0 \quad \text{ and } \quad  \dd_\zeta \H  - c v=0 ,\]
 which, by~\eqref{GN-Hamilton}, is the desired system of equations. However, notice that from Weyl's essential spectrum theorem, one has
 \[ \text{spec}_{\text{ess}}\left(\dd^2 \H(\zeta,v)-c \dd^2  \I(\zeta,v)\right)=\text{spec}_{\text{ess}}\left(\dd^2 \H_\infty(\zeta,v)-c \dd^2  \I_\infty(\zeta,v)\right)\]
 where $\H_\infty,\I_\infty$ are the asymptotic operators as $|x|\to\infty$, \ie
 \[\dd^2 \H_\infty(\zeta,v)-c \dd^2  \I_\infty(\zeta,v)=\dd^2 \H(0,0)-c \dd^2  \I(0,0)=\begin{pmatrix} \gamma+\delta & -c \\ -c & (\gamma+\delta-\frac\gamma3(\partial_x\F_1)^2-\frac1{3\delta}(\partial_x\F_2)^2)^{-1}\end{pmatrix}.\]
 Since the spectrum of the above operator has both negative and positive components, the desired critical point is neither a minimizer nor a maximizer, as noticed (for the Green-Naghdi system) in~\cite{Li02}.

We will obtain solutions to~\eqref{GN-sol} from a {\em constrained} minimization problem depending solely on the variable $\zeta$. 
Notice that for each fixed $c$ and $\zeta$, the functional $v\mapsto \H(\zeta, v)-c \I(\zeta, v)$ has a unique critical point,
$v_{c,\zeta}=c\A_{\gamma, \delta}^\F[\zeta]\zeta$ . Substituting $v_{c, \zeta}$ into 
$\H(\zeta, v)-c \I(\zeta, v)$, we obtain
\begin{align*}
\H(\zeta, v_{c, \zeta})-c\I(\zeta, v_{c,\zeta})&=
 \frac12\int_{-\infty}^\infty (\gamma+\delta)\zeta^2-c^2 \zeta\A_{\gamma, \delta}^\F[\zeta] \zeta\, \dd x
 \\
 &=\frac{\gamma+\delta}{2} \Norm{\zeta}_{L^2}^2-\frac{c^2}2 \I(\zeta, \A_{\gamma, \delta}^\F[\zeta] \zeta).
\end{align*}
Observe now that $(\zeta, v)$ is a critical point of $\H(\zeta, v)-c\I(\zeta, v)$ if and only if 
$\zeta$ is a critical point of $\H(\zeta, v_{c, \zeta})-c\I(\zeta, v_{c,\zeta})$ and
$v=v_{c, \zeta}$.
We thus define
\begin{equation}
\label{def-E}
\begin{aligned}
\E(\zeta)&\eqdef  \I (\zeta,\A_{\gamma, \delta}^\F[\zeta]\zeta)=\int_{-\infty}^\infty \zeta\A_{\gamma, \delta}^\F[\zeta]\zeta\, \dd x\\
&=\int_{-\infty}^\infty \frac{h_1+\gamma h_2}{h_1h_2}\zeta^2+\frac\gamma3 (1-\zeta)^3 \big(\partial_x \F_1\{\frac{\zeta}{1-\zeta}\}\big)^2
+\frac13 (\delta^{-1}+\zeta)^3 \big(\partial_x \F_2\{\frac{\zeta}{\delta^{-1}+\zeta}\}\big)^2\ \dd x\\
&=\gamma\overline\E(\zeta)+\underline\E(\zeta),
\end{aligned}
\end{equation}
where 
\begin{align*}
\overline{\E}(\zeta)&=\int_{-\infty}^\infty \frac{\zeta^2}{1-\zeta}+\frac13 (1-\zeta)^3 \big(\partial_x \F_1\{\frac{\zeta}{1-\zeta}\}\big)^2\dd x,\\
\underline{\E}(\zeta)&=\int_{-\infty}^\infty\frac{\zeta^2}{\delta^{-1}+\zeta}+\frac13 (\delta^{-1}+\zeta)^3 \big(\partial_x \F_2\{\frac{\zeta}{\delta^{-1}+\zeta}\}\big)^2\dd x
\end{align*}
and look for critical points of $\H(\zeta, v_{c, \zeta})-c\I(\zeta, v_{c,\zeta})$ 
by considering the minimization problem
\begin{equation}
\label{reduced minimization problem}
\argmin \left\{  \E(\zeta), \ (\gamma+\delta)\Norm{\zeta}_{L^2}^2=q\right\},
\end{equation}
with $c^{-2}$ acting as a Lagrange multiplier.

Another way of thinking of this reduction is to solve the first equation in~\eqref{GN-sol} for $w$ and substitute the solution $w=c\zeta$ into the second equation, yielding
\begin{equation}\label{sol-c}
 c^2 \A_{\gamma,\delta}^\F[\zeta]\zeta \ = \   (\gamma+\delta)\zeta  + \frac{c^2}{2} \frac{h_1^2-\gamma h_2^2}{h_1^2h_2^2}\zeta^2 - c^2\R_{\gamma,\delta}^\F[\zeta,\zeta ].
 \end{equation}
Computing
\begin{multline}\label{def-dE}\dd \E(\zeta)=2\frac{h_1+\gamma h_2}{h_1 h_2}\zeta-\frac{h_1^2-\gamma h_2^2}{h_1^2h_2^2}\zeta^2-\frac23 \delta^{-1} h_2^{-2}\partial_x \F_2\big\{h_2^3\partial_x \F_2 \{h_2^{-1}\zeta\}\big\}-\frac{2\gamma}3 h_1^{-2}\partial_x \F_1\big\{h_1^3\partial_x \F_1 \{h_1^{-1}\zeta\}\big\} \\
\dsp \hfill +\big(h_2\partial_x \F_2\{h_2^{-1}\zeta\}\big)^2-\gamma\big(h_1\partial_x \F_1\{h_1^{-1}\zeta\}\big)^2,\end{multline}
we find that
\[\frac12 \dd \E(\zeta)=  \A_{\gamma, \delta}^\F[\zeta]\zeta \ -\ \frac{1}{2} \frac{h_1^2-\gamma h_2^2}{h_1^2h_2^2}\zeta^2 \ + \  \R_{\gamma,\delta}^\F[\zeta,\zeta ],\]
which allows us to recognize~\eqref{sol-c} as the Euler-Lagrange equation for~\eqref{reduced minimization problem}.

\subsection{Statement of the results}\label{S.statements}

For the sake of readability, we postpone to Section~\ref{S.preliminaries} the definition and (standard) notations of the functional spaces used herein. The class of Fourier multipliers for which our main result is valid is the following.
\begin{Definition}[Admissible class of Fourier multipliers]\label{D.admissible}~
\begin{enumerate}
\item \label{D.admissible 1}
$\F(k)=\F(|k|)$ and $0< \F \leq 1$;
\item \label{D.admissible 2}
$\F\in \C^2(\RR)$, $\F(0)=1$ and $\F'(0)=0$;
\item \label{D.admissible 3}
There exists an integer $j\ge 2$ such that
\[
\partial_k^j(k\F(k))\in L^2(\RR);
\]
\item \label{D.admissible 4}
There exists $\theta\in[0,1)$ and $C_\pm^{\F}>0$ such that
\[
 C^{\F}_-(1+\abs{k})^{-\theta}\leq \F(k)\leq C^{\F}_+(1+ |k|)^{-\theta}.
\]
\end{enumerate}
\end{Definition}
We also introduce a second class of strongly admissible Fourier multipliers which is used in our second result.
\begin{Definition}[Strongly admissible class of Fourier multipliers]\label{D.stradmissible}~
An admissible Fourier multipler $\F$ in the sense of Definition~\ref{D.admissible} is strongly admissible if $\F \in \C^\infty(\RR)$ and for each $j\in \NN$ there exists a constant $C_j$ such that
\[
|\partial_k^j \F(k)|\le C_j(1+|k|)^{-\theta-j}.
\]
\end{Definition}

Notice the following.
\begin{Proposition} \label{P.examples-are-subadditive}
The two aforementioned examples, namely $\F_i^{\rm id}$ and $\F_i^{\rm imp}$ are strongly admissible, and satisfy Definition~\ref{D.admissible},\ref{D.admissible 4} with (respectively) $\theta=0,1/2$.
\end{Proposition}

\begin{Assumption}[Admissible parameters]\label{D.parameters} In the following, we fix $\gamma\geq 0$, $\delta\in(0,\infty)$ such that $\delta^2-\gamma\neq0$.
We also fix a positive number $\nu$ such that $\nu\ge 1-\theta$ and $\nu>1/2$ (the second condition is automatically satisfied if $\theta <1/2$). Finally, fix $R$ an arbitrary positive constant.
\end{Assumption}

\begin{Remark}
Our results hold for any values of the parameters $(\gamma,\delta)\in[0,\infty)\times(0,\infty)$ such that $\delta^2\neq\gamma$, although admissible values for $q_0$ depend on the choice of the parameters. However, not all parameters are physically relevant in the oceanographic context. When $\gamma>1$, the upper fluid is heavier than the lower fluid, and the system suffers from strong Rayleigh-Taylor instabilities~\cite{Chandrasekhar61}. In the bilayer setting, the use of the rigid-lid assumption is well-grounded only when the density contrast, $1-\gamma$, is small. In this situation, one may use the Boussinesq approximation, that is set $\gamma=1$; see~\cite{Duchene14a} in the dispersionless setting. Notice however that system~\eqref{GN} exhibits unstable modes that are reminiscent of Kelvin-Helmholtz instabilities when the Fourier multipliers $\F_i$ satisfy Definition~\ref{D.admissible},\ref{D.admissible 4} with $\theta\in[0,1)$; see~\cite{DucheneIsrawiTalhouk16}. It is therefore noteworthy that internal solitary waves in the ocean and in laboratory experiments are remarkably stable and fit very well with the Miyata-Choi-Camassa predictions~\cite{HelfrichMelville06}. The sign of the parameter $\delta^2-\gamma$ is known to determine whether long solitary waves are of elevation or depression type, as corroborated by Theorem~\ref{T.asymptotic-result}. At the critical value $\delta^2=\gamma$, the first-order model would be the modified (cubic) {\rm KdV} equation,
predicting that no solitary wave exists~\cite{DjordjevicRedekopp78}. 
\end{Remark}

 We study the constrained minimization problem
\begin{equation}\label{min-pb}
\argmin_{\zeta\in V_{q,R}}  \E(\zeta),
\end{equation}
with
\[
V_{q,R}=\{\zeta \in H^\nu(\RR)\ :\ \Norm{\zeta}_{H^\nu(\RR)}<R,\ (\gamma+\delta)\Norm{\zeta}_{L^2(\RR)}^2=q\},
\]
and $q\in(0,q_0)$, with $q_0$ sufficiently small. Notice in particular that as soon as $q$ is sufficiently small $\Norm{\zeta}_{L^\infty}<\min(1,\delta^{-1})$  (by Lemma~\ref{L.interpolation} thereafter and since $\nu>1/2$) and $\E(\zeta)$ is well-defined  (by Lemmata~\ref{L.product} and~\ref{L.Fi-bounds} and since $\nu\geq 1-\theta$) for any $\zeta\in V_{q,R}$. Any solution will satisfy the Euler-Lagrange equation
\begin{equation}\label{sol-alpha}  \dd  \E(\zeta)+2\alpha(\gamma+\delta)\zeta=0,
\end{equation}
where $\alpha$ is a Lagrange multiplier. Equation~\eqref{sol-alpha} is exactly~\eqref{sol-c} with $(-\alpha)^{-1}=c^2$, and therefore provides a traveling-wave solution to~\eqref{GN}.

Our goal is to prove the following theorems.
\begin{Theorem}\label{T.main-result}
Let $\gamma,\delta,\nu,R$ satisfying Assumption~\ref{D.parameters} and $\F_i$, $i=1,2$ be admissible in the sense of Definition~\ref{D.admissible}. Let $D_{q,R}$ be the set of minimizers of $\E$ over $V_{q,R}$. Then there exists $q_0>0$ such that for all $q\in (0,q_0)$, the following statements hold:
\begin{itemize}
\item The set $D_{q,R}$ is nonempty and each element in $D_{q,R}$ solves the traveling wave equation~\eqref{sol-c}, with $c^2=(-\alpha)^{-1}>1 $. Thus for any $\zeta\in D_{q,R}$, $\big(\zeta(x\pm ct),w_\pm=\pm c \zeta(x\pm ct)\big)$ is a supercritical solitary wave solution to~\eqref{GN}.
\item For any  minimizing sequence $\{\zeta_n\}_{n\in\mathbb{N}}$ for $\E$ in $V_{q,R}$ such that $\sup_{n\in\NN} \Norm{\zeta_n}_{H^\nu(\RR)}<R$, there exists a sequence $\{x_n\}_{n\in\mathbb{N}}$ of real numbers such that a subsequence of $\{\zeta_n(\cdot+x_n)\}_{n\in\mathbb{N}}$ converges (strongly in $H^\nu(\RR)$ if $\nu=1-\theta>1/2$; weakly in $H^\nu(\RR)$ and strongly in $H^s(\RR)$ for $s\in[0,\nu)$ otherwise) to an element in $D_{q,R}$.

\item There exist two constants $m,M>0$ 
 such that 
\[ \Norm{\zeta}_{H^\nu(\RR)}^2\leq  Mq \quad \text{ and } \quad c^{-2}=-\alpha\leq
 1-m q^{\frac23},\]
uniformly over $q\in(0,q_0)$ and $\zeta\in D_{q,R}$.
\end{itemize}
\end{Theorem} 

\begin{Theorem}\label{T.asymptotic-result}
In addition to the hypotheses of Theorem~\ref{T.main-result}, assume that $\F_i$, $i=1,2$, are strongly admissible in the sense of Definition~\ref{D.stradmissible}. Then there exists $q_0>0$ such that for any $q\in(0,q_0)$,
each $\zeta \in D_{q,R}$ belongs to $H^s(\RR)$ for any $s\ge 0$ and 
\[
\sup_{\zeta\in D_{q,R}} \inf_{x_0\in\RR} \Norm{q^{-\frac23}\zeta(q^{-1/3}\cdot)-\xi_{\rm KdV}(\cdot-x_0)}_{H^1(\RR)}=\O(q^{\frac16}) 
\]
where 
\[\xi_{\rm KdV}(x)=\frac{\alpha_0(\gamma+\delta)}{\delta^2-\gamma}\sech^2\left(\frac{1}{2}\sqrt{\frac{3\alpha_0(\gamma+\delta)}{\gamma+\delta^{-1}}}x\right)\]
is the unique (up to translation) solution of the {\rm KdV} equation~\eqref{KdV-eq} and 
\[
\alpha_0=\frac{3}{4}\left(\frac{(\delta^2-\gamma)^4}{(\gamma+\delta)^4(\gamma+\delta^{-1})}\right)^\frac{1}{3}.
\]
In addition, the number $\alpha$, defined in Theorem~\ref{T.main-result}, satisfies
\[
\alpha+1=q^\frac{2}{3}\alpha_0+\mathcal{O}(q^\frac{5}{6}),
\]
uniformly over $q\in(0,q_0)$ and $\zeta\in D_{q,R}$.
\end{Theorem} 

\section{Technical results}\label{S.preliminaries}

In the following, we denote $C(\lambda_1,\lambda_2,\dots)$ a positive constant depending non-decreasingly on the parameters  $\lambda_1,\lambda_2,\dots$. We write $A\lesssim B$ when $A\leq CB $ with $C$ a nonnegative constant whose value is of no importance. 
We do not display the dependence with respect to the parameters $\gamma,\delta,C^{\F_i}_\pm$ and regularity indexes.

\paragraph{Functional setting on the real line} Here and thereafter, we denote $L^2(\RR)$ the standard Lebesgue space of square-integrable functions, endowed with the norm $\Norm{f}_{L^2}=(\int_{-\infty}^\infty \abs{f(x)}^2\ \dd x)^{1/2}$. 
The real inner product of $f_1,f_2\in L^2(\RR)$ is denoted by
$
\langle f_1,f_2\rangle=\int_{\RR}f_1(x)f_2(x) \dd x
 $.
We use the same notation for duality pairings which are clear from the context. The space $L^\infty(\RR)$ consists of all essentially bounded, Lebesgue-measurable functions
 $f$, endowed with the norm
$
\Norm{f}_{L^\infty}= {\rm ess\,sup}_{x\in\RR} \abs{ f(x)}
$.
For any real constant $s\in\RR$, $H^s(\RR)$ denotes the Sobolev space of all tempered distributions $f$ with finite norm $\Norm{ f}_{H^s}=\Norm{ \Lambda^s f}_{L^2} < \infty$, where $\Lambda$ is the pseudo-differential operator $\Lambda=(1-\partial_x^2)^{\frac12}$. For $n\in\NN$, $\C^n(\RR)$ is the space of functions having continuous derivatives up to order $n$, and $\C^\infty(\RR)=\bigcap_{n\in\NN} \C^n(\RR)$.  The Schwartz space is denoted $\mathcal S(\RR)$ and the tempered distributions $\mathcal{S}'(\RR)$. We use the following convention for the Fourier transform:
\[\mathcal F\big(f\big)(k)=\hat f(k)\eqdef \frac{1}{\sqrt{2\pi}}\int_{\RR} f(x)e^{-\mathrm i x k}\ \dd x.\]

We start with standard estimates in Sobolev spaces. 
The following interpolation estimates are standard and used without reference in our proofs.
\begin{Lemma}[Interpolation estimates]\label{L.interpolation}
Let $f\in H^\mu(\RR)$, with $\mu>1/2$.
\begin{enumerate}
\item \label{L.interpolation 1}
One has $f\in L^\infty(\RR)$ and
\[ \Norm{f}_{L^\infty}\lesssim  \Norm{f}_{L^2}^{1-\frac1{2\mu}}\Norm{f}_{H^\mu}^{\frac1{2\mu}}.\]
\item \label{L.interpolation 2}
For any $\delta\in (0,\mu)$, one has $f\in H^{\mu-\delta}(\RR)$ and
\[\Norm{f}_{H^{\mu-\delta}}\leq \Norm{f}_{L^2}^{\frac\delta\mu}\Norm{f}_{H^\mu}^{1-\frac\delta\mu}.\]
\end{enumerate}
\end{Lemma}
The following Lemma is given for instance in~\cite[Theorem~C.12]{Benzoni-GavageSerre07}.
\begin{Lemma}[Composition estimate]\label{L.composition}  
Let $G$ be a smooth function vanishing at $0$, and  $f\in H^\mu(\RR)$ with $\mu>1/2$. Then  $G\circ f\in H^\mu(\RR)$ and we have
\[
\Norm{G\circ f}_{H^\mu} \le C(\Norm{f}_{L^\infty}) \Norm{f}_{H^\mu}.
\]
\end{Lemma}

\begin{Lemma}[Product estimates]\label{L.product}~
\begin{enumerate}
\item \label{L.product 1}
For any $f,g\in L^\infty(\RR)\cap H^s(\RR)$ with $s\geq 0$, one has $fg\in H^s(\RR)$ and
\[ \Norm{fg}_{H^s}\lesssim  \Norm{f}_{H^s}\Norm{ g}_{L^\infty}+\Norm{g}_{H^s}\Norm{ f}_{L^\infty}.\]
\item \label{L.product 2}
For any $f\in H^s(\RR),g\in H^t(\RR)$ with $s+t\ge 0$, and let $r$ such that $\min(s,t)\geq r$ and $r<s+t-1/2$. Then one has $fg\in H^r(\RR)$ and
\[ \Norm{fg}_{H^r}\lesssim  \Norm{f}_{H^s}\Norm{ g}_{H^t}.\]
\item \label{L.product 3}
 For any $\zeta\in L^\infty(\RR)$ such that $\Norm{\zeta}_{L^\infty}\leq 1-h_0$ with $h_0>0$ and any $f\in L^\infty(\RR)$, one has
\[  \Norm{\frac{f}{1+\zeta}}_{L^\infty}\leq C(h_0^{-1})\Norm{f}_{L^\infty}.\]
\item \label{L.product 4}
For any $\zeta\in H^\mu(\RR)$ with $\mu>1/2$ such that $\Norm{\zeta}_{L^\infty}\leq 1-h_0$ with $h_0>0$ and any $f\in H^s(\RR)$ with $s\in [-\mu,\mu]$, one has
\[  \Norm{\frac{f}{1+\zeta}}_{H^s}\leq C(h_0^{-1},\Norm{\zeta}_{H^\mu})\Norm{f}_{H^s}.\]
\end{enumerate}
\end{Lemma}
\begin{proof}The first two items are standard (see for instance~\cite[Prop. C.11 and Th. C.10]{Benzoni-GavageSerre07}. The third item is obvious.
For the last item, we use the second item and deduce that for any $f\in H^s(\RR)$, $s\in[-\mu,\mu]$, and $g\in H^\mu(\RR)$,
\[ \Norm{fg}_{H^s}\lesssim \Norm{g}_{H^\mu}\Norm{f}_{H^s}.\]
Hence
\[
\Norm{\frac{f}{1+\zeta}}_{H^s}\leq \Norm{f}_{H^s}+\Norm{\frac{f\zeta}{1+\zeta}}_{H^s}
\leq \Norm{f}_{H^s}+\Norm{\frac{\zeta}{1+\zeta}}_{H^\mu}\Norm{f}_{H^s}
\]
We conclude by
\[ \Norm{\frac{\zeta}{1+\zeta}}_{H^\mu}\lesssim C(h_0^{-1})\Norm{\zeta}_{H^\mu}, \]
where we have used Lemma~\ref{L.composition}, 
and the estimate is proved.
\end{proof}

The following Lemma justifies the assumptions of admissible Fourier multipliers in Definition~\ref{D.admissible}.
\begin{Lemma}[Properties of admissible Fourier multipliers]\label{L.Fi-bounds} Any admissible Fourier multipler (in the sense of Definition~\ref{D.admissible}), $\F_i$, satisfies the following.
\begin{enumerate}
\item\label{L.Fi-bounds 1}
 The linear operator $\partial_x\F_i(D)$ is bounded from $H^s(\RR)$ to $H^{s-1+\theta}(\RR)$, for any $s\in\RR$, and 
\[\Norm{\partial_x\F_i}_{H^s\to H^{s-1+\theta}}\lesssim C^{\F_i}_+.\]
Moreover, for any $\zeta\in H^{s+1-\theta}$, one has
 \[\Norm{\zeta}_{H^s}^2 + (C^{\F_i}_+)^{-2} \Norm{\partial_x\F_i\{\zeta\}}_{H^s}^2 \lesssim \Norm{\zeta}_{H^{s+1-\theta}}^2\lesssim \Norm{\zeta}_{H^s}^2 + (C^{\F_i}_-)^{-2} \Norm{\partial_x\F_i\{\zeta\}}_{H^s}^2.\]
\item\label{L.Fi-bounds 2}
Let $\varphi \in \C^\infty(\RR)$ with compact support and $[\partial_x \F_i, \varphi] \zeta= \partial_x \F_i\{\varphi\zeta\}-\varphi \partial_x \F_i\{\zeta\}$. Then
\[
\Norm{[\partial_x \F_i, \varphi] \zeta}_{L^2} \lesssim \Norm{\widehat{\varphi'}}_{L^1} \Norm{\zeta}_{H^{1-\theta}}.
\]
\item\label{L.Fi-bounds 3}
There exists $j\geq 2$ and $C_j$ such that for any $\zeta\in L^2(\RR)$ with compact support
\[ \abs{\partial_x\F_i\{\zeta\}}(x)\leq \frac{C_j}{\dist(x,\supp(\zeta))^j}\Norm{\zeta}_{L^2},\quad \text{for a.a. } x\in\RR\setminus\supp(\zeta).\]
\end{enumerate}
\end{Lemma}

\begin{proof}
The first result is obvious from Definition~\ref{D.admissible},\ref{D.admissible 1} and the definition of Sobolev spaces. For the second, we shall first prove that the function ${\mathsf G}_i:k\mapsto k\F_i(k)$ satisfies
\begin{equation}\label{est-kF}
\abs{{\mathsf G}_i'(k)}\lesssim \langle k\rangle^{1-\theta}.
\end{equation}
To this aim, let us first consider ${\mathsf G}\in\mathcal{S}(\RR)$ and $\chi$ a smooth cut-off function, such that $\chi(k)=1$ for $|k|\leq 1/2$ and $\chi(k)=0$ for $|k|\geq 1$. We decompose
\[\abs{{\mathsf G}'}(k)\leq \abs{\chi(D){\mathsf G}'}(k)+ \abs{(1-\chi(D)){\mathsf G}'}(k).\]
For the first contribution, one has
\[ \abs{\chi(D){\mathsf G}'}(k)=\frac{1}{\sqrt{2\pi}}\left|\int_{\RR} \hat\chi(\xi){\mathsf G}'(k+\xi)\dd \xi\right|=\frac{1}{\sqrt{2\pi}} \left|\int_{\RR} (\hat\chi)'(\xi){\mathsf G}(k+\xi)\dd \xi \right|\lesssim \sup_{\xi\in\RR}\frac{\abs{{\mathsf G}(k+\xi)}}{\langle k+\xi\rangle^{1-\theta}}\langle k\rangle^{1-\theta} \Norm{ \langle \cdot\rangle^{1-\theta}\hat\chi'}_{L^1},\]
and the second contribution satisfies for any $j\geq 2$,
\[ \abs{(1-\chi(D)){\mathsf G}'}(k)\lesssim\Norm{(1-\chi(\xi))|\xi|\hat {\mathsf G}(\xi)}_{L^1}\lesssim \Norm{\langle \xi\rangle^{-(j-1)}|\xi|^j\hat {\mathsf G}(\xi)}_{L^1}\lesssim \Norm{{\mathsf G}^{(j)}}_{L^2},\]
by the Cauchy-Schwarz inequality and Parseval's theorem. Thus we find, for any $j\geq 2$,
\[\abs{{\mathsf G}'}(k)\lesssim \Norm{\langle\cdot\rangle^{\theta-1}{\mathsf G}}_{L^\infty} \langle k\rangle^{1-\theta} +\Norm{{\mathsf G}^{(j)}}_{L^2}.\]
The same estimate applies to ${\mathsf G}(k)= k\F_i(k) $ by smooth approximation, and~\eqref{est-kF} follows from Definition~\ref{D.admissible}.
Now, we
note that the Fourier transform of $[\partial_x \F_i, \varphi] \zeta$ is given by
\[
\frac1{\sqrt{2\pi}}
\int_{\RR} (\mathrm ik\F_i(k)-\mathrm is\F_i(s))\hat \varphi(k-s)\hat \zeta(s)\dd s.
\]
By~\eqref{est-kF} and  the mean value theorem, we find that $|k\F_i(k)-s\F_i(s)|\lesssim (1+|s|)^{1-\theta} |k-s|$.
Due to  Young's inequality and Parseval's theorem, we find
\[
\Norm{[\partial_x \F_i, \varphi] \zeta}_{L^2}  \lesssim \Norm{\widehat{\varphi'}}_{L^1} \Norm{(1+|\cdot|)^{1-\theta}\hat \zeta}_{L^2} \lesssim \Norm{\widehat{\varphi'}}_{L^1} \Norm{\zeta}_{H^{1-\theta}}.
\]
For the third result, let us assume at first that the kernel $K_i\eqdef\mathcal{F}^{-1}(\mathrm ik\F_i(k))\in L^2(\RR)$. Then one has
\begin{align*}
\abs{\partial_x\F_i\{\zeta\}}(x)&=\frac{1}{\sqrt{2\pi}}\left|\int_{\RR}K_i(x-y)\zeta(y)\ \dd y\right|=\frac{1}{\sqrt{2\pi}}\left|\int_{\text{supp}(\zeta)}\frac{(x-y)^jK_i(x-y)\zeta(y)}{(x-y)^j}\ \dd y\right|\\
&\leq \frac{\big(\abs{K_{i,j}}\ast\abs{\zeta}\big)(x)}{\sqrt{2\pi}\dist(x,\text{supp}(\zeta))^j}\lesssim \frac{\Norm{\zeta}_{L^2}}{\dist(x,\text{supp}(\zeta))^j},
\end{align*}
where we denote $K_{i,j}(x)=x^jK_i(x) $, remark that $K_{i,j}\in L^2(\RR)$ by Definition~\ref{D.admissible},\ref{D.admissible 3} and Plancherel's theorem, and apply the Cauchy-Schwarz inequality to the convolution. If $K_i\notin L^2(\RR)$, we obtain the result by regularizing $K_i$ (\ie smoothly truncating $\F_i$) and passing to the limit.
\end{proof}

\begin{Lemma}\label{L.EvsY}
Let $\gamma\geq 0$, $\delta>0$, $\mu>1/2$ and $\F_i$ be admissible Fourier multipliers. Assume that $\zeta\in H^\mu(\RR)$ is such that $1-\Norm{\zeta}_{L^\infty}\geq h_0$, $\delta^{-1}-\Norm{\zeta}_{L^\infty}\geq h_0$, with $h_0>0$. Then there exist a constant $C_0=C(h_0^{-1},\Norm{\zeta}_{H^\mu})$ such that
\[C_0^{-1}\Norm{\zeta}_{H^{1-\theta}}^2\leq \E(\zeta)\leq C_0\Norm{\zeta}_{H^{1-\theta}}^2 .\]
\end{Lemma}
\begin{proof}
We first deal with the contribution of $\overline{\E}(\zeta)$ defined in~\eqref{def-E}. By Lemma~\ref{L.Fi-bounds},\ref{L.Fi-bounds 1}
we get that
\[
\overline{\E}(\zeta)\leq C(\Norm{\zeta}_{L^\infty})\Norm{\frac{\zeta}{1-\zeta}}_{H^{1-\theta}}^2 \quad \text{ and } \quad  \Norm{\frac{\zeta}{1-\zeta}}_{H^{1-\theta}}^2 \leq C(h_0^{-1})\overline{\E}(\zeta).
\]
By Lemma~\ref{L.product},\ref{L.product 4},  one has
\[\Norm{\frac{\zeta}{1-\zeta}}_{H^{1-\theta}}\leq C(h_0^{-1},\Norm{\zeta}_{H^\mu})\Norm{\zeta}_{H^{1-\theta}},\]
and the triangle inequality together with Lemma~\ref{L.product},\ref{L.product 2}  yields
\[
\Norm{\zeta}_{H^{1-\theta}}\lesssim \Norm{\frac{\zeta}{1-\zeta}}_{H^{1-\theta}}+\Norm{\frac{\zeta^2}{1-\zeta}}_{H^{1-\theta}}\\
\lesssim  \Norm{\frac{\zeta}{1-\zeta}}_{H^{1-\theta}}+\Norm{\zeta}_{H^\mu}\Norm{\frac{\zeta}{1-\zeta}}_{H^{1-\theta}}.\]
Collecting the above information, we find that
\[
C_0^{-1}\Norm{\zeta}_{H^{1-\theta}}^2\leq \overline{\E}(\zeta)\leq C_0\Norm{\zeta}_{H^{1-\theta}}^2,
\]
with $C_0=C(h_0^{-1},\Norm{\zeta}_{H^\mu})$. Similar estimates hold for $\underline{\E}(\zeta)$, and thus for
 $\E(\zeta)=\gamma\overline{\E}(\zeta)+\underline{\E}(\zeta)$.
\end{proof}

\begin{Lemma}\label{L.diff-E}
Let $\gamma\geq 0$, $\delta>0$, $\mu>1/2$ and $\F_i$ be admissible Fourier multipliers. Assume that, for $j\in\{1,2\}$, $\zeta_j\in H^\mu(\RR)$ is such that $1-\Norm{\zeta_j}_{L^\infty}\geq h_0$ and $\delta^{-1}-\Norm{\zeta_j}_{L^\infty}\geq h_0$, with $h_0>0$. Then one has
\[\E(\zeta_1)-\E(\zeta_2)\leq C(h_0^{-1},\Norm{\zeta_1}_{H^\mu},\Norm{\zeta_2}_{H^\mu}) \Norm{\zeta_1-\zeta_2}_{H^{\mu}}.\]
\end{Lemma}

\begin{proof}
As previously, we detail the result for $\overline{\E}(\zeta)$, as the similar estimate for $\underline{\E}(\zeta)$ is obtained in the same way. One has 
\begin{multline*}
\overline\E(\zeta_1)-\overline\E(\zeta_2)=\int_\RR \frac{\zeta_1^2}{1-\zeta_1}-\frac{\zeta_2^2}{1-\zeta_2}+\frac13 \big[(1-\zeta_1)^3-(1-\zeta_2)^3\big] \big(\partial_x \F_1\{\frac{\zeta_1}{1-\zeta_1}\}\big)^2\\
+\frac13(1-\zeta_2)^3\Big[ \big(\partial_x \F_1\{\frac{\zeta_1}{1-\zeta_1}\}\big)^2-\big(\partial_x \F_1\{\frac{\zeta_2}{1-\zeta_2}\}\big)^2\Big]
\dd x,
\end{multline*}
By Lemma~\ref{L.product},~\ref{L.product 3}, and the Cauchy-Schwarz inequality, we immediately have
\[\int_\RR \left|\frac{\zeta_1^2}{1-\zeta_1}-\frac{\zeta_2^2}{1-\zeta_2}\right|\dd x\leq  C(h_0^{-1},\Norm{\zeta_1}_{L^\infty},\Norm{\zeta_2}_{L^\infty})(\Norm{\zeta_1}_{L^2}+\Norm{\zeta_2}_{L^2})\Norm{\zeta_1-\zeta_2}_{L^2}.\]
Similarly, we find by Lemma~\ref{L.Fi-bounds},\ref{L.Fi-bounds 1} 
\[\int_\RR \left|\big[(1-\zeta_1)^3-(1-\zeta_2)^3\big] \big(\partial_x \F_1\{\frac{\zeta_1}{1-\zeta_1}\}\big)^2\right|\dd x\leq  C(\Norm{\zeta_1}_{L^\infty},\Norm{\zeta_2}_{L^\infty})\Norm{\zeta_1-\zeta_2}_{L^\infty}\Norm{\frac{\zeta_1}{1-\zeta_1}}_{H^{1-\theta}}^2,\]
and by Lemma~\ref{L.product},\ref{L.product 4},
\[\Norm{\frac{\zeta_1}{1-\zeta_1}}_{H^{1-\theta}}^2\leq   C(h_0^{-1},\Norm{\zeta_1}_{H^\mu}).\]
Finally, 
\begin{multline*}
\int_\RR \left| (1-\zeta_2)^3\Big[\big(\partial_x \F_1\{\frac{\zeta_1}{1-\zeta_1}\}\big)^2-\big(\partial_x \F_1\{\frac{\zeta_2}{1-\zeta_2}\}\big)^2\Big]\right|\dd x\\
\leq C(\Norm{\zeta_2}_{L^\infty})\Norm{\frac{\zeta_1}{1-\zeta_1}-\frac{\zeta_2}{1-\zeta_2}}_{H^{1-\theta}}\Norm{\frac{\zeta_1}{1-\zeta_1}+\frac{\zeta_2}{1-\zeta_2}}_{H^{1-\theta}},\end{multline*}
and we conclude by Lemma~\ref{L.product},~\ref{L.product 4}
\[\Norm{\frac{\zeta_1}{1-\zeta_1}-\frac{\zeta_2}{1-\zeta_2}}_{H^{1-\theta}}\leq C(h_0^{-1},\Norm{\zeta_1}_{H^\mu},\Norm{\zeta_2}_{H^\mu})\Norm{\zeta_1-\zeta_2}_{H^{1-\theta}},\]
and
\[\Norm{\frac{\zeta_1}{1-\zeta_1}+\frac{\zeta_2}{1-\zeta_2}}_{H^{1-\theta}}\leq C(h_0^{-1},\Norm{\zeta_1}_{H^\mu},\Norm{\zeta_2}_{H^\mu}).\]
The result is proved.
\end{proof}

 \begin{Lemma}\label{L.E-decomp-KdV}
Let $\gamma\geq 0$, $\delta>0$, and $\F_i$ be admissible Fourier multipliers. Let $l\in\{1,2,3\}$ and $\zeta\in H^l(\RR)$ such that $1-\Norm{ \zeta}_{L^\infty}\geq h_0$ and $\delta^{-1}-\Norm{\zeta}_{L^\infty}\geq h_0$, with $h_0>0$. Then one can decompose
\[
\E(\zeta)=\int_\RR (\gamma+\delta)\zeta^2+(\gamma-\delta^2)\zeta^3+\frac{\gamma+\delta^{-1}}{3}(\partial_x\zeta)^2\ \dd x+\E_{{\rm rem}}(\zeta),
\]
and
\[
 \langle \dd\E(\zeta),\zeta\rangle =\int_\RR  2(\gamma+\delta)\zeta^2+3(\gamma-\delta^2)\zeta^3+2\frac{\gamma+\delta^{-1}}{3}(\partial_x\zeta)^2\ \dd x +\langle \dd\E_{\rm rem}(\zeta),\zeta\rangle,
\]
 where
\[
\abs{\E_{{\rm rem}}}+\abs{\langle \dd\E_{{\rm rem}}(\zeta),\zeta\rangle}\leq C(h_0^{-1},\Norm{\zeta}_{H^1})\big( \Norm{\zeta}_{L^\infty}^2\Norm{\zeta}_{L^2}^2+\Norm{\zeta}_{L^\infty}\Norm{\partial_x\zeta}_{L^2}^2+\Norm{\partial_x^l\zeta}_{L^2}\Norm{\partial_x\zeta}_{L^2}\big).
\]
 \end{Lemma}
 \begin{proof}
We consider $\overline{\E}(\zeta)$; the corresponding expansion for $\underline{\E}(\zeta)$ is obtained similarly.
We write
\[\overline{\E}(\zeta)=\int_\RR \zeta^2+\zeta^3+\frac{1}{3}(\partial_x\zeta)^2\ \dd x+\overline\E_{{\rm rem}}(\zeta),
\]
where
 \begin{multline*} \overline\E_{{\rm rem}}(\zeta)= \int_\RR\frac{\zeta^4}{1-\zeta} \, \dd x+\frac13 \int_\RR(1-\zeta)^3 \big(\partial_x\{\frac{\zeta}{1-\zeta}\}\big)^2-(\partial_x \zeta)^2 \, \dd x\\
 +\int_\RR (1-\zeta)^3 \left[ \big(\partial_x \F_1\{\frac{\zeta}{1-\zeta}\}\big)^2- \big(\partial_x\{\frac{\zeta}{1-\zeta}\}\big)^2\right]\, \dd x.
\end{multline*}
Note that by the Cauchy-Schwarz inequality
\[
\left|\int_\RR\frac{\zeta^4}{1-\zeta} \, \dd x\right|\leq \frac{\Norm{\zeta}_{L^\infty}^2\Norm{\zeta}_{L^2}^2}{h_0}
\]
and
\[
\left|\int_\RR(1-\zeta)^3\left(\partial_x\left\{\frac{\zeta}{1-\zeta}\right\}\right)^2-(\partial_x\zeta)^2\, \dd x\right|= \left| \int_\RR \frac{\zeta(\partial_x\zeta)^2}{1-\zeta}\, \dd x\right|
\leq \frac{\Norm{\zeta}_{L^\infty}\Norm{\partial_x\zeta}_{L^2}^2}{h_0}.
\]
Moreover
 \begin{multline*}
\left|\int_\RR(1-\zeta)^3\left[\left(\partial_x\F_1\left\{\frac{\zeta}{1-\zeta}\right\}\right)^2-\left(\partial_x\left\{\frac{\zeta}{1-\zeta}\right\}\right)^2\right]\, \dd x\right| \\
\leq\int_\RR (1+\abs{\zeta})^3 \left|(\partial_x\F_1-\partial_x)\left(\frac{\zeta}{1-\zeta}\right)\right|\,\left|(\partial_x\F_1+\partial_x)\left(\frac{\zeta}{1-\zeta}\right)\right|\, \dd x.
\end{multline*}
Applying the Cauchy-Schwarz inequality, Plancherel's theorem and the estimates
\[ |\F_1(k)-1|\lesssim |k|^{l-1}, \qquad  |\F_1(k)+1|\lesssim 1,\]
(by Definition~\ref{D.admissible},\ref{D.admissible 1} and~\ref{D.admissible 2}), we deduce
\begin{align*}
&\left|\int_\RR(1-\zeta)^3\left[\left(\partial_x\F_1\left\{\frac{\zeta}{1-\zeta}\right\}\right)^2-\left(\partial_x\left\{\frac{\zeta}{1-\zeta}\right\}\right)^2\right] \, \dd x\right| \\
&  \leq (1+\Norm{\zeta}_{L^\infty})^3 \Norm{\partial_x^l \left(\frac{\zeta}{1-\zeta}\right)}_{L^2} \Norm{\partial_x \left(\frac{\zeta}{1-\zeta}\right)}_{L^2}\\
& \leq C(\Norm{\zeta}_{H^\mu})\Norm{\partial_x^l\zeta}_{L^2}\Norm{\partial_x\zeta}_{L^2},
\end{align*}
where the last inequality follows from Leibniz's rule and standard bilinear estimates~\cite[Prop.~C.12]{Benzoni-GavageSerre07}.
Combining the above estimates together with similar calculations for $\underline\E$ yields the desired estimate for $\abs{\E_{{\rm rem}}}$. The estimate for $\abs{\langle \dd\E_{{\rm rem}}(\zeta),\zeta\rangle}$ follows in the same way when decomposing
\begin{multline}\label{def-dEzeta}\langle\dd \E(\zeta),\zeta\rangle=\int_\RR 2\frac{h_1+\gamma h_2}{h_1 h_2}\zeta^2-\frac{h_1^2-\gamma h_2^2}{h_1^2h_2^2}\zeta^3+\frac23 \delta^{-1} h_2^3\big(\partial_x \F_2 \{h_2^{-1}\zeta\}\big)\big(\partial_x \F_2 \{h_2^{-2}\zeta\}\big)\\
\dsp \hfill +\frac{2\gamma}3 h_1^3\big(\partial_x \F_1 \{h_1^{-1}\zeta\}\big)\big(\partial_x \F_1 \{h_1^{-2}\zeta\}\big) +\zeta\big(h_2\partial_x \F_2\{h_2^{-1}\zeta\}\big)^2-\gamma\zeta\big(h_1\partial_x \F_1\{h_1^{-1}\zeta\}\big)^2\, \dd x,\end{multline}
and we do not detail for the sake of conciseness.
\end{proof}

\begin{Lemma}\label{L.E-decomp-order}
Let $\gamma\geq 0$, $\delta>0$, $\mu>1/2$ and $\F_i$ be admissible Fourier multipliers such that $\mu\geq 1-\theta$. Let $\zeta\in H^\mu(\RR)$ such that $1-\Norm{\zeta}_{L^\infty}\geq h_0$ and $\delta^{-1}-\Norm{\zeta}_{L^\infty}\geq h_0$, with $h_0>0$. Then one can decompose
\[
\E(\zeta)=\E_{2}(\zeta)+\E_{3}(\zeta)+\E_{{\rm rem}}^{(1)}(\zeta)
\]
and
\[\langle\dd \E(\zeta),\zeta\rangle=2\E_{2}(\zeta)+3\E_{3}(\zeta)+\E_{{\rm rem}}^{(2)}(\zeta),\]
where
\begin{align*}
\E_{2}(\zeta)&=\int_\RR (\gamma+\delta)\zeta^2 +\gamma \frac13(\partial_x\F_1\{\zeta\})^2+\delta^{-1}\frac13(\partial_x\F_2\{\zeta\})^2\ \dd x,\\ \E_{3}(\zeta)&=\int_\RR (\gamma-\delta^2)\zeta^3- \gamma \zeta (\partial_x\F_1\{\zeta\})^2+\zeta(\partial_x\F_2\{\zeta\})^2 +\gamma \frac{2}{3}(\partial_x\F_1\{\zeta\})(\partial_x\F_1\{\zeta^2\})-\frac{2}{3}(\partial_x\F_2\{\zeta\})(\partial_x\F_2\{\zeta^2\})\ \dd x.
\end{align*}
Moreover, one has $\E_2(\zeta)\geq (\gamma+\delta)\Norm{\zeta}_{L^2}^2$  and
\begin{align*}
 \abs{\E_{3}(\zeta)}&\leq C(h_0^{-1},\Norm{\zeta}_{H^\mu})\Norm{\zeta}_{L^\infty}\Norm{\zeta}_{H^{1-\theta}}^2,\\
\forall j\in\{1,2\}, \quad \abs{\E_{{\rm rem}}^{(j)}(\zeta)}&\leq C(h_0^{-1},\Norm{\zeta}_{H^\mu}) \Norm{\zeta}_{L^\infty}^2\Norm{\zeta}_{H^{1-\theta}}^2,
\end{align*}
\end{Lemma}
\begin{proof}
The estimate on $\E_3$ is straightforward by the Cauchy-Schwarz inequality and applying Lemma~\ref{L.Fi-bounds},\ref{L.Fi-bounds 1} and Lemma~\ref{L.product},\ref{L.product 1}. We detail the estimate for $\abs{\E_{{\rm rem}}^{(j)}(\zeta)}$.
As above, we focus on the terms involving $\overline\E$, the terms involving $\underline\E$ being obtained identically. One has
\begin{multline*}
\overline\E_{\rm rem}^{(1)}(\zeta)
=\int_\RR\frac{\zeta^4}{1-\zeta}+\big(\zeta^2-\frac{\zeta^3}{3}\big)\left(\partial_x\F_1\left\{\frac{\zeta}{1-\zeta}\right\}\right)^2\,\dd x\\
+\frac{1}{3}\int_\RR(1-3\zeta)\left(\partial_x\F_1\left\{\frac{\zeta}{1-\zeta}\right\}\right)^2
 -(\partial_x\F_1\{\zeta\})^2-2\partial_x\F_1\{\zeta\}\partial_x\F_1\{\zeta^2\}+3\zeta (\partial_x\F_1\{\zeta\})^2\,\dd x.
\end{multline*}
We estimate each bracket separately. First note that, by the Cauchy-Schwarz inequality, Lemma~\ref{L.Fi-bounds},\ref{L.Fi-bounds 1} and Lemma~\ref{L.product},\ref{L.product 3},\ref{L.product 4} one has
\begin{align*}
\int_\RR\left|\frac{\zeta^4}{1-\zeta}\right| \, \dd x&\leq \frac{\Norm{\zeta}_{L^\infty}^2\Norm{\zeta}_{L^2}^2}{h_0},\\
\int_\RR\left|\big(\zeta^2-\frac{\zeta^3}{3}\big)\left(\partial_x\F_1\left\{\frac{\zeta}{1-\zeta}\right\}\right)^2\right| \, \dd x &\leq C(h_0^{-1},\Norm{\zeta}_{H^\mu})\Norm{\zeta}_{L^\infty}^2\Norm{\zeta}_{H^{1-\theta}}^2.
\end{align*}
Next consider
\begin{align*}
&\frac{1}{3}\int_\RR(1-3\zeta)\left(\partial_x\F_1\left\{\frac{\zeta}{1-\zeta}\right\}\right)^2-(\partial_x\F_1\{\zeta\})^2-2\partial_x\F_1\{\zeta\}\partial_x\F_1\{\zeta^2\}+3\zeta (\partial_x\F_1\{\zeta\})^2 \, \dd x\\
&=\frac{1}{3}\int_\RR(1-3\zeta)\left(\partial_x\F_1\left\{\frac{\zeta^2}{1-\zeta}\right\}\right)^2+2(\partial_x\F_1\{\zeta\})(\partial_x\F_1\left\{\frac{\zeta^3}{1-\zeta}\right\}) -6\zeta(\partial_x\F_1\{\zeta\})\left(\partial_x\F_1\left\{\frac{\zeta^2}{1-\zeta}\right\}\right) \, \dd x.
\end{align*}
where we used the identity $\partial_x\F_1\left\{\frac{\zeta}{1-\zeta}\right\}=\partial_x\F_1\{\zeta\}+\partial_x\F_1\left\{\frac{\zeta^2}{1-\zeta}\right\}$.
It follows from the Cauchy-Schwarz inequality, Lemma~\ref{L.Fi-bounds},\ref{L.Fi-bounds 1} and Lemma~\ref{L.product},\ref{L.product 1},\ref{L.product 4} that
\begin{multline*}
\left|\frac{1}{3}\int_\RR\left(\partial_x\F_1\left\{\frac{\zeta}{1-\zeta}\right\}\right)^2-(\partial_x\F_1\{\zeta\})^2-2\partial_x\F_1\{\zeta\}\partial_x\F_1\{\zeta^2\} \, \dd x \right|\\
\lesssim \Norm{\frac{\zeta^2}{1-\zeta}}_{H^{1-\theta}}^2+\Norm{\zeta}_{H^{1-\theta}}\Norm{\frac{\zeta^3}{1-\zeta}}_{H^{1-\theta}}
\leq C(h_0^{-1},\Norm{\zeta}_{H^\mu})\Norm{\zeta}_{L^\infty}^2\Norm{\zeta}_{H^{1-\theta}}^2
\end{multline*}
and
\begin{multline*}
\left|\int_\RR \zeta\left[\left(\partial_x\F_1\left\{\frac{\zeta^2}{1-\zeta}\right\}\right)^2+2(\partial_x\F_1\{\zeta\})\left(\partial_x\F_1\left\{\frac{\zeta^2}{1-\zeta}\right\}\right)  \right]\, \dd x \right|\\
\lesssim \Norm{\zeta}_{L^\infty}\Norm{\frac{\zeta^2}{1-\zeta}}_{H^{1-\theta}}^2+\Norm{\zeta}_{L^\infty}\Norm{\zeta}_{H^{1-\theta}}\Norm{\frac{\zeta^2}{1-\zeta}}_{H^{1-\theta}}
\leq C(h_0^{-1},\Norm{\zeta}_{H^\mu})\Norm{\zeta}_{L^\infty}^2\Norm{\zeta}_{H^{1-\theta}}^2.
\end{multline*}
This concludes the estimate for $\abs{\overline\E_{{\rm rem}}^{(1)}(\zeta)}$. The estimate for $\abs{\underline\E_{{\rm rem}}^{(1)}(\zeta)}$ is obtained identically, and the one for $\abs{\E_{{\rm rem}}^{(2)}(\zeta)}$ are obtained using similar estimates when decomposing $\langle\dd \E(\zeta),\zeta\rangle$ given in~\eqref{def-dEzeta}.
We do not detail for the sake of conciseness.
\end{proof}

\paragraph{Periodic functional setting}
Given $P>0$, we denote $L^2_P$ the space of $P$-periodic, locally square-integrable functions, endowed with the norm
\[ \Norm{u}_{L^2_P}=\Norm{u}_{L^2(-P/2,P/2)}\eqdef \left(\int_{-P/2}^{P/2}\abs{u(x)}^2\ \dd x\right)^{\frac12}.\]
The Fourier coefficients of $u\in L^2_P$ are defined by
\[ \hat u_k\eqdef \frac1{\sqrt{P}} \int_{-P/2}^{P/2} u(x)e^{-\frac{2\mathrm i\pi kx}{P}}\ \dd x, \qquad u(x)= \frac1{\sqrt{P}}   \sum_{k\in\ZZ} \hat u_ke^{\frac{2\mathrm i\pi kx}{P}}.\]
We define, for $s\geq 0$,
\[ H^s_P\eqdef \left\{u\in L^2_P, \quad \Norm{u}_{H^s_P}^2\eqdef \sum_{k\in\ZZ}\left(1+\frac{4\pi^2k^2}{P^2}\right)^s\abs{\hat u_k}^2<\infty\right\}.\]
The Fourier multiplier operator $\Lambda\colon \mathcal S'(\RR) \to \mathcal S'(\RR)$ is defined as usual by 
$\Lambda=(1-\partial_x^2)^{\frac12}$. It maps periodic distributions to periodic distributions and we have
\[
\widehat{\Lambda u}_k=\left(1+\frac{4\pi^2k^2}{P^2}\right)^{\frac12} \hat u_k.
\]
Thus
\[
\Norm{u}_{H_P^s}^2=\int_{-P/2}^{P/2} u \Lambda^{2s} u\ \dd x
\]
and $\Lambda^m$ is an isomorphism from $H^s_P$ to $H^{s-m}_P$ for any $s,m\in \RR$.
Similarly, the operators $\partial_x\F_i$ extend to operators from $\mathcal S'(\RR)$ to itself, and maps smoothly $H^s_P$ into $H_P^{s-1+\theta}$, acting on the Fourier coefficients by pointwise multiplication:
\[ \widehat{\partial_x\F_i u}_k=\frac{2\pi \mathrm{i} k}{P} \F_i(2\pi k/P) \hat u_k.\]
For any $s>1/2$, the continuous embedding
\[ \Norm{u}_{L^\infty}\leq \frac1{\sqrt{P}} \sum_{k\in\ZZ} \abs{\hat u_k}\leq \Norm{u}_{H^s_P}\times\frac1{\sqrt{P}} \left(\sum_{k\in\ZZ} \frac1{(1+\frac{4\pi^2k^2}{P^2})^s}\right)^\frac{1}{2}\lesssim \Norm{u}_{H^s_P},\]
holds uniformly with respect to $P\geq 1$.
More generally, one checks by a partition of unity argument, or repeating the proofs in the periodic setting, that Lemmata~\ref{L.interpolation},\ref{L.composition},~\ref{L.product} and as a consequence Lemmata~\ref{L.EvsY}--\ref{L.E-decomp-order} have immediate analogues in the periodic setting, with uniform estimates with respect to $P\geq1$, when defining
\[
\E_P(\zeta)=\gamma\overline\E_P(\zeta)+\underline\E_P(\zeta)
\]
where 
\begin{align*}
\overline{\E}_P(\zeta)&=\int_{-P/2}^{P/2} \frac{\zeta^2}{1-\zeta}+\frac13 (1-\zeta)^3 \big(\partial_x \F_1\{\frac{\zeta}{1-\zeta}\}\big)^2\dd x,\\
\underline{\E}_P(\zeta)&=\int_{-P/2}^{P/2}\frac{\zeta^2}{\delta^{-1}+\zeta}+\frac13 (\delta^{-1}+\zeta)^3 \big(\partial_x \F_2\{\frac{\zeta}{\delta^{-1}+\zeta}\}\big)^2\dd x.
\end{align*}

\section{The periodic problem}

Our first task is to construct periodic traveling-wave solutions with large periods by considering the periodic minimization problem corresponding to~\eqref{min-pb}. We will use this in the next section to construct a special minimizing sequence for~\eqref{min-pb}, which is useful when $\nu>1-\theta$.
When $\theta<1/2$ and $\nu=1-\theta$, any minimizing sequence has the special property and therefore it is strictly speaking unnecessary to first consider the periodic minimization problem. Nevertheless, we consider here all possible parameters in order to highlight some interesting differences between the cases $\nu=1-\theta$ and $\nu>1-\theta$.

We ensure that the hypotheses of Section~\ref{S.preliminaries}, namely
\[\zeta\in H_P^\nu\quad \text{ and } \quad \Norm{\zeta}_{L^\infty}<\min(1,\delta^{-1})\]
 will be satisfied through a penalization argument. To this aim, we fix $R>0$ and restrict ourselves to values $q\in(0,q_0)$ sufficiently small so that $\Norm{\zeta}_{H_P^\nu(\RR)}\le 2R$ and $ (\gamma+\delta)\Norm{\zeta}_{L^2_P}^2=q$ ensures (by Lemma~\ref{L.interpolation},\ref{L.interpolation 1} in the periodic setting) that $\Norm{\zeta}_{L^\infty}<\min(1,\delta^{-1})-h_0$ with some $h_0>0$, uniformly with respect to $P\geq P_0$ sufficiently large (and likewise in the real line setting).
We then define $\varrho:[0,(2R)^2)\to[0,\infty)$ a smooth, non-decreasing penalization function, satisfying
\begin{enumerate}
\item $\varrho(t)=0$ for $0\leq t\leq R^2$;
\item $\varrho(t)\to\infty$ as $t\nearrow (2R)^2$;
\item For any $a_1\in(0,1)$, there exists $M_1,M_2>0$ and $a_2>1$ such that
\begin{equation}\label{assumption-rho}
 \varrho'(t)\leq M_1\varrho(t)^{a_1}+M_2\varrho(t)^{a_2};
 \end{equation}
\end{enumerate}
for instance $\varrho\colon (R^2,(2R)^2)\ni t \mapsto ((2R)^2-t)^{-1}\exp(\frac{1}{R^2-t})$.

Now consider the functional
\[
\E_{P,\varrho}(\zeta)\eqdef \varrho(\Norm{\zeta}_{H_P^\nu}^2)+\E_P(\zeta)
\]
and the constraint set
\[ V_{P,q,2R}\eqdef \left\{ \zeta\in H_P^\nu, \quad (\gamma+\delta)\Norm{\zeta}_{L^2_P}^2=q\text{ and } \Norm{\zeta}_{H_P^\nu}<2R \right\}.\]
In the following we solve for $q$ sufficiently small and $P$ sufficiently large
\begin{equation}\label{sol-min}
\argmin_{\zeta\in V_{P,q,2R}}  \E_{P,\varrho}(\zeta).
\end{equation}

\begin{Lemma}\label{L.def-zetaP}
There exists $q_0>0$ such that for any $q\in(0,q_0)$, the  functional $\E_{P,\varrho}:V_{P,q,2R}\to\RR $ is weakly lower semi-continuous, bounded from below and $\E_{P,\varrho}(\zeta)\to\infty$ as $\Norm{\zeta}_{H_P^\nu}\nearrow 2R$. In particular, it has a minimizer $\zeta_P\in V_{P,q,2R}$, which satisfies the Euler-Lagrange equation
\begin{equation}\label{Euler-Lagrange-penalization}
 2\varrho'(\Norm{\zeta_P}_{H_P^\nu}^2)\Lambda^{2\nu}\zeta_P+\dd \E_{P}(\zeta_P)+2\alpha_P(\gamma+\delta)\zeta_P=0
 \end{equation}
for some Lagrange multiplier $\alpha_P(\zeta_P)\in\RR$.
\end{Lemma}
\begin{proof}
As explained above, we restrict ourselves to $q\in(0,q_0)$ so that $\E_{P,\varrho}$ is well-defined, and in particular $\sup_{\zeta\in  V_{P,q,2R}} \Norm{\zeta}_{L^\infty}<\min(1,\delta^{-1})$. The argument is standard; see \eg~\cite[\S I.1,I.2]{Struwe}. Since $\E_{P,\varrho}(\zeta)\to\infty$ as $\Norm{\zeta}_{H_P^\nu}\nearrow 2R$, any minimizing sequence is bounded, and therefore weakly convergent (up to a subsequence) in the (reflexive) Hilbert space $H_P^\nu$.  We only need to show that for any $\zeta_n\in V_{P,q,2R}$ such that $\zeta_n \rightharpoonup\zeta$ weakly in $H_P^\nu$, one has
\[ 0\leq \E_{P,\varrho}(\zeta)\leq \liminf_{n\to\infty}\E_{P,\varrho}(\zeta_n).\]
Notice first that by the weak lower semi-continuity of $\Norm{\cdot}_{H_P^\nu}$, and since $\varrho$ is non-decreasing, one has
\[\varrho(\Norm{\zeta}_{H_P^\nu}^2)\leq \liminf_{n\to\infty}\varrho(\Norm{\zeta_n}_{H_P^\nu}^2).\]
Now, by the Rellich-Kondrachov theorem, $\zeta_n\rightharpoonup \zeta$ in $H_P^\nu$ implies that $\zeta_n\rightarrow \zeta$ in $H_P^s$, for $s\in (1/2,\nu)$; and in particular, $\zeta_n\to\zeta$ in $L^\infty$. Moreover, one has $\sup_n\Norm{\zeta_n}_{L^\infty}<\min(1,\delta^{-1})$ since $\zeta_n\in V_{P,q,2R}$, and therefore $\frac{\zeta_n}{1-\zeta_n}\to \frac{\zeta}{1-\zeta}$ in $L^\infty$, and $\frac{\zeta}{1-\zeta}\in H_P^\nu$ by  Lemma~\ref{L.product},\ref{L.product 3}, \ref{L.product 4}.
Since $\frac{\zeta_n}{1-\zeta_n}$ is uniformly bounded in $H^\nu_P$ and converges in $L^\infty$, it 
follows that $\frac{\zeta_n}{1-\zeta_n}\rightharpoonup \frac{\zeta}{1-\zeta}$ in $H_P^\nu$, and therefore $\partial_x\F_1\left\{\frac{\zeta_n}{1-\zeta_n}\right\}\rightharpoonup \partial_x\F_1\left\{\frac{\zeta}{1-\zeta}\right\}$ in $L^2_P$ by Lemma~\ref{L.Fi-bounds},\ref{L.Fi-bounds 1}, and finally \[(1-\zeta_n)^{\frac32}\partial_x\F_1\left\{\frac{\zeta_n}{1-\zeta_n}\right\}\rightharpoonup (1-\zeta)^{\frac32}\partial_x\F_1\left\{\frac{\zeta}{1-\zeta}\right\} \quad \text{ in } L^2_P.\]
In the same way, we find
\[(\delta^{-1}+\zeta_n)^{\frac32} \partial_x \F_2\left\{\frac{\zeta_n}{\delta^{-1}+\zeta_n}\right\}\rightharpoonup (\delta^{-1}+\zeta)^{\frac32} \partial_x \F_2\left\{\frac{\zeta}{\delta^{-1}+\zeta}\right\} \quad \text{ in } L^2_P.\]
and
\[(1-\zeta_n)^{-\frac12}\zeta_n \rightarrow (1-\zeta)^{-\frac12}\zeta, \qquad (\delta^{-1}+\zeta_n)^{-\frac12}\zeta_n \rightarrow  (\delta^{-1}+\zeta)^{-\frac12}\zeta  \quad \text{ in } L^2_P.\]
The result follows from the weak lower semi-continuity of $\Norm{\cdot}_{L^2_P}$.
\end{proof}

Now we wish to prove that $\zeta_P\in V_{P,q,R}$, and in particular satisfies the Euler-Lagrange equation
\[ \dd \E_{P}(\zeta_P)+2\alpha_P (\gamma+\delta)\zeta_P=0.\]
From this point on, we heavily make use of the property (see Assumption~\ref{D.parameters})
\[\gamma-\delta^2\neq 0.\]
without explicit references in the statements.

\begin{Lemma}\label{L.I-bound-from-below} 
There exists $m>0$ and $q_0>0$ such that for any $q\in(0,q_0)$, 
\[ I_q\eqdef \inf\{\E(\zeta),\ \zeta\in V_{q,R}\}<q(1-m q^{\frac23})\]
and there exists $P_q>0$ such that
\[ I_{P,\varrho,q}\eqdef \inf\{\E_{P,\varrho}(\zeta),\ \zeta\in V_{P,q,2R}\}<q(1-m q^{\frac23})\]
for any $P\geq P_q$.
\end{Lemma}
\begin{proof}
Let us first consider the case of the real line.
Consider $\psi\in \C^\infty(\RR)$ with compact support, such that ${(\gamma+\delta)\Norm{\psi}_{L^2}^2=1}$; and denote $\psi_\lambda:x\mapsto \lambda^{\frac12}\psi(\lambda x)$. One has
\[ \int_\RR\psi_\lambda^3 \, \dd x =\lambda^{\frac12}\int_\RR\psi^3 \, \dd x \quad \text{ and } \quad \Norm{\partial_x\psi_\lambda}_{L^2}=\lambda \Norm{\partial_x\psi}_{L^2}.\]
It follows that, for the case when $\gamma-\delta^2<0$, one can choose $\psi\geq 0$ and $\lambda$ small enough so that
\[
\int_\RR(\gamma-\delta^2)\psi_\lambda^3+\frac{\gamma+\delta^{-1}}{3}(\partial_x\psi_\lambda)^2\ \dd x\eqdef-2m<0.
\]
If $\gamma-\delta^2>0$, we instead let $\psi\leq 0$ and again choose $\lambda$ small enough so that the above holds.

Now, consider $\phi_q:x\mapsto q^{\frac23}\psi_\lambda(q^{\frac13}x)$. One has
\[ \Norm{\phi_q}_{L^\infty}\leq q^{\frac23}\Norm{\psi_\lambda}_{L^\infty} , \quad \int_\RR\phi_q^3\, \dd x=q^{\frac53}\int_\RR\psi_\lambda^3 \, \dd x \quad \text{ and } \quad \Norm{\partial_x^n\phi_q}_{L^2}=q^{\frac12+\frac{n}3} \Norm{\partial_x^n\psi_\lambda}_{L^2} \quad (n\in\NN).\]
In particular, for $q$ sufficiently small, $\Norm{\phi_q}_{H^\nu}<R$; and by Lemma~\ref{L.E-decomp-KdV} with $l=3$,
\begin{align*}
\E(\phi_q)&=(\gamma+\delta)\int_\RR\phi_q^2\ \dd x+\int_\RR(\gamma-\delta^2)\phi_q^3+\frac{\gamma+\delta^{-1}}{3}(\partial_x\phi_q)^3\ \dd x+\mathcal{O}(q^\frac73)\\
& =q-2m q^\frac{5}{3}+\mathcal{O}(q^\frac73).
\end{align*}
The result follows in the real-line setting.

We now deduce the result in the periodic setting. By taking $P\geq P_q$ sufficiently large, we may ensure
\[ \supp \phi_q \subset (-P/2,P/2),\]
and define 
\[ \phi_{P,q}=\sum_{j\in\ZZ} \phi_q(x-jP)\in H_P^n \quad (n\in\NN).\]
One has
\[ (\gamma+\delta)\Norm{\phi_{P,q}}_{L^2_P}^2= (\gamma+\delta)\Norm{\phi_{q}}_{L^2}^2=q\]
and$ \Norm{\phi_{P,q}}_{H_P^\nu}<R$ for $q$ sufficiently small, so that
as well as 
\begin{align*}
\E_{P,\varrho}(\phi_{P,q})=\E_P(\phi_{P,q})&=(\gamma+\delta)\int_{-P/2}^{P/2}\phi_{P,q}^2\ \dd x+\int_{-P/2}^{P/2}(\gamma-\delta^2)\phi_{P,q}^3+\frac{\gamma+\delta^{-1}}{3}(\partial_x\phi_{P,q})^2\ \dd x+\mathcal{O}(q^\frac73)\\
& =q-2m q^\frac{5}{3}+\mathcal{O}(q^\frac73).
\end{align*}
The result is proved.
\end{proof}

\begin{Lemma}\label{L.alpha-bounds}
There exists $q_0>0$ such that for any $q\in(0,q_0)$, one has 
\[\forall P\geq P_q, \qquad |\alpha_P+1|< \frac{1}{2},\]
where $\alpha_P$ is defined in Lemma~\ref{L.def-zetaP} and $P_q$ in Lemma~\ref{L.I-bound-from-below}.
\end{Lemma}
\begin{proof}
We use the Euler-Lagrange equation satisfied by $\zeta_P$, namely
\[ 2\varrho'(\Norm{\zeta_P}_{H_P^\nu}^2)\Lambda^{2\nu}\zeta_P+\dd \E_{P}(\zeta_P)+2\alpha_P(\gamma+\delta)\zeta_P=0.\]
This equation is well-defined in $(H_P^\nu)'$ and testing against $\zeta_P\in H_P^\nu$ yields 
\[
2\varrho'(\Norm{\zeta_P}_{H_P^\nu}^2)\Norm{\zeta_P}_{H_P^\nu}^2+2\alpha_P(\gamma+\delta)\Norm{\zeta_P}_{L^2_P}^2+\langle \dd\E_P(\zeta_P), \zeta_P\rangle=0.
\]
Using the decompositions in Lemma~\ref{L.E-decomp-order} (changing the domain of integration to $[-P/2,P/2]$) yields
\begin{align*}
\langle\dd \E_P(\zeta_P),\zeta_P\rangle&=2{\E}_{2,P}(\zeta_P)+3{\E}_{3,P}(\zeta_P)+{\E}_{{\rm rem},P}^{(2)}(\zeta_P)\\
&=2{\E}_{P}(\zeta_P)+{\E}_{3,P}(\zeta_P)+{\E}_{{\rm rem},P}^{(2)}(\zeta_P)-2{\E}_{{\rm rem},P}^{(1)}(\zeta_P),
\end{align*}
so that one obtains the identity
\begin{equation}\label{eq-1}
-\alpha_P\ q = {\E}_{P}(\zeta_P)+\frac12{\E}_{3,P}(\zeta_P)+\frac12{\E}_{{\rm rem},P}^{(2)}(\zeta_P)-{\E}_{{\rm rem},P}^{(1)}(\zeta_P) + \varrho'(\Norm{\zeta_P}_{H_P^\nu}^2)\Norm{\zeta_P}_{H_P^\nu}^2.
\end{equation}

Let us now use Lemma~\ref{L.I-bound-from-below}, which asserts
\begin{equation} \label{eq-2} 
\varrho(\Norm{\zeta_P}_{H_P^\nu}^2)+\E_P(\zeta_P)< q(1-mq^{\frac23})\leq q.
\end{equation}
Remark that, since $1-\zeta,\delta^{-1}+\zeta\geq h_0>0$, one has
\begin{align*}\E_P(\zeta_P)&\geq \int_{-P/2}^{P/2}\gamma\frac{\zeta_P^2}{1-\zeta_P}+\frac{\zeta_P^2}{\delta^{-1}+\zeta_P}\ \dd x\\
&=(\gamma+\delta) \int_{-P/2}^{P/2}\zeta_P^2 \ \dd x+\gamma\int_{-\frac{P}{2}}^{\frac{P}{2}}\frac{\zeta_P^3}{1-\zeta_P}\ \dd x-\delta\int_{-\frac{P}{2}}^{\frac{P}{2}}\frac{\zeta_P^3}{\delta^{-1}+\zeta_P}\ \dd x\\
&=q+\O(q^{1+\frac\epsilon{2\nu}}),
\end{align*}
where $\epsilon=\nu-1/2>0$ and we use in the last estimate that $\Norm{\zeta_P}_{L^\infty}^2\lesssim \Norm{\zeta_P}_{L^2_P}^{2-\frac{1}{\nu}}\Norm{\zeta_P}_{H_P^\nu}^{\frac{1}{\nu}}=\O(q^{\frac{2\nu-1}{2\nu}})$, by the interpolation estimate Lemma~\ref{L.interpolation},\ref{L.interpolation 1} in the periodic-setting.

Combining with~\eqref{eq-2} yields
\[\E_P(\zeta_P)=q+\O(q^{1+\frac{\epsilon}{2\nu}})\]
and
\[
\varrho(\Norm{\zeta_P}_{H_P^\nu}^2)=\mathcal{O}(q^{1+\frac\epsilon{2\nu}}).
\]
Hence, by~\eqref{assumption-rho}
\[
\varrho'(\Norm{\zeta_P}_{H_P^\nu}^2)=\mathcal{O}( q^{1+\frac\epsilon{4\nu}}).
\]
Now, by Lemma~\ref{L.EvsY} and using once again~\eqref{eq-2}, one has
\[
\Norm{\zeta_P}_{H_P^{1-\theta}}^2 \lesssim \E_P(\zeta_P)=\O(q).
\]
Thus by Lemma~\ref{L.E-decomp-order}, one has
\[
\abs{{\E}_{3,P}(\zeta_P)}\lesssim \Norm{\zeta_P}_{L^\infty}\Norm{\zeta_P}_{H_P^{1-\theta}}^2=\mathcal{O}(q^{1+\frac\epsilon{2\nu}})
\]
and
\[
\abs{{\E}_{{\rm rem}, P}^{(2)}(\zeta_P)}+\abs{{\E}_{{\rm rem}, P}^{(1)}(\zeta_P)}\lesssim \Norm{\zeta_P}_{L^\infty}^2\Norm{\zeta_P}_{H_P^{1-\theta}}^2=\mathcal{O}(q^{1+\frac\epsilon\nu}).
\]
Plugging the above estimates into~\eqref{eq-1} yields
\[-\alpha_P\ q =q+\mathcal{O}(q^{1+\frac\epsilon{4\nu}}),\]
and the proof is complete.
\end{proof}

\begin{Lemma}\label{L.zetaP-H1-estimate}
Let $q_0$, $q\in(0,q_0)$ and $P_q$ be as in Lemma~\ref{L.alpha-bounds}. There exists $M>0$ such that one has
\[\Norm{\zeta}_{H_P^\nu}^2\leq M q\]
 uniformly over $q\in(0,q_0)$, $P\geq P_q$ and $\zeta$ in the set of minimizers of $\E_{P,\varrho}$ over $V_{P,q,2R}$.
\end{Lemma}
\begin{proof}
Recall the Euler-Lagrange equation:
\begin{equation}\label{Euler-Lagrange-2}
2\varrho'(\Norm{\zeta_P}_{H_P^\nu}^2)\Lambda^{2\nu}\zeta_P+\dd \E_{P}(\zeta_P)+2\alpha_P\zeta_P=0.\end{equation}
It follows from the proof of Lemma~\ref{L.alpha-bounds} that for $q_0$ sufficiently small
\[ \Norm{\zeta_P}_{H_P^{1-\theta}}^2\lesssim q\]
with $0\leq \theta<1$. Thus the result is proved if $\nu=1-\theta$, and we focus below on the situation $\nu>1-\theta$. In this case we obtain the desired estimate in a similar fashion after finite induction. Indeed, define $r_n=\min(\nu-(1-\theta), n(1-\theta))$, $n\in \NN$, and assume that $\Norm{\zeta_P}_{H_P^{r_n}}^2\lesssim q$. Note that this is satisfied for $n=0$ by assumption. We will show below that 
\begin{equation}\label{induction}
\Norm{\zeta_P}_{H_P^{1-\theta+r_n}}^2\lesssim \Norm{\zeta_P}_{H_P^{r_n}}^2 \lesssim q.
\end{equation}
Since $1-\theta>0$, the desired result follows by finite induction.
\medskip

Let us now prove~\eqref{induction}. We test~\eqref{Euler-Lagrange-2} against $\Lambda^{2r_n}\zeta_P$, and obtain
\begin{equation}
2\varrho'(\Norm{\zeta_P}_{H_P^\nu}^2)\langle \Lambda^{2\nu}\zeta_P,\Lambda^{2r_n}\zeta_P\rangle+ \langle\dd \E_{P}(\zeta_P),\Lambda^{2r_n}\zeta_P\rangle+2\alpha_P\langle\zeta_P,\Lambda^{2r_n}\zeta_P\rangle=0.\label{Euler-Lagrange-Lambda-1}
\end{equation}
Here, the notation $\langle,\rangle$ represents the $H_P^{\nu-2(1-\theta)}-H_P^{-\nu+2(1-\theta)}$ duality bracket. We will use the same notation for $H_P^s-H_P^{-s}$, where the value of $s\in(-\nu,\nu]$ is clear from the context.
Note that all the terms are well-defined, since $\zeta_P\in H_P^\nu$, and therefore by~\eqref{def-dE} and Lemma~\ref{L.product}, $\dd\E_P(\zeta_P)\in H_P^{\nu-2(1-\theta)}$. Moreover, by~\eqref{Euler-Lagrange-2}, if $\varrho'(\Norm{\zeta_P}_{H_P^\nu}^2)>0$ then $\Lambda^{2\nu}\zeta_P\in H_P^{\nu-2(1-\theta)}$ as well.  Finally,  $\Lambda^{2r_n}\zeta_P\in H_P^{\nu-2r_n}$, and $r_n+1-\theta\leq \nu$ so that $\nu-2r_n\ge -\nu+2(1-\theta)$. 

Now, using that $\varrho'(\Norm{\zeta_P}_{H_P^\nu}^2)\geq 0$, we get from~\eqref{Euler-Lagrange-Lambda-1} and Lemma~\ref{L.alpha-bounds} that 
\begin{equation}
\gamma\langle\dd \overline{\E}_{P}(\zeta_P),\Lambda^{2r_n}\zeta_P\rangle+\langle\dd \underline{\E}_{P}(\zeta_P),\Lambda^{2r_n}\zeta_P\rangle\leq 2(-\alpha_P)\Norm{\Lambda^{r_n}\zeta_P}_{L^2_P}^2\leq 3  \Norm{\zeta_P}_{H_P^{r_n}}^2,\label{Euler-Lagrange-Lambda-2}
\end{equation}
where we define $\dd \overline{\E}_{P}$ and $\dd \underline{\E}_{P}$ from~\eqref{def-dE} as in~\eqref{def-E}. In particular, 
\begin{multline}
\langle \dd \overline{\E}(\zeta_P), \Lambda^{2r_n}\zeta_P\rangle=\left\langle   \frac{2\zeta_P }{1-\zeta_P},\Lambda^{2r_n}\zeta_P\right\rangle+\left\langle \frac{\zeta_P^2}{(1-\zeta_P)^2},\Lambda^{2r_n}\zeta_P\right\rangle+\left\langle (1-\zeta_P)^2\left(\partial_x\F\left\{\frac{\zeta_P}{1-\zeta_P}\right\}\right)^2,\Lambda^{2r_n}\zeta_P\right\rangle\\ +\left\langle \frac{2}{3}(1-\zeta_P)^3\partial_x\F_1\left\{\frac{\zeta_P}{1-\zeta_P}\right\},\partial_x\F_1\left\{\frac{\Lambda^{2r_n}\zeta_P}{(1-\zeta_P)^2}\right\}\right\rangle.
\label{Der-zeta}
\end{multline}
We estimate each term of~\eqref{Der-zeta}, using that $(\gamma+\delta)\Norm{\zeta_P}_{L_P^2}^2=q$, $\Norm{\zeta_P}_{H_P^\nu}<2R$ and $\Norm{\zeta_P}_{L^\infty}<1-h_0$, recalling that Lemmata~\ref{L.interpolation},\ref{L.composition},\ref{L.product} and~\ref{L.Fi-bounds}.\ref{L.Fi-bounds 1} are valid in the periodic setting.

The first term in~\eqref{Der-zeta} is estimated by the Cauchy-Schwarz inequality and Lemma~\ref{L.composition}:
\begin{align}
\left|\left\langle   \frac{2\zeta_P }{1-\zeta_P},\Lambda^{2r_n}\zeta_P\right\rangle\right|
&=\left|\left\langle \Lambda^{r_n}\left(\frac{2\zeta_P}{1-\zeta_P}\right),\Lambda^{r_n}\zeta_P\right\rangle\right|
\leq 2\Norm{\frac{\zeta_P}{1-\zeta_P}}_{H_P^{r_n}}\Norm{\zeta_P}_{H_P^{r_n}}\lesssim \Norm{\zeta_P}_{H_P^{r_n}}^2. \label{zeta-est1}
\end{align} 
It is clear that the second term can be estimated in the same way. Next we see that
\begin{align*}
\left|\left\langle (1-\zeta_P)^2\left(\partial_x\F\left\{\frac{\zeta_P}{1-\zeta_P}\right\}\right)^2,\Lambda^{2r_n}\zeta_P\right\rangle\right|
&\leq\Norm{(1-\zeta_P)^2\left(\partial_x\F_1\left\{\frac{\zeta_P}{1-\zeta_P}\right\}\right)^2}_{H_P^{\theta - 1+r_n}}\Norm{\zeta_P}_{H_P^{1-\theta+r_n}}\\
&\lesssim \Norm{\left(\partial_x\F_1\left\{\frac{\zeta_P}{1-\zeta_P}\right\}\right)^2}_{H_P^{\theta-1+r_n}}\Norm{\zeta_P}_{H_P^{1-\theta+r_n}},
\end{align*}
and Lemma~\ref{L.product},\ref{L.product 2}, Lemma~\ref{L.Fi-bounds},\ref{L.Fi-bounds 1} and Lemma~\ref{L.interpolation},\ref{L.interpolation 2} yield for any $0<\epsilon<\min(\nu+\theta-1,1-\theta,\nu-1/2)$,
\begin{align*}
\Norm{\left(\partial_x\F_1\left\{\frac{\zeta_P}{1-\zeta_P}\right\}\right)^2}_{H_P^{\theta-1+r_n}}&\lesssim \Norm{\partial_x\F_1\left\{\frac{\zeta_P}{1-\zeta_P}\right\}}_{H_P^{r_n}}\Norm{\partial_x\F_1\left\{\frac{\zeta_P}{1-\zeta_P}\right\}}_{H_P^{\nu+\theta-1-\epsilon}}\\
&\lesssim \Norm{\zeta_P}_{H_P^{1-\theta+r_n}}\Norm{\zeta_P}_{H_P^{\nu-\epsilon}}\\
& \lesssim \Norm{\zeta_P}_{H_P^{1-\theta+r_n}}q^{\frac\epsilon{2\nu}}.
\end{align*}
We therefore have that
\begin{equation}
\left|\left\langle (1-\zeta_P)^2\left(\partial_x\F\left\{\frac{\zeta_P}{1-\zeta_P}\right\}\right)^2,\Lambda^{2r_n}\zeta_P\right\rangle\right|
\lesssim q^{\frac\epsilon{2\nu}} \Norm{\zeta_P}_{H_P^{1-\theta+r_n}}^2. \label{zeta-est3}
\end{equation}
We next consider the remaining term in~\eqref{Der-zeta} and note that
\begin{align}
\nonumber
&\left\langle (1-\zeta_P)^3\partial_x\F_1\left\{\frac{\zeta_P}{1-\zeta_P}\right\},\partial_x\F_1\left\{\frac{\Lambda^{2r_n}\zeta_P}{(1-\zeta_P)^2}\right\}\right\rangle\\
&=\left\langle\partial_x\F_1\{\zeta_P\},\partial_x\F_1\{\Lambda^{2r_n}\zeta_P\} \right\rangle\nonumber \\
&\quad+\underbrace{\left\langle(-3\zeta_P+3\zeta_P^2-\zeta_P^3)\partial_x\F_1\left\{\frac{\zeta_P}{1-\zeta_P}\right\},\partial_x\F_1\left\{\frac{\Lambda^{2r_n}\zeta_P}{(1-\zeta_P)^2}\right\}\right\rangle}_{I}\nonumber\\
&\quad +\underbrace{\left\langle\partial_x\F_1\left\{\frac{\zeta_P}{1-\zeta_P}-\zeta_P\right\},\partial_x\F_1\left\{\frac{\Lambda^{2r_n}\zeta_P}{(1-\zeta_P)^2}\right\}\right\rangle}_{II}\nonumber\\
&\quad+\underbrace{\left\langle\partial_x\F_1\{\zeta_P\},\partial_x\F_1\left\{\frac{\Lambda^{2r_n}\zeta_P}{(1-\zeta_P)^2}-\Lambda^{2r_n}\zeta_P\right\}\right\rangle}_{III}. \label{int-expansion}
\end{align}
First we see that, by Lemma~\ref{L.Fi-bounds},\ref{L.Fi-bounds 1},
\begin{equation}\label{zeta-est4}
\Norm{\zeta}_{H_P^{r_n}}^2+\left\langle\partial_x\F_1\{\zeta_P\},\partial_x\F_1\{\Lambda^{2r_n}\zeta_P\} \right\rangle\gtrsim\Norm{\zeta}_{H_P^{1-\theta+r_n}}^2.
\end{equation}
We estimate $I$ in~\eqref{int-expansion} proceeding as previously:
\begin{align}
\left|\left\langle\zeta_P\partial_x\F_1\left\{\frac{\zeta_P}{1-\zeta_P}\right\},\partial_x\F_1\left\{\frac{\Lambda^{2r_n}\zeta_P}{(1-\zeta_P)^2}\right\}\right\rangle\right|&\lesssim \Norm{\partial_x\F_1\left\{\frac{\zeta_P}{1-\zeta_P}\right\}}_{H_P^{r_n}}\Norm{\zeta_P
\partial_x\F_1\left\{\frac{\Lambda^{2r_n}\zeta_P}{(1-\zeta_P)^2}\right\}}_{H_P^{-r_n}}\nonumber\\
&\lesssim \Norm{\zeta_P}_{H_P^{1-\theta+r_n}} \Norm{\zeta_P}_{H_P^{\nu-\epsilon}} 
\Norm{\partial_x\F_1\left\{\frac{\Lambda^{2r_n}\zeta_P}{(1-\zeta_P)^2}\right\}}_{H_P^{-r_n}} \nonumber\\
&\lesssim 
\Norm{\zeta_P}_{H_P^{\nu-\epsilon}} 
\Norm{\zeta_P}_{H_P^{1-\theta+r_n}}^2\\
& \lesssim q^{\frac\epsilon{2\nu}} \Norm{\zeta_P}_{H_P^{1-\theta+r_n}}^2, 
\label{zeta-est5}
\end{align}
where we choose $0<\epsilon<\min\{\nu-1/2, 1-\theta\}$.
The remaining terms in $I$ are of higher order and can be estimated in the same way. Next we estimate $II$:
\begin{align}
\left|\left\langle\partial_x\F_1\left\{\frac{\zeta_P^2}{1-\zeta_P}\right\},\partial_x\F_1\left\{\frac{\Lambda^{2r_n}\zeta_P}{(1-\zeta_P)^2}\right\}\right\rangle\right|&\lesssim\Norm{\zeta_P^2}_{H_P^{1-\theta+r_n}}\Norm{\Lambda^{2r_n}\zeta_P}_{H_P^{1-\theta-r_n}}\nonumber\\
&\lesssim \Norm{\zeta_P}_{L^\infty}\Norm{\zeta_P}_{H_P^{1-\theta+r_n}}^2\\
&\lesssim q^{\frac{\nu-1/2}{2\nu}} \Norm{\zeta_P}_{H_P^{1-\theta+r_n}}^2\label{zeta-est6},
\end{align}
where we used Lemma~\ref{L.product},\ref{L.product 1} and Lemma~\ref{L.interpolation},\ref{L.interpolation 1}.
 Finally consider $III$: proceeding as above,
 \begin{align}\nonumber
 \left|\left\langle\partial_x\F_1\{\zeta_P\},\partial_x\F_1\left\{\frac{\Lambda^{2r_n}\zeta_P}{(1-\zeta_P)^2}-\Lambda^{2r_n}\zeta_P\right\}\right\rangle\right|
 &\lesssim 
 \Norm{\zeta_P}_{H_P^{1-\theta+r_n}}\Norm{\frac{2\zeta_P-\zeta_P^2}{(1-\zeta_P)^2}\Lambda^{2r_n}\zeta_P}_{H_P^{1-\theta-r_n}} \nonumber\\
 &\leq \Norm{\zeta_P}_{H_P^{1-\theta+r_n}}^2\Norm{\frac{2\zeta_P-\zeta_P^2}{(1-\zeta_P)^2}}_{H_P^{\nu-\epsilon}} \\
 &\lesssim q^{\frac\epsilon{2\nu}}\Norm{\zeta_P}_{H_P^{1-\theta+r_n}}^2,\label{zeta-est7}
 \end{align}
 with $0<\epsilon<\min(\nu-1/2, \nu-(1-\theta),1-\theta)$.
 
 Collecting~\eqref{zeta-est1}--\eqref{zeta-est7} in~\eqref{Der-zeta} yields
 \[\Norm{\zeta}_{H_P^{1-\theta+r_n}}^2\lesssim \langle \dd \overline{\E}(\zeta_P), \Lambda^{2r_n}\zeta_P\rangle+\Norm{\zeta_P}_{H_P^{r_n}}^2+q^{\frac\epsilon{2\nu}}\Norm{\zeta_P}_{H_P^{1-\theta+r_n}}^2\]
 with $\epsilon>0$ sufficiently small.
It is clear that the same estimate holds for $\langle \dd\underline{\E}(\zeta_P),\Lambda^{2r_n}\zeta_P\rangle$, and if we use these in~\eqref{Euler-Lagrange-Lambda-2}, we obtain
\[
\Norm{\zeta_P}_{H_P^{1-\theta+r_n}}^2 \lesssim  \Norm{\zeta_P}_{H_P^{r_n}}^2+\Norm{\zeta_P}_{H_P^{1-\theta+r_n}}^2q^{\frac\epsilon{2\nu}}.
\]
Thus one may choose $q$ sufficiently small, so that~\eqref{induction} holds. This concludes the proof.
\end{proof}

We now collect the preceding results and deduce solutions of the non-penalized periodic problem.
\begin{Theorem}[Existence of periodic minimizers]\label{T.periodic-minimizers}
There exists $q_0>0$ such that for any $q\in(0,q_0)$, one can define $P_q>0$ and the following holds. For each $P\geq P_q$, there exists $\zeta_P\in V_{P,q,R}$ such that
\[\E_P(\zeta_P)=\inf_{\zeta\in V_{P,q,R}} \E_P(\zeta)\eqdef I_{P,q}.\]
and the Euler-Lagrange equation holds with $\alpha_P\in(-3/2,-1/2)$:
\begin{equation}\label{Euler-Lagrange} \dd \E_P(\zeta_P)+2\alpha_P(\gamma+\delta) \zeta_P=0.\end{equation}
Furthermore, there exists $M>0$, independent of $q$, such that
\[\Norm{\zeta_P}_{H_P^\nu}^2\leq M q
\]
uniformly with respect to $P\geq P_q$.
\end{Theorem}
\begin{proof}
From Lemma~\ref{L.zetaP-H1-estimate}, any minimizer of $\E_{P,\varrho}$ over $V_{P,q,2R}$ satisfies, for $q_0$ sufficiently small and $P\geq P_q$ sufficiently large,
\[ \Norm{\zeta_P}_{H_P^\nu}^2\leq Mq< R^2.\]
Thus the Euler-Lagrange equation~\eqref{Euler-Lagrange-penalization} becomes~\eqref{Euler-Lagrange}, and the control on $\alpha_P$ is stated in Lemma~\ref{L.alpha-bounds}.
Moreover, since $\E_{P,\varrho}=\E_{P}$ over $V_{P,q,R}$, $\zeta_P$ minimizes $\E_{P} $ over $V_{P,q,R}$. The theorem is proved.
\end{proof}

\begin{Remark}\label{R.nu=1-theta}
If $\theta\in [0,1/2)$ and $\nu=1-\theta$, then the functional $\E_{P}$ is coercive on $V_{P,q,R}$ by Lemma~\ref{L.EvsY}, and it isn't necessary to consider the penalized functional $\E_{P, \varrho}$ to construct periodic minimizers. Indeed, one can minimize $\E_{P}$ over $V_{P,q,R}$ directly, noting that any minimizing sequence satisfies (up to subsequences) $\sup_n \Norm{\zeta_{P, n}}_{H_P^\nu}^2 \le Mq<R^2$ if $q_0$ is sufficiently small.
\end{Remark}

\section{The real line problem}

The construction of a minimizer for the real line problem,~\eqref{min-pb}, will follow from Lions' concentration compactness principle. The main difficulty consists in excluding the ``dichotomy'' scenario. To this aim, we shall use a special minimizing sequence (satisfying the additional estimate $\Norm{\zeta_n}_{H^\nu}^2\lesssim q$) to show that the function $q\mapsto I_q$ is strictly subhomogeneous (see Proposition~\ref{P.subhomogeneity}), which implies that it is also strictly subadditive (Corollary~\ref{C.subadditivity}). This special subsequence is constructed from the solutions of the periodic problem, obtained in Theorem~\ref{T.periodic-minimizers}, with period $P_n\to\infty$.

\subsection{A special minimizing sequence}

\begin{Theorem}[Special minimizing sequence for $\E$]\label{T.special-minimizing-sequence}
There exists $q_0>0$ such that for any $q\in(0,q_0)$, one can define constants $m,M>0$ and a sequence $\{\zeta_n\}_{n\in\NN}$ satisfying 
\[ (\gamma+\delta)\Norm{\zeta_n}_{L^2}^2=q, \quad \Norm{\zeta_n}_{H^\nu}^2\leq M q \]
and
\[\lim_{n\to\infty}\E(\zeta_n)=I_q\eqdef \inf_{\zeta\in V_{q,R}}\E(\zeta)< q(1-mq^{\frac23}).\]
\end{Theorem}

\begin{proof}The estimate on $I_q$ was proved in Lemma~\ref{L.I-bound-from-below}; thus we only need to construct a minimizing sequence satisfying  $\Norm{\zeta_n}_{H^\nu}^2\leq M q$.
If $\nu=1-\theta$, then any minimizing sequence satisfies this property as a consequence of Lemma~\ref{L.EvsY}, so we assume in the following that $\nu>1-\theta$.
Let $q_0$ be sufficiently small so that Theorem~\ref{T.periodic-minimizers} holds. By the construction of~\cite[p. 2918 and proof of Theorem 3.8]{EhrnstromGrovesWahlen12}, one obtains, for any $P_n$ sufficiently large, $x_n\in\RR$, $\widetilde\zeta_{P_n}\in H^\nu_{P_n}$ and $\zeta_n\in H^\nu(\RR)$ such that
\begin{equation}\label{prop1}
\Norm{\widetilde\zeta_{P_n}-\zeta_{P_n}(\cdot-x_n)}_{L^2_{P_n}}\to 0 \quad (P_n\to\infty)
\end{equation}
where $\zeta_{P_n}$ is defined by Theorem~\ref{T.periodic-minimizers}, 
\begin{equation}\label{prop2}
 \supp \zeta_n\subset (-P_n/2+P_n^{1/2},P_n/2-P_n^{1/2}) \quad \text{ and } \quad \widetilde\zeta_{P_n}=\sum_{l\in\ZZ} \zeta_n(\cdot +lP_n).
\end{equation}
Moreover, one has
\begin{equation}\label{prop3}
\Norm{\zeta_{n}}_{L^2}=\Norm{\widetilde\zeta_{P_n}}_{L^2_P}=\Norm{\zeta_{P_n} }_{L^2_P}.\
\end{equation}
and
\begin{equation}\label{prop4}
\Norm{\zeta_n}_{H^\nu}\lesssim\Norm{\widetilde\zeta_{P_n}}_{H^\nu_{P_n}}\lesssim \Norm{\zeta_{P_n}}_{H^\nu_{P_n}}
\end{equation}
 uniformly with respect to $P_n$ sufficiently large.

By~\eqref{prop4} and Theorem~\ref{T.periodic-minimizers}, one has $\Norm{\zeta_n}_{H^\nu}^2 \leq Mq<R^2$ provided that $P_n$ is sufficiently large and $q_0$ is sufficiently small; and $\zeta_n\in V_{q,R}$ by~\eqref{prop3}. Thus there only remains to prove that $\zeta_n$ is a minimizing sequence.

Notice that by~\eqref{prop1} and~\eqref{prop4} and Lemma~\ref{L.interpolation},\ref{L.interpolation 2},
\begin{equation}\label{prop5}
\Norm{\widetilde \zeta_{P_n}-\zeta_{P_n}(\cdot -x_n)}_{H^{\nu'}_{P_n}}\to 0
\end{equation}
for any $\nu'\in [0,\nu)$.

One has
\begin{multline*}\overline\E_{P_n}(\widetilde\zeta_{P_n})-\overline\E(\zeta_n)=\frac13  \int_{-P_n/2}^{P_n/2} (1-\zeta_n)^3\Big[ \big(\partial_x \F_1\{\frac{\widetilde\zeta_{P_n}}{1-\widetilde\zeta_{P_n}}\}\big)^2-\big(\partial_x \F_1\{\frac{\zeta_{n}}{1-\zeta_{n}}\}\big)^2\Big]\dd x\\
-\frac13\int_{\RR\setminus(-P_n/2,P_n/2)} (1-\zeta_n)^3\big(\partial_x \F_1\{\frac{\zeta_{n}}{1-\zeta_{n}}\}\big)^2\, \dd x.
\end{multline*}
Using Lemma~\ref{L.Fi-bounds},\ref{L.Fi-bounds 1},~\eqref{prop4} and the Cauchy-Schwarz inequality, we deduce
\begin{multline*}
\left| \int_{-P_n/2}^{P_n/2} (1-\zeta_n)^3\Big[ \big(\partial_x \F_1\{\frac{\widetilde\zeta_{P_n}}{1-\widetilde\zeta_{P_n}}\}\big)^2-\big(\partial_x \F_1\{\frac{\zeta_{n}}{1-\zeta_{n}}\}\big)^2\Big]\dd x\right|\\
\lesssim C(\Norm{\zeta_n}_{H^\nu}) \Norm{\partial_x \F_1\{\frac{\widetilde\zeta_{P_n}}{1-\widetilde\zeta_{P_n}}\}-\partial_x \F_1\{\frac{\zeta_{n}}{1-\zeta_{n}}\}}_{L^2(-P_n/2,P_n/2)}.
\end{multline*}
Notice now that, by uniqueness of the Fourier decomposition in $L^2_{P_n}$, one has the identity
\[\partial_x \F_1\{\frac{\widetilde\zeta_{P_n}}{1-\widetilde\zeta_{P_n}}\}(x)=\sum_{l\in\ZZ}\Big(\partial_x \F_1\{\frac{\zeta_{n}}{1-\zeta_{n}}\}\Big)(x+lP_n),\]
and therefore, by Lemma~\ref{L.Fi-bounds},\ref{L.Fi-bounds 3} and~\eqref{prop2}, one has
\begin{align*}
\Norm{\partial_x \F_1\{\frac{\widetilde\zeta_{P_n}}{1-\widetilde\zeta_{P_n}}\}-\partial_x \F_1\{\frac{\zeta_{n}}{1-\zeta_{n}}\}}_{L^2(-P_n/2,P_n/2)}^2&=\int_{-P_n/2}^{P_n/2} \left(\sum_{|l|\geq 1} |\partial_x \F_1\{\frac{\zeta_{n}}{1-\zeta_{n}}\}(y+lP_n)|\right)^2 \dd y\\
&\lesssim \int_{-P_n/2}^{P_n/2} \left( \sum_{|l|\geq 1} \frac{1}{(P_n^{\frac12}+(l-1)P_n)^{j}}\right)^2 \dd y\\
&\to 0 \quad (P_n\to\infty).
\end{align*}
since $j\geq 2$. Similarly, Lemma~\ref{L.Fi-bounds},\ref{L.Fi-bounds 3} and~\eqref{prop2} yield
\begin{align*}\left|\int_{\RR\setminus(-P_n/2,P_n/2)} (1-\zeta_n)^3\big(\partial_x \F_1\{\frac{\zeta_{n}}{1-\zeta_{n}}\}\big)^2\, \dd x\right|&\leq C(\Norm{\zeta}_{H^\nu})\int_{P_n/2}^\infty \frac1{(x-P_n/2+P_n^{\frac12})^{2j}} \dd x \\
&\to 0 \quad (P_n\to\infty).
\end{align*}
The component $\underline\E$ satisfies the same bounds, thus we proved
\[\E_{P_n}(\widetilde\zeta_{P_n})-\E(\zeta_n)\to 0 \quad (P_n\to\infty).\]
Now by Lemma~\ref{L.diff-E} (which holds in the periodic setting and uniformly with respect to $P>0$) with $\nu$ replaced by some $\nu'\in(1/2, \nu)$ and~\eqref{prop5}, one has
\[ \E_{P_n}(\widetilde\zeta_{P_n})-I_{P_n,q}=\E_{P_n}(\widetilde\zeta_{P_n})-\E_{P_n}(\zeta_{P_n}(\cdot-x_n))\to 0 \quad (P_n\to\infty).\]
Thus we found that 
\[ I_q \leq \E(\zeta_n) =I_{P_n,q}+o(1) \quad (P_n\to\infty).\]
There remains to prove the converse inequality. For any $\epsilon>0$, there exists $\zeta \in V_{q,R}$ such that 
\[ \E(\zeta) \leq I_q+\frac\epsilon3.\]
By the same argument as above, we construct by smoothly truncating and rescaling, $\check\zeta\in V_{q,R}$ such that $\supp\check\zeta\in(-P_\star,P_\star)$, and
\[ \E(\check\zeta) \leq  \E(\zeta) +\frac\epsilon3.\]
Then for $P_n\geq 2 P_\star$, one has $\check\zeta_{P_n}=\sum_{j\in\ZZ}\check\zeta(\cdot+jP_n)\in V_{P,q,R}$ and, as above,
\[\E_{P_n}(\check\zeta_{P_n})-\E(\check\zeta)\to 0 \quad (P_n\to\infty).\]
Combining the above yields, for $P_n$ sufficiently large,
\[ I_{P_n,q}\leq \E_{P_n}(\check\zeta_{P_n})\leq \E(\check\zeta) +\frac\epsilon3\leq I_q+\epsilon.\]
Thus we proved that 
\[\E(\zeta_n) \to I_q \quad (P_n\to\infty).\]
This concludes the proof.
\end{proof}

The following proposition is essential to rule out the ``dichotomy'' scenario in Lions' concentration-compactness principle (see below).
\begin{Proposition}\label{P.subhomogeneity}
There exists $q_0>0$ such that the map $q\mapsto I_q$ is strictly subhomogeneous for $q\in(0,q_0)$:
\[ I_{aq}<aI_q \quad \text{ whenever }\quad 0<q<a q<q_0.\]
\end{Proposition}
\begin{proof}
Let us consider $\zeta_n$ the special minimizing sequence defined in Theorem~\ref{T.special-minimizing-sequence}. We first fix $a_0>1$, and restrict $q_0>0$ if necessary, so that for any $a\in(1,a_0]$ and $q\in(0,q_0)$ such that $aq<q_0$, one has $\Norm{a^{\frac12}\zeta_n}_{H^\nu}^2\leq M a q \leq M q_0<R^2$. 
Thus we have, by definition of $I_{aq}$ and Lemma~\ref{L.E-decomp-order},
\begin{align*}
I_{aq}&\leq \E(a^{\frac12}\zeta_n)=a\E_{2}(\zeta_n)+a^{\frac32}\E_{3}(\zeta_n)+\E_{\rm rem}^{(1)}(a^{\frac12}\zeta_n)\\
&=a\E(\zeta_n) +(a^{\frac32}-a)\E_{3}(\zeta_n)+\E_{\rm rem}^{(1)}(a^{\frac12}\zeta_n)-a\E_{\rm rem}^{(1)}(\zeta_n).
\end{align*}
Moreover, by Theorem~\ref{T.special-minimizing-sequence}, one has
\[\lim_{n\to\infty}\E(\zeta_n)=I_q< q(1-mq^{\frac23}),\]
and Lemma~\ref{L.E-decomp-order} yields
\[-\E_{2}(\zeta_n)\leq -q \quad \text{ and } \quad \E_{\rm rem}^{(1)}(\zeta_n)\lesssim \Norm{\zeta_n}_{H^\nu}^4\lesssim a_0^2 q^2.\] 
It follows that one has for $q\in (0,q_0)$ with $q_0$ sufficiently small and $n$ sufficiently large,
\[
\E_3(\zeta_n)=\E(\zeta_n)-\E_{2}(\zeta_n)-\E_{\rm rem}^{(1)}(\zeta_n)\leq -\frac12m q^{\frac53}.
\]
Thus we find for $n$ sufficiently large,
\begin{equation}\label{a0-est}
I_{aq}\leq a I_q -(a^{\frac32}-a)(m/2)q^{\frac53}+\limsup_{n\to\infty} \big(\E_{\rm rem}^{(1)}(a^{\frac12}\zeta_n)-a\E_{\rm rem}^{(1)}(\zeta_n)\big).
\end{equation}
We now estimate the last contribution, treating separately $\overline{\E}_{\rm rem}^{(1)}$ and $\underline{\E}_{\rm rem}^{(1)}$ in the same spirit as in the proof of Lemma~\ref{L.E-decomp-order}. Consider $\overline{\E}_{\rm rem}^{(1)}$ for instance.
We develop each contribution in $\overline{\E}_{\rm rem}^{(1)}(a^{\frac{1}{2}}\zeta_n)$ using Neumann series in powers of $a^{\frac12}\zeta_n$:
\begin{multline*}
\overline{\E}_{\rm rem}^{(1)}(a^{\frac{1}{2}}\zeta_n)=\int_{\RR} \sum_{k\geq 4} (a^{\frac12}\zeta_n)^k\, \dd x\\
+ \int_{\RR}\sum_{k_1+k_2+k_3\geq 4}c_{k_1,k_2,k_3}  (a^{\frac12}\zeta_n)^{k_1}(\partial_x\F_1\{(a^{\frac12}\zeta_n)^{k_2}\})(\partial_x\F_1\{(a^{\frac12}\zeta_n)^{k_3}\})\, \dd x.
\end{multline*}
The series are absolutely convergent provided $q_0$ is sufficiently small, and start at index $k=4$, as pointed out in the proof of Lemma~\ref{L.E-decomp-order}. We now subtract the contributions of $a\overline{\E}_{\rm rem}^{(1)}(\zeta_n)$ and by the triangle and Cauchy-Schwarz inequalities, 
\begin{multline*}\abs{\overline{\E}_{\rm rem}^{(1)}(a^{\frac{1}{2}}\zeta_n)-a\overline{\E}_{\rm rem}^{(1)}(\zeta_n)}\leq \sum_{k\geq 4}  (a^{\frac{k}2}-a)\Norm{\zeta_n}_{L^\infty}^{k-2}\Norm{\zeta_n}_{L^2}^2\\
+ \sum_{k_1+k_2+k_3\geq 4}|c_{k_1,k_2,k_3}|  (a^{\frac{k_1+k_2+k_3}2}-a) \Norm{\zeta_n}_{L^\infty}^{k_1}\Norm{\partial_x\F_1\{\zeta_n^{k_2}\}}_{L^2}\Norm{\partial_x\F_1\{\zeta_n^{k_3}\}}_{L^2}.
\end{multline*}
 Using that $| a^{\frac{k}2}-a|\leq (a^\frac32-a)( k-2) a^{\frac{k-3}2}$, Lemma~\ref{L.Fi-bounds},\ref{L.Fi-bounds 1}, that $H^\nu$ is a Banach algebra as well as the continuous embedding $H^\nu\subset L^\infty$, we find that one can restrict $q_0>0$ such that the above series is convergent and yields
 \[\abs{\overline{\E}_{\rm rem}^{(1)}(a^{\frac{1}{2}}\zeta_n)-a\overline{\E}_{\rm rem}^{(1)}(\zeta_n)}\leq C(a_0) (a^\frac32-a) q^2,\]
uniformly over $q\in(0,q_0)$ and $a\in (1,a_0]$ such that $aq<q_0$. Plugging this estimate in~\eqref{a0-est} and restricting $q_0$ if necessary, we deduce
\[
I_{aq}<aI_q\ \text{for }0<q<aq<q_0, \ a\in (1,a_0].
\]
Consider now the case when $a\in (1,a_0^p]$ for an integer $p\geq 2$. Then $a^\frac{1}{p}\in (1,a_0]$ and so
\[
I_{aq}=I_{a^\frac{1}{p}a^\frac{p-1}{p}q}
<a^\frac{1}{p}I_{a^\frac{p-1}{p}q}
=a^\frac{1}{p}I_{a^\frac{1}{p}a^\frac{p-2}{p}q}
<a^\frac{2}{p}I_{a^\frac{p-2}{p}q}
<\ldots
<aI_q.
\]
The result is proved.
\end{proof}

By a standard argument, Proposition~\ref{P.subhomogeneity} induces the subadditivity of the map $q\mapsto I_q$.

\begin{Corollary}\label{C.subadditivity}
There exists $q_0>0$ such that the map $q\mapsto I_q$ is strictly subadditive for $q\in(0,q_0)$:
\[ I_{q_1+q_2}< I_{q_1} +I_{q_2}\quad \text{ whenever }\quad 0<q_1<q_1+q_2<q_0.\]
\end{Corollary}

\subsection{Concentration-compactness; proof of Theorem~\ref{T.main-result}}\label{S.concentration-compactness}
We now prove Theorem~\ref{T.main-result}. Let us first recall Lions' concentration compactness principle~\cite{Lions84}.
\begin{Theorem}[Concentration-compactness] \label{T.concentration-compactness}
Any sequence $\{e_n\}_{n\in\NN}\subset L^1(\RR)$ of non-negative functions such that
\[\lim_{n\to \infty}\int_{\RR}e_n\ \dd x= I>0\]
admits a subsequence, denoted again $\{e_n\}_{n\in\NN}$, for which one of the following phenomena occurs.
\begin{itemize}
\item (Vanishing) For each $r>0$, one has
\[\lim_{n\to\infty} \left( \sup_{x\in\RR}\int_{x-r}^{x+r}e_n\ \dd x\right)=0.\]
\item (Dichotomy) There are real sequences $\{x_n\}_{n\in\NN},\{M_n\}_{n\in\NN},\{N_n\}_{n\in\NN}\subset\RR$ and $I^* \in(0,I)$ such that $M_n,N_n\to\infty$, $M_n/N_n\to0$, and
\[\int_{x_n-M_n}^{x_n+M_n}e_n\ \dd x\to I^* \quad \text{ and } \quad  \int_{x_n-N_n}^{x_n+N_n}e_n\ \dd x\to I^*  \]
as $n\to\infty$.
\item (Concentration) There exists a sequence $\{x_n\}_{n\in\NN}\subset\RR$ with the property that for each $\epsilon>0$, there exists $r>0$ with 
\[\int_{x_n-r}^{x_n+r} e_n\ \dd x \geq I-\epsilon\]
for all $n\in\NN$.
\end{itemize}
\end{Theorem}
We shall apply Theorem~\ref{T.concentration-compactness} to 
\begin{align*}
e_n=\gamma\left(\frac{\zeta_n^2}{1-\zeta_n}+\frac13 (1-\zeta_n)^3 \big(\partial_x \F_1\{\frac{\zeta_n}{1-\zeta_n}\}\big)^2\right)
+\frac{\zeta_n^2}{\delta^{-1}+\zeta_n}+\frac13 (\delta^{-1}+\zeta_n)^3 \big(\partial_x \F_2\{\frac{\zeta_n}{\delta^{-1}+\zeta_n}\}\big)^2,
\end{align*}
where $\zeta_n$ is a minimizing sequence of $\E$ over $V_{q,R}$ with $\sup_n \Norm{\zeta_n}_{H^\nu}^2 <R^2$. Such a sequence is known to exist provided that $q\in(0,q_0)$ is sufficiently small, by Theorem~\ref{T.special-minimizing-sequence} (and any minimizing sequence is valid when $\nu=1-\theta$, by Lemma~\ref{L.EvsY}; see Remark~\ref{R.nu=1-theta}). 
The choice of density is inspired by the recent paper~\cite{Arnesen16}, and allows (contrarily to the more evident choice $e_n=\zeta_n^2$) to show, when $\nu=1-\theta$, that the constructed limit satisfies $\E(\eta)=I_q$ and is therefore a solution to the constrained minimization problem~\eqref{min-pb}.
Notice that 
\[
\int_{\RR} e_n\ \dd x =\E(\zeta_n)\to I_q \quad (n\to\infty)
\]
and that there exists a constant $C$ such that
\begin{equation}
\label{J estimate}
\Norm{\zeta_n}_{L^2(J)}^2= \int_J \abs{\zeta_n}^2\ \dd x  \le  C\int_J e_n\ \dd x
\end{equation}
for any interval $J\subseteq \RR$.

We exclude the two first scenarii in Lemmata~\ref{L.vanishing} and~\ref{L.dichotomy}, below. Thus the concentration scenario holds and, using~\eqref{J estimate}, we find that there exists $\{x_n\}_{n\in\NN}\subset\RR$ such that for any $\epsilon>0$, there exists $r>0$ with
\[\Norm{\eta_n}_{L^2(|x|>r)}<\epsilon,\]
where $\{\eta_n\}_{n\in\NN}\eqdef \{\zeta_n(\cdot+x_n)\}_{n\in\NN}$.
Since $\sup_{n\in\NN}\Norm{\eta_n}_{H^\nu(\RR)}<R$, there exists $\eta\in H^\nu(\RR)$ satisfying $\Norm{\eta}_{H^\nu(\RR)}<R$ and $\eta_n\rightharpoonup \eta$ weakly in $H^\nu(\RR)$ (up to the extraction of a subsequence). Increasing $r$ if necessary, we have 
\[\Norm{\eta}_{L^2(|x|>r)}<\epsilon.\]
Now, consider $\chi$ a smooth cut-off function, such that $\chi(x)=1$ for $|x|\leq r$ and $\chi(x)=0$ for $|x|\geq 2r$. One has $\chi\eta_n\rightharpoonup \chi\eta$ weakly in $H^\nu(\RR)$, and by compact embedding~\cite[Corollary~2.96]{BahouriCheminDanchin}, we find that one can extract a subsequence, still denoted $\eta_n$, such that $\Norm{\chi(\eta_n-\eta)}_{L^2(\RR)}\leq \epsilon$ for $n$ sufficiently large. Combining the above estimates, we find that the subsequence satisfies  $\Norm{\eta_n-\eta}_{L^2(\RR)}<3\epsilon$ for $n$ sufficiently large.
 By Cantor's diagonal extraction process, we construct a subsequence satisfying $\Norm{\eta_n-\eta}_{L^2}\to 0$; and by interpolation, $\Norm{\eta_n-\eta}_{H^s}\to 0$ for any $s\in[0,\nu)$.  In particular $(\gamma+\delta)\Norm{\eta}_{L^2}^2=q$, and recall $\Norm{\eta}_{H^\nu}\leq\sup_n \Norm{\zeta_n}_{H^\nu} <R$, thus $\eta\in V_{q,R}$.
If $\nu>1-\theta$, we deduce $\E(\eta_n)\to \E(\eta)$ as $n\to\infty$ by Lemma~\ref{L.diff-E}.
If on the other hand $\nu=1-\theta$ we use the weak lower semi-continuity argument in the proof of Lemma~\ref{L.def-zetaP} to deduce that $I_q\le \E(\eta)\le \lim_{n\to \infty} \E(\eta_n)=I_q$. In either case  
we have that $\E(\eta)=I_q$.

The constructed function $\eta\in H^\nu(\RR)$ is therefore a solution to the constrained minimization problem~\eqref{min-pb}.
In particular, it solves the Euler-Lagrange equation~\eqref{sol-alpha} with $\alpha<0$ provided that $q\in(0,q_0)$ is sufficiently small (proceeding as in Lemma~\ref{L.alpha-bounds}), and therefore satisfies~\eqref{sol-c} with $c^2=(-\alpha)^{-1}>0$.

This proves the first item of Theorem~\ref{T.main-result}, as well as the second item --- except for the strong convergence in $H^\nu(\RR)$ when $\nu=1-\theta>1/2$. This result follows from the fact that weak convergence together with convergence of the norm implies strong convergence in a Hilbert space (applied to $(\gamma^{1/2}(1-\zeta_n)^{3/2} \big(\partial_x \F_1\{\frac{\zeta_n}{1-\zeta_n}\}\big), (\delta^{-1}+\zeta_n)^{3/2} \big(\partial_x \F_2\{\frac{\zeta_n}{\delta^{-1}+\zeta_n}\}\big))\in (L^2(\RR))^2$).

There remains to prove the estimates of the third item. Proceeding as in Lemma~\ref{L.zetaP-H1-estimate}, we find
\[\Norm{\zeta}_{H^\nu}^2\leq M q,\]
 uniformly over the minimizers of $\E$ over $V_{q,R}$. Moreover, by Lemma~\ref{L.E-decomp-order}, one has
\begin{align*}
-\alpha\ q&=-\alpha(\gamma+\delta)\Norm{\zeta}_{L^2}^2=\frac{1}{2}\langle\dd\E(\zeta),\zeta\rangle=\E_2(\zeta)+\frac{3}{2}\E_3(\zeta)+\frac{1}{2}\E_{\rm rem}^{(2)}(\zeta)\\
&=\frac32\E(\zeta)-\frac{1}{2}\E_2(\zeta)-\frac32\E_{\rm rem}^{(1)}(\zeta)+\frac{1}{2}\E_{\rm rem}^{(2)}(\zeta).
\end{align*}
where $\E_{2}(\zeta)\geq q$ and
\[ \abs{\E_{\rm rem}^{(1)}(\zeta)}+\abs{ \E_{\rm rem}^{(2)}(\zeta)} =\O(q^2).\]
Altogether, using that $\E(\zeta)< q(1-mq^{\frac23})$ by Lemma~\ref{L.I-bound-from-below}, we find
\[-\alpha\ q<  \frac32 q(1-mq^{\frac23}) -\frac12 q+\O(q^2)= q\left(1-\frac32mq^{\frac23}\right)+\O(q^2),\]
and the result follows.
Theorem~\ref{T.main-result} is proved.

\begin{Lemma}[Excluding ``vanishing'']\label{L.vanishing}
No subsequence of $\{e_n\}_{n\in\NN}$ has the ``vanishing'' property.
\end{Lemma}
\begin{proof}
By Lemmata~\ref{L.E-decomp-order} and~\ref{L.I-bound-from-below}, one has for $n$ sufficiently large
\[
q(1-mq^{\frac23}) > \E(\zeta_n)=\E_2(\zeta_n) +\E_3(\zeta_n)+\E_{\rm rem}^{(1)}(\zeta_n)\\
\geq q + \E_3(\zeta_n)+\E_{\rm rem}^{(1)}(\zeta_n)
\]
and hence
\[ m q^{\frac53}\leq \abs{\E_3(\zeta_n)}+\abs{\E_{\rm rem}^{(1)}(\zeta_n)} \lesssim\Norm{\zeta_n}_{L^\infty} .\]
On the other hand, one has
\[
\Norm{\zeta_n}_{L^\infty((x-\frac12,x+\frac12))} \leq \Norm{\varphi_x \zeta_n}_{L^\infty(\RR)} \leq \Norm{\varphi_x \zeta_n}_{L^2(\RR)}^{1-\frac{1}{2\nu}}  \Norm{\varphi_x \zeta_n}_{H^\nu(\RR)}^{\frac{1}{2\nu}}    \le C \Norm{\zeta_n}_{L^2((x-1, x+1)}^{1-\frac{1}{2\nu}} \Norm{\zeta_n}_{H^{\nu}(\RR)}^{\frac1{2\nu}},
\]
where $\varphi_x=\varphi(\cdot-x)$ with $\varphi$ a smooth function such that $\varphi=1$ for $|x|\leq 1/2$, $ \varphi=0$ for $|x|\geq 1$, and $0\leq \varphi\leq 1$ otherwise; and using Lemma~\ref{L.interpolation},\ref{L.interpolation 1} and Lemma~\ref{L.product},\ref{L.product 2}.
Since $C$ is independent of $x\in \mathbb R$, this shows that
\[
\Norm{\zeta_n}_{L^\infty} \le CR^\frac1{2\nu} \sup_{x\in \RR} \Norm{\zeta_n}_{L^2((x-1, x+1))}^{1-\frac{1}{2\nu}}.
\]
Hence one has for $n$ sufficiently large
\[
q^\frac{5}{3}\lesssim \sup_{x\in \RR} \Norm{\zeta_n}_{L^2((x-1, x+1))}^{1-\frac{1}{2\nu}},
\]
from which, using~\eqref{J estimate}, it follows that ``vanishing'' cannot occur.
\end{proof}
\begin{Lemma}[Excluding ``dichotomy'']\label{L.dichotomy}
No subsequence of $\{e_n\}_{n\in\NN}$ has the ``dichotomy'' property.
\end{Lemma}
\begin{proof}
We denote by $ \chi\in\C^\infty(\RR^+)$ a non-increasing function with
\begin{equation}
\label{cut-off}
\chi(r)=1 \text{ if } 0\leq r\leq 1\quad \text{ and } \quad  \chi(r)=0 \text{ if }  r\geq 2,
\end{equation}
and such that
\[ \chi=\chi_1^2, \quad 1-\chi=\chi_2^2\]
where $\chi_1$ and $\chi_2$ are smooth. For instance, set $
\chi(r)=1-(1-\widetilde\chi^2(r))^2$ with $\widetilde\chi\in\C^\infty(\RR^+)$ non-increasing and satisfying~\eqref{cut-off}.
Define $\eta_n=\zeta_n(\cdot +x_n)$, and
\[ \eta_n^{(1)}(x)=\eta_n(x)\chi(|x|/M_n) \quad \text{ and } \quad \eta_n^{(2)}(x)=\eta_n(x)\Big(1-\chi(2|x|/N_n) \Big),\]
noting that
\[ \supp(\eta_n^{(1)})\subset [-2M_n,2M_n]\quad \text{ and } \quad  \supp(\eta_n^{(2)})\subset\RR\setminus [-N_n/2,N_n/2].\]
After possibly extracting a subsequence, we can 
assume that
\begin{equation}\label{lim-eta1}
\Norm{\eta_n^{(1)}}_{L^2}^2 \to \frac{q^*}{\gamma+\delta}
\end{equation}
with $q^*\in [0,q]$.
For $n$ sufficiently large, one has $N_n>N_n/2>2M_n>M_n$ and therefore
\begin{align}
 \Norm{\eta_n^{(2)}}_{L^2}^2&=\Norm{\eta_n}_{L^2(|x|>N_n)}^2+\Norm{\eta_n^{(2)}}_{L^2(M_n<|x|<N_n)}^2 \nonumber\\
 &=\frac{q}{\gamma+\delta}-\Norm{\eta_n^{(1)}}_{L^2}^2-\Norm{\eta_n}_{L^2(M_n<|x|<N_n)}^2\nonumber\\
 &\qquad+\Norm{\eta_n^{(1)}}_{L^2(M_n<|x|<N_n)}^2+\Norm{\eta_n^{(2)}}_{L^2(M_n<|x|<N_n)}^2\nonumber\\
 &\to \frac{q-q^*}{\gamma+\delta} \label{lim-eta2}
\end{align}
since $\Norm{\eta_n^{(1)}}_{L^2(M_n<|x|<N_n)}^2+\Norm{\eta_n^{(2)}}_{L^2(M_n<|x|<N_n)}^2\leq 2\Norm{\eta_n}_{L^2(M_n<|x|<N_n)}^2$ and
\[
\Norm{\eta_n}_{L^2(M_n<|x|<N_n)}^2
\le C\int_{M_n<|x-x_n|<N_n} e_n \dd x\to 0
\]
by~\eqref{J estimate} and the assumption of the dichotomy scenario.

We claim that $\E(\eta_n^{(1)})\to I^*$. To show this, note that
\[
\overline{\E}(\eta_n^{(1)})=
\int_\RR \frac{(\eta_n^{(1)})^2}{1- \eta_n^{(1)}}+\frac13 (1-\eta_n^{(1)})^3 \big(\partial_x \F_1\{\frac{\eta_n^{(1)}}{1-\eta_n^{(1)}}\}\big)^2\dd x,
\]
where
\[
\left|\int_\RR \frac{(\eta_n^{(1)})^2}{1- \eta_n^{(1)}}\dd x-
\int_{|x|\le M_n} \frac{\eta_n^2}{1- \eta_n} \dd x \right|
\lesssim \int_{M_n\le |x-x_n|\le N_n} \eta_n^2 \dd x\to 0.
\]
We next find that
\begin{align*}
&\int_\RR (1-\eta_n^{(1)})^3 \big(\partial_x \F_1\{\frac{\eta_n^{(1)}}{1-\eta_n^{(1)}}\}\big)^2
-(1-\eta_n^{(1)})^3 \big(\partial_x \F_1\{\frac{\eta_n^{(1)}}{1-\eta_n}\}\big)^2\dd x\\
&=\int_\RR (1-\eta_n^{(1)})^3 \left(\partial_x \F_1\{\frac{\eta_n^{(1)}}{1-\eta_n^{(1)}}\}+\partial_x \F_1\{\frac{\eta_n^{(1)}}{1-\eta_n}\}\right)
\partial_x \F_1\left\{\frac{\eta_n^{(1)}(\eta_n^{(1)}-\eta_n)}{(1-\eta_n^{(1)})(1-\eta_n)}\right\} \dd x.
\end{align*}
Noting that
\[
\eta_n^{(1)}(\eta_n^{(1)}-\eta_n)=-\chi_1^2(|x|/M_n) \chi_2^2(|x|/M_n) \eta_n^2
\]
we can estimate
\begin{align*}
&\left|\int_\RR (1-\eta_n^{(1)})^3 \left(\partial_x \F_1\{\frac{\eta_n^{(1)}}{1-\eta_n^{(1)}}\}+\partial_x \F_1\{\frac{\eta_n^{(1)}}{1-\eta_n}\}\right)
\partial_x \F_1\left\{\frac{\eta_n^{(1)}(\eta_n^{(1)}-\eta_n)}{(1-\eta_n^{(1)})(1-\eta_n)}\right\} \dd x
\right|
\\
&\lesssim  \Norm{\chi_1(|\cdot|/M_n) \chi_2(|\cdot|/M_n) \eta_n}_{L^\infty}  \Norm{\chi_1(|\cdot|/M_n) \chi_2(|\cdot|/M_n) \eta_n}_{H^\nu}\Norm{\eta_n}_{H^\nu}\\
&\lesssim  \Norm{\chi_1(|\cdot|/M_n) \chi_2(|\cdot|/M_n) \eta_n}_{L^2}^{1-\frac1{2\nu}}  \Norm{\chi_1(|\cdot|/M_n) \chi_2(|\cdot|/M_n) \eta_n}_{H^\nu}^{1+\frac1{2\nu}}\Norm{\eta_n}_{H^\nu}\\
&\lesssim \Norm{\eta_n}_{L^2(M_n\le |x|\le N_n)}^{1-\frac1{2\nu}}\\
&\to 0
\end{align*}
by Lemma~\ref{L.product}~\ref{L.product 1} and Lemma~\ref{L.interpolation},\ref{L.interpolation 1}, and using $\Norm{\eta_n^{(1)}}_{H^\nu}\lesssim \Norm{\eta_n}_{H^\nu} \leq R$. On the other hand,
\begin{align*}
&\left|\int_\RR (1-\eta_n^{(1)})^3 \big(\partial_x \F_1\{\frac{\eta_n^{(1)}}{1-\eta_n}\}\big)^2
-(1-\eta_n^{(1)})^3 \chi^2(|\cdot|/M_n)\big(\partial_x \F_1\{\frac{\eta_n}{1-\eta_n}\}\big)^2\dd x\right|\\
&=\left|\int_\RR (1-\eta_n^{(1)})^3 
\Big( \big(\partial_x \F_1\{\frac{\eta_n^{(1)}}{1-\eta_n}\}\big)
+ \chi(|\cdot|/M_n)\big(\partial_x \F_1\{\frac{\eta_n}{1-\eta_n}\}\big)\Big)
[\partial_x \F_1, \chi(|\cdot|/M_n)] \big(\frac{\eta_n}{1-\eta_n}\big)\dd x\right|\\
&\lesssim M_n^{-1} \Norm{\eta_n}_{H^\nu}^2
\end{align*}
by Lemma~\ref{L.Fi-bounds}~\ref{L.Fi-bounds 2}.
Finally
\[
\int_\RR (1-\eta_n^{(1)})^3 \chi^2(|\cdot|/M_n)\big(\partial_x \F_1\{\frac{\eta_n}{1-\eta_n}\}\big)^2\dd x
=\int_{|x|\le M_n} (1-\eta_n)^3 \big(\partial_x \F_1\{\frac{\eta_n}{1-\eta_n}\}\big)^2\dd x +o(1)
\]
as the remainder term is bounded by a constant times $\int_{M_n \le |x-x_n|\le N_n} e_n \dd x$.
An analogous argument for $\underline{\E}$ reveals that
\[
\E(\eta_n^{(1)})=\int_{|x-x_n|\le M_n} e_n \dd x+o(1)\to I^*
\]
and by similar reasoning one finds that
\[
\E(\eta_n^{(2)})=\int_{|x-x_n|\ge N_n} e_n \dd x+o(1)\to I_q-I^*.
\]

We next claim that $q^*>0$. Indeed, if $q^*=0$, we set
\[
\widetilde\eta_n^{(2)}\eqdef c_n \eta_n^{(2)} , \qquad c_n\eqdef \frac{q^{\frac12}}{(\gamma+\delta)^{\frac12}\Norm{\eta_n^{(2)}}_{L^2}}.
\]
By~\eqref{lim-eta2} and since $q^*=0$, one has $c_n\to 1$. Thus we note
\[
\abs{\E(\widetilde\eta_n^{(2)})-\E(\eta_n^{(2)})}\lesssim \Norm{\widetilde\eta_n^{(2)}-\eta_n^{(2)}}_{H^\nu}\to 0
\]
by Lemma~\ref{L.diff-E} and
\[ 
\limsup_{n\to\infty} \Norm{\widetilde\eta_n^{(2)}}_{H^\nu}<R,
\]
resulting in the contradiction 
$I_q\le \E(\widetilde \eta_n^{(2)}) \to I_q-I^*<I_q$ as $n\to \infty$. We obtain a similar contradiction involving $\eta_n^{(1)}$ and~\eqref{lim-eta1} if we assume that $q^*=q$. Hence, $0<q^*<q$.

In view of the above, we can rescale
\[ \widetilde\eta_n^{(1)}\eqdef \frac{(q^*)^{\frac12}}{(\gamma+\delta)^{\frac12}\Norm{\eta_n^{(1)}}_{L^2}} \eta_n^{(1)} \quad \text{ and } \quad \widetilde\eta_n^{(2)}\eqdef \frac{(q-q^*)^{\frac12}}{(\gamma+\delta)^{\frac12}\Norm{\eta_n^{(2)}}_{L^2}} \eta_n^{(2)} ,\]
so that $(\gamma+\delta)\Norm{\widetilde\eta_n^{(1)}}_{L^2}^2=q$ and $(\gamma+\delta)\Norm{\widetilde\eta_n^{(2)}}_{L^2}^2=q-q_*$ for any $n\in\NN$.
One easily checks that
\[ \limsup_{n\to\infty} \Norm{\widetilde\eta_n^{(1)}}_{H^\nu}<R, \quad \limsup_{n\to\infty} \Norm{\widetilde\eta_n^{(2)}}_{H^\nu}<R\]
and that
\[
\lim_{n\to \infty} (\E(\widetilde \eta_n^{(1)})-\E(\eta_n^{(1)}))
=
\lim_{n\to \infty}
(\E(\widetilde \eta_n^{(2)})-\E(\eta_n^{(2)}))= 0.
\]
Thus we arrive at the following contradiction to Corollary~\ref{C.subadditivity}:
\begin{align*}
I_{q}&< I_{q^*}+I_{q-q^*}\\
&\leq \lim_{n\rightarrow \infty}(\E(\widetilde{\eta}_n^{(1)})+\E(\widetilde{\eta}_n^{(2)}))\\
&=I^*+I_q-I^*\\
&=I_q.
\end{align*}
This concludes the proof of Lemma~\ref{L.dichotomy}.
\end{proof}

\section{Long-wave asymptotics}
In this section we prove that the solutions of~\eqref{sol-c} obtained in Theorem~\ref{T.main-result} are approximated by solutions of the corresponding KdV equation in the long-wave regime, \ie letting $q\to0$ in the constrained minimization problem~\eqref{min-pb}. Indeed, if we introduce the scaling 
\begin{equation}\label{KdV-scaling}\zeta(x)=S_{\rm KdV}(\xi)(x)\eqdef q^\frac{2}{3}\xi(q^\frac{1}{3}x)\end{equation}
 in~\eqref{sol-alpha} and denote $\alpha+1=\alpha_0q^\frac{2}{3}$, then we find that the leading order part of the equation as $q\to0$ is 
\begin{equation}\label{KdV-eq}
\alpha_0(\gamma+\delta)\xi+\frac{3(\gamma-\delta^2)\xi^2}{2}-\frac{(\gamma+\delta^{-1})}{3}\partial_x^2 \xi=0.
\end{equation}
Recall (see \eg~\cite{Angulo-Pava09}) that $\xi\in L^2(\RR)$ satisfying~\eqref{KdV-eq} uniquely defines (up to spatial translation) a solitary-wave solution of the KdV equation, with explicit formula \[\xi_{\rm KdV}(x)=\frac{\alpha_0(\gamma+\delta)}{\delta^2-\gamma}\sech^2\left(\frac{1}{2}\sqrt{\frac{3\alpha_0(\gamma+\delta)}{\gamma+\delta^{-1}}}x\right).\]
Equation~\eqref{KdV-eq} can also be obtained as the Euler-Lagrange equation associated with the minimizer of the scalar functional $\E_{\rm KdV}$ (consistently with Lemma~\ref{L.E-decomp-KdV})
\[
\E_{\rm KdV}(\xi)=\int_{\RR}(\gamma-\delta^2)\xi^3+\frac{(\gamma+\delta^{-1})}{3}(\partial_x\xi)^2\ \dd x,
\]
over the set
\[
U_{1}\eqdef \{\xi\in H^1(\RR)\ :\ (\gamma+\delta)\Norm{\xi}_{L^2}^2=1\}.
\]
Indeed, any minimizer satisfies the Euler-Lagrange equation
\begin{equation}\label{Euler-Lagrange-KdV}
\text{d}\E_{\rm KdV}(\xi)+2(\gamma+\delta)\alpha_0\xi=0,
\end{equation}
which is~\eqref{KdV-eq} with $\alpha_0$ the Lagrange multiplier. Testing the constraint $(\gamma+\delta)\Norm{\xi}_{L^2}^2=1$ with the above explicit formula, we find that 
\begin{equation}\label{def-alpha0}
(\gamma+\delta)\alpha_0=\frac{3}{4}\left(\frac{(\delta^2-\gamma)^4}{(\gamma+\delta)(\gamma+\delta^{-1})}\right)^\frac{1}{3}.
\end{equation}
Additional computations show that
\[
I_{\rm KdV}=\inf\{\E_{\rm KdV}(\xi)\ :\ \xi\in U^1\}=\E_{\rm KdV}(\xi_{\rm KdV})=-\frac35\alpha_0.
\]

We aim at proving that the variational characterization of~\eqref{KdV-eq}, and therefore its explicit solutions, approximate (after suitable rescaling) the corresponding one of~\eqref{sol-c}, namely~\eqref{min-pb}, in the limit $q\to 0$.

\subsection{Refined estimates}

We start by establishing estimates on $\zeta\in D_{q,R}$ the set of minimizers of $\E$ over $V_{q,R}$, as provided by Theorem~\ref{T.main-result}. Here and below, we rely on extra assumptions on the Fourier multipliers, which are assumed to be {\em strongly admissible}, in the sense of Definition~\ref{D.stradmissible}.

\begin{Lemma}\label{L.ellipticity}
There exists $q_0>0$ such that $\zeta\in H^s$ for any $s\geq 0$, and there exists $M_s>0$ such that
\[\Norm{\zeta}_{H^s}^2 \leq M_s \ q\]
uniformly for $q\in(0,q_0)$ and $\zeta\in D_{q,R}$.
\end{Lemma}
\begin{proof}
Once the regularity property $\zeta\in H^s$ has been established, the corresponding estimate is obtained as in the proof of Lemma~\ref{L.zetaP-H1-estimate}, thus we focus only on the regularity issue. This follows from the Euler-Lagrange equation~\eqref{sol-alpha} and elliptic estimates. However, the ellipticity property is not straightforward to ascertain when $\gamma\neq 0$, and we will make use of paradifferential calculus. These tools are recalled in Appendix~\ref{S.para}.

By assumption, one has $\zeta\in H^{\nu}$ with $\nu>1/2$ and $\nu\geq 1-\theta>0$. We fix $\epsilon\in (0,\nu-1/2)$ and $r=\min(1-\theta,\nu-1/2-\epsilon)>0$. We show below that $\zeta\in H^\nu$ satisfying~\eqref{sol-alpha} yields $\zeta\in H^{\nu+r}$, and the argument can be bootstrapped to obtain arbitrarily high regularity, $\zeta\in H^s$, $s\geq 0$.

First we write~\eqref{sol-alpha} as the equality, valid in $H^{-\nu}$,
\begin{align}
&\frac23 h_2^{-2}\partial_x \F_2\big\{h_2^3\partial_x \F_2 \{h_2^{-1}\zeta\}\big\} 
+\frac{2\gamma}3 h_1^{-2}\partial_x \F_1\big\{h_1^3\partial_x \F_1 \{h_1^{-1}\zeta\}\big\} \nonumber \\
&\quad =
2\alpha(\gamma+\delta)\zeta+2\frac{h_1+\gamma h_2}{h_1 h_2}\zeta-\frac{h_1^2-\gamma h_2^2}{h_1^2h_2^2}\abs{\zeta}^2
 +\big(h_2\partial_x \F_2\{h_2^{-1}\zeta\}\big)^2-\gamma\big(h_1\partial_x \F_1\{h_1^{-1}\zeta\}\big)^2 \nonumber \\
 &\quad \eqdef R(\zeta) \label{identity}
 \end{align}
 denoting $h_1=1-\zeta$, $h_2=\delta^{-1}+\zeta$, and recalling $\alpha\in(-3/2,-1/2)$.

Using Lemma~\ref{L.product},~\ref{L.product 2} and Lemma~\ref{L.Fi-bounds},\ref{L.Fi-bounds 1}, one easily checks that $R(\zeta)\in H^{2(\nu-(1-\theta))-1/2-\epsilon}$ in the case $1/2<\nu\le 1/2+(1-\theta)$, and  $R(\zeta)\in H^{\nu-(1-\theta)}$ if $\nu> 1/2+(1-\theta)$. In other words, we find
\begin{equation}\label{RHS}
R(\zeta) \in H^{\nu-2(1-\theta)+r}.
\end{equation}

Above, we used that $\Norm{\zeta}_{L^\infty}<\min(1,\delta^{-1})$ and therefore $h_1(x)^n-1\in H^\nu$ and $h_2(x)^n-(\delta^{-1})^n\in H^\nu$ for any $n\in\ZZ$. This holds as well in the H\"older space $W^{r,\infty}$ since $r\in(0,\nu-1/2)$. In particular, we have 
\[\forall n\in\ZZ, \quad h_1(x)^n \in \Gamma^0_r \qquad \text{ and } \qquad \partial_x\F_i\in \Gamma^{1-\theta}_r,\]
recalling Definition~\ref{D.symbols}.

By Lemma~\ref{Lemma:paralinearization}, we find $\zeta h_1^{-1}-T_{h_1^{-2}}\zeta\in H^{\nu+r}$, and Lemma~\ref{L.Fi-bounds} and Lemma~\ref{L.product} yield
\begin{equation}\label{diff-1}
h_1^{-2}\partial_x\F_1\big\{h_1^3\partial_x\F_1\{\zeta h_1^{-1}-T_{h_1^{-2}}\zeta\}\big\} \in H^{\nu-2(1-\theta)+r}.
\end{equation}
By Lemma~\ref{Lemma:action}, one has $ \partial_x\F_1T_{h_1^{-2}}\zeta=T_{\mathrm ik\F_1(k)}T_{h_1^{-2}}\zeta \in H^{\nu-(1-\theta)}$. We deduce by Lemma~\ref{Lemma:composition} that $\mathrm ik\F_1(k)h_1^{-2}\in \Gamma^{1-\theta}_r$ and
\[ \partial_x\F_1T_{h_1^{-2}}\zeta-T_{\mathrm ik\F_1(k)h_1^{-2}}\zeta\in H^{\nu-(1-\theta)+r},\]
from which we deduce as above
\begin{equation}\label{diff-2}
h_1^{-2}\partial_x\F_1\big\{h_1^3 \big( \partial_x\F_1\{T_{h_1^{-2}}\zeta\}-T_{\mathrm ik\F_1(k)h_1^{-2}}\zeta\big)\big\}  \in H^{\nu-2(1-\theta)+r}.
\end{equation}
Using that $T_{\mathrm ik\F_1(k)h_1^{-2}}\zeta\in   H^{\nu-(1-\theta)}$ and Lemma~\ref{Lemma:paraproduct}, one obtains
\[(h_1^3-T_{h_1^3})T_{\mathrm ik\F_1(k)h_1^{-2}}\zeta \in H^{\nu-(1-\theta)+r}.\]
As above, it follows by Lemma~\ref{L.Fi-bounds} and Lemma~\ref{L.product} that
\begin{equation}\label{diff-3}
 h_1^{-2}\partial_x\F_1\big\{(h_1^3-T_{h_1^3})T_{\mathrm ik\F_1(k)h_1^{-2}}\zeta\big\}\in H^{\nu-2(1-\theta)+r}.
 \end{equation}
We use again Lemma~\ref{Lemma:action} and Lemma~\ref{Lemma:composition} to deduce  that $-(k\F_1(k))^2h_1\in  \Gamma^{2(1-\theta)}_r$ and
\begin{equation}\label{diff-4}
 h_1^{-2}\left(\partial_x\F_1T_{h_1^3} T_{\mathrm ik\F_1(k)h_1^{-2}}\zeta-T_{-(k\F_1(k))^2h_1}\zeta \right)\in H^{\nu-2(1-\theta)+r}.
 \end{equation}
Finally, Lemma~\ref{Lemma:paraproduct} yields
\begin{equation}\label{diff-5}
(h_1^{-2}-T_{h_1^{-2}})T_{-(k\F_1(k))^2h_1}\zeta \in  H^{\nu-2(1-\theta)+r}
\end{equation}
and Lemma~\ref{Lemma:composition} yields
\begin{equation}\label{diff-6}
(T_{h_1^{-2}}T_{-(k\F_1(k))^2h_1}-T_{-(k\F_1(k))^2h_1^{-1}})\zeta\in  H^{\nu-2(1-\theta)+r}.
\end{equation}
Collecting~\eqref{diff-1}--\eqref{diff-6}, we proved
\begin{equation}\label{RHS-composition}  h_1^{-2}\partial_x \F_1\big\{h_1^3\partial_x \F_1 \{h_1^{-1}\zeta\}\big\} -T_{-(k\F_1(k))^2h_1^{-1}}\zeta\in H^{\nu-2(1-\theta) +r}.
\end{equation}

By~\eqref{RHS},~\eqref{RHS-composition} and the corresponding estimate for the second contribution in the left-hand side of~\eqref{identity}, one finds
\[
T_{\frac23h_2^{-1}(\mathrm ik\F_2(k))^2+\frac{2\gamma}3h_1^{-1}(\mathrm ik\F_1(k))^2}\zeta \in H^{\nu-2(1-\theta)+r}.
\]
Moreover, since $\zeta\in H^\nu$, one has $\frac23h_2^{-1}(x)+\frac{2\gamma}3h_1^{-1}(x) \in \Gamma^0_r$ and therefore
\[T_{\frac23h_2^{-1}(x)+\frac{2\gamma}3h_1^{-1}(x)}\zeta\in H^\nu\subset H^{\nu-2(1-\theta)+r}.\]
Adding the two terms yields
\begin{equation}\label{RHS-paraproduct} 
T_{a(x,k)}\zeta \in H^{\nu-2(1-\theta)+r}
\end{equation}
with
\[ a(x,k)\eqdef \frac23h_2^{-1}(x)\big(1+(k\F_2(k))^2\big)+\frac{2\gamma}3h_1^{-1}(x)\big(1+(k\F_1(k))^2\big).\]
Notice that 
\[a(x, k)\in \Gamma^{2(1-\theta)}_r\quad \text{ and }\quad a(x, k)^{-1}\in \Gamma^{-2(1-\theta)}_r.\]
In particular, Lemma~\ref{Lemma:action} and~\eqref{RHS-paraproduct} yield
\[T_{a(x,k)^{-1}}T_{a(x,k)}\zeta \in H^{\nu+r}.\]
Additionally, by Lemma~\ref{Lemma:composition}, we have
\[\zeta-T_{a(x,k)^{-1}}T_{a(x,k)}\zeta=T_{a(x,k)^{-1}a(x,k)}\zeta-T_{a(x,k)^{-1}}T_{a(x,k)}\zeta\in H^{\nu+r}.\]
Adding the two terms shows that $\zeta\in H^{ \nu+r}$, which concludes the proof.
\end{proof}

\begin{Remark}\label{R.one-layer}
In the one-layer situation, namely $\gamma=0$, the use of paradifferential calculus is not necessary, and Lemma~\ref{L.ellipticity} can be obtained through a direct use of Lemmata~\ref{L.composition} and~\ref{L.product}. In particular, Lemma~\ref{L.ellipticity}  and subsequent results hold for (non-necessarily strongly) admissible Fourier multipliers, in the sense of Definition~\ref{D.admissible}.
\end{Remark}

The following lemma shows that the minimizers of $\E$ over $V_{q,R}$, as provided by Theorem~\ref{T.main-result}, scale as~\eqref{KdV-scaling}.
\begin{Lemma}\label{L.zeta-Linfty-bound}
There exists $q_0>0$ and $C>0$ such that the estimates
\begin{align}
\Norm{\zeta}_{L^\infty}&\leq Cq^\frac{2}{3}\label{L-infty-est},\\
\Norm{\partial_x\zeta}_{L^2}^2&\leq Cq^\frac{5}{3}\label{derivative-est},\\
\Norm{\partial_{x}^2\zeta}_{L^2}^2&\leq Cq^\frac{7}{3}\label{second-derivative-est}
\end{align}
hold uniformly for $q\in(0,q_0)$ and $\zeta\in D_{q,R}$, the set of minimizers of $\E$ over $V_{q,R}$.
\end{Lemma}
\begin{proof}
Let $\zeta$ be minimizer over $V_{q,R}$. Since $2(\gamma+\delta)\alpha\zeta+\dd \E(\zeta)=0$, we get from Lemma~\ref{L.E-decomp-order} that
\begin{equation}\label{E-L-rem} 2\alpha(\gamma+\delta)\zeta +\dd \E_2(\zeta) = 2\alpha(\gamma+\delta)\zeta+\dd \E(\zeta)-\dd\E_3(\zeta)-\dd \E_{\rm rem}^{(1)}(\zeta)=-\dd\E_3(\zeta)-\dd \E_{\rm rem}^{(1)}(\zeta),
\end{equation}
where
\[
\dd\E_3(\zeta)=\gamma\dd\overline{\E}_3(\zeta)+\dd\underline{\E}_3(\zeta),
\]
and
\begin{align*}
\dd\overline{\E}_3(\zeta)&=3\zeta^2-\big(\partial_x\F_1\{\zeta\}\big)^2+2\partial_x\F_1\{\zeta\partial_x\F_1\{\zeta\}\}-\frac{2}{3}\partial_x\F_1\{\partial_x\F_1\{\zeta^2\}\}-\frac{4}{3}\zeta\partial_x\F_1\{\partial_x\F_1\{\zeta\}\},\\
\dd\underline{\E}_3(\zeta)&=-3\delta^2\zeta^2+\big(\partial_x\F_2\{\zeta\}\big)^2-2\partial_x\F_2\{\zeta\partial\F_2\{\zeta\}\}+\frac{2}{3}\partial_x\F_2\{\partial_x\F_2\{\zeta^2\}\}+\frac{4}{3}\partial_x\F_2\{\partial_x\F_2\{\zeta\}\}.
\end{align*}
We also have that
\[
\dd\E_2(\zeta)=\gamma\dd\overline{\E}_2(\zeta)+\dd\underline{\E}_2(\zeta)=2(\gamma+\delta)\zeta-\frac{2}{3}\left(\gamma\partial_x\F_1\{\partial_x\F_1\{\zeta\}\}+\delta^{-1}\partial_x\F_2\{\partial_x\F_2\{\zeta\}\}\right).
\]
In frequency space equation~\eqref{E-L-rem} becomes
\[
 2\big((\gamma+\delta)\alpha+\gamma+\delta+\frac13(\gamma(k\F_1(k))^2+\delta^{-1}(k\F_2(k))^2)\big)\widehat\zeta(k)=-\mathcal{F}\Big(\dd\E_3(\zeta)+\dd \E_{\rm rem}^{(1)}(\zeta)\Big)(k).
\]
By using the estimate for $\alpha$ in Theorem~\ref{T.main-result}, we deduce
\begin{equation}\label{freq-est}
\abs{\widehat\zeta(k)}\leq\frac12 \frac{\left|\mathcal{F}\Big(\dd\E_3(\zeta)+\dd \E_{\rm rem}^{(1)}(\zeta)\Big)(k)\right|}{mq^{\frac23}+\frac13(\gamma(k\F_1(k))^2+\delta^{-1}(k\F_2(k))^2)}.
\end{equation}
The estimates follow from~\eqref{freq-est} and a suitable decomposition into high- and low-frequency components. In order to estimate the right-hand-side, we heavily make use of Lemma~\ref{L.ellipticity}: $\Norm{\zeta}_{H^n}^{2}\lesssim q$ for all $n\in\mathbb{N}$. This will be used again throughout the proof without reference.

We first deduce from Lemma~\ref{L.product} that
\[
\Norm{\mathcal{F}\big(\dd\E_3(\zeta)\big)}_{L^1}\lesssim \Norm{\dd\E_3(\zeta)}_{H^1}\lesssim q
\]
and
\[
\Norm{\mathcal{F}\big(\dd\E_3(\zeta)\big)}_{L^\infty}\lesssim  \Norm{\zeta}_{L^2}^2+\Norm{\partial_x\zeta}_{L^2}^2+\Norm{\zeta}_{L^2}\Norm{\partial_x^2\zeta}_{L^2}\lesssim q\]
and, similarly,
\[
\Norm{\mathcal{F}\big(\dd \E_{\rm rem}^{(1)}(\zeta)\big)}_{L^1}+\Norm{\mathcal{F}\big(\dd \E_{\rm rem}^{(1)}(\zeta)\big)}_{L^\infty}\lesssim q^\frac{3}{2}.
\]
By the definition of admissible Fourier multipliers in~\eqref{D.admissible}, there exists $c_0,k_0>0$ such that
\[\forall k\in \RR\setminus[-k_0,k_0], \quad |k|\F(k) \geq c_0.\]
We also assume that $\F(k)>0$, and therefore there exists $c_0'>0$ such that
\[\forall k\in [-k_0,k_0], \quad \F(k) \geq c_0'.\]
As a consequence, we have
\[\sup_{k\in\RR\setminus [-k_0,k_0]} \frac1{mq^{\frac23}+\frac13(\gamma(k\F_1(k))^2+\delta^{-1}(k\F_2(k))^2)}\lesssim 1\]
and
\[\int_{-k_0}^{k_0}\frac1{mq^{\frac23}+\frac13(\gamma(k\F_1(k))^2+\delta^{-1}(k\F_2(k))^2)}\dd k\lesssim q^{-\frac13}.\]

Now, we decompose
\begin{equation} \label{bound-zeta-Linfty}
\Norm{\zeta}_{L^\infty}\leq \frac1{\sqrt{2\pi}}\Norm{\widehat\zeta}_{L^1}\leq \frac12\int_{\RR}\frac1{ mq^{\frac23}+\frac13(\gamma(k\F_1(k))^2+\delta^{-1}(k\F_2(k))^2)}\left|\mathcal{F}\Big(\dd\E_3(\zeta)+\dd \E_{\rm rem}^{(1)}(\zeta)\Big)(k)\right|\dd k.
\end{equation}
into high- and low-frequency components and estimate each part.
For the low frequency part we have
\begin{multline*}
\int_{-k_0}^{k_0}\frac1{mq^{\frac23}+\frac13(\gamma(k\F_1(k))^2+\delta^{-1}(k\F_2(k))^2)}\left|\mathcal{F}\Big(\dd\E_3(\zeta)+\dd \E_{\rm rem}^{(1)}(\zeta)\Big)(k)\right|\dd k\\
\leq \Norm{\mathcal{F}\Big(\dd\E_3(\zeta)+\dd \E_{\rm rem}^{(1)}(\zeta)\Big)(k)}_{L^\infty} \times \int_{-k_0}^{k_0}\frac1{mq^{\frac23}+\frac13(\gamma(k\F_1(k))^2+\delta^{-1}(k\F_2(k))^2)}\dd k\lesssim q^{\frac23}.
\end{multline*}
For the high-frequency part
\begin{align*}
&\int_{\RR\setminus [-k_0,k_0]}\frac1{mq^{\frac23}+\frac13(\gamma(k\F_1(k))^2+\delta^{-1}(k\F_2(k))^2)}\left|\mathcal{F}\Big(\dd\E_3(\zeta)+\dd \E_{\rm rem}(\zeta)\Big)(k)\right|\dd k\\
&\leq \Norm{\mathcal{F}\Big(\dd\E_3(\zeta)+\dd \E_{\rm rem}(\zeta)\Big)(k)}_{L^1} \sup_{k\in\RR\setminus [-k_0,k_0]} \frac1{mq^{\frac23}+\frac13(\gamma(k\F_1(k))^2+\delta^{-1}(k\F_2(k))^2)}\lesssim q.
\end{align*}
Combining the above estimates in~\eqref{bound-zeta-Linfty} gives us the inequality~\eqref{L-infty-est}.

Let us now turn to~\eqref{derivative-est}. By~\eqref{freq-est} we have
\begin{equation}\label{bound-dxzeta}
\Norm{\partial_x\zeta}_{L^2}^2\leq\frac{1}{4}\int_\RR\frac{k^{2}\abs{\mathcal{F}(\dd\E_3(\zeta)+\dd\E_{\rm rem}^{(1)}(\zeta))}^2}{(mq^\frac{2}{3}+\frac13(\gamma(k\F_1(k))^2+\delta^{-1}(k\F_2(k))^2))^2}\dd k
\end{equation}
We estimate the low-frequency part as above:
\begin{multline*}
\int_{-k_0}^{k_0}\frac{k^2\abs{\mathcal{F}(\dd\E_3(\zeta)+\dd\E_{\rm rem}^{(1)}(\zeta))(k)}^2}{(mq^\frac{2}{3}+\frac13(\gamma(k\F_1(k))^2+\delta^{-1}(k\F_2(k))^2))^2}\ dk\\
\leq \Norm{\mathcal{F}(\dd\E_3(\zeta)+\dd\E_{\rm rem}^{(1)}(\zeta))(k)}_{L^\infty}^2 \times \int_{-k_0}^{k_0}\frac{k^2}{(mq^\frac{2}{3}+\frac13(\gamma(k\F_1(k))^2+\delta^{-1}(k\F_2(k))^2))^2}\ \dd k\lesssim q^{\frac53}.
\end{multline*}
As for the high frequency part, we notice that
\[ \Norm{k\mathcal{F}(\dd\E_{\rm rem}^{(1)}(\zeta))}_{L^2}\lesssim q^{\frac32}\]
and
\begin{align*}
\Norm{k\mathcal{F}(\dd\E_3(\zeta))}_{L^2}&\lesssim \Norm{\zeta}_{L^\infty}\Norm{\partial_x\zeta}_{L^2}+\Norm{\partial_x\zeta}_{L^2}\Norm{\partial_x^2\zeta}_{H^1}+\Norm{\zeta}_{L^\infty}\Norm{\partial_x^3\zeta}_{L^2}\\
&\lesssim q^\frac{7}{6}+q^\frac{1}{2}\Norm{\partial_x\zeta}_{L^2},
\end{align*}
where we used~\eqref{L-infty-est}. It follows that
\begin{multline*}
\int_{\RR\setminus[-k_0,k_0]}\frac{k^2\abs{\mathcal{F}(\dd\E_3(\zeta)+\dd\E_{\rm rem}^{(1)}(\zeta))(k)}^2}{(cq^\frac{2}{3}+\frac13(\gamma(k\F_1(k))^2+\delta^{-1}(k\F_2(k))^2))^2}\ \dd k\\
\leq \Norm{k\mathcal{F}(\dd\E_3(\zeta)+\dd\E_{\rm rem}^{(1)}(\zeta))}_{L^2}^2\times \sup_{k\in\RR\setminus [-k_0,k_0]} \frac{1}{(mq^\frac{2}{3}+\frac13(\gamma(k\F_1(k))^2+\delta^{-1}(k\F_2(k))^2))^2}\lesssim q^\frac{7}{3}+q\Norm{\partial_x\zeta}_{L^2}^2 .
\end{multline*}
Combining the high and low frequency estimates into~\eqref{bound-dxzeta} gives us
\[
\Norm{\partial_x\zeta}_{L^2}^2\lesssim q^\frac{5}{3}+q\Norm{\partial_x\zeta}_{L^2}^2,
\]
hence, for $q_0$ sufficiently small we get~\eqref{derivative-est}.

We conclude with the proof of~\eqref{second-derivative-est}. By~\eqref{freq-est} we have
\begin{equation}\label{bound-dx2zeta}
\Norm{\partial_x^2\zeta}_{L^2}^2\leq\frac{1}{4}\int_\RR\frac{k^{4}\abs{\mathcal{F}(\dd\E_3(\zeta)+\dd\E_{\rm rem}^{(1)}(\zeta))}^2}{(mq^\frac{2}{3}+\frac13(\gamma(k\F_1(k))^2+\delta^{-1}(k\F_2(k))^2))^2}\dd k.
\end{equation}
Now we remark that by the Cauchy-Schwarz inequality and integration by parts, one has 
\[\forall j\geq 2, \qquad \Norm{\partial_x^j \zeta}_{L^2}^2\leq \Norm{\partial_x\zeta}_{L^2}\Norm{\partial_x^{2j-1}\zeta}_{L^2}\lesssim q^{\frac43}.\]
Since $\Norm{\zeta}_{L^\infty}$ and $\Norm{\partial_x\zeta}_{L^2}$ satisfy similar estimates by~\eqref{L-infty-est} and~\eqref{derivative-est}, we have
\begin{align*}
\Norm{k^2\mathcal{F}(\dd\E_3(\zeta))}_{L^2}&\lesssim \Norm{\zeta}_{L^\infty}\Norm{\partial_x^2\zeta}_{L^2}+\Norm{\partial_x\zeta}_{H^1}\Norm{\partial_x\zeta}_{L^2}+\Norm{\zeta}_{L^\infty}\Norm{\partial_x^4\zeta}_{L^2}+ \Norm{\partial_x\zeta}_{H^1}\Norm{\partial_x^3\zeta}_{L^2}\\
&\qquad+\Norm{\partial_x^2\zeta}_{H^1}\Norm{\partial_x^2\zeta}_{L^2}\\
&\lesssim q^{\frac43},
\end{align*}
and
\begin{align*}
\Norm{k\mathcal{F}(\dd\E_3(\zeta))}_{L^\infty}&\lesssim \Norm{\zeta}_{L^2}\Norm{\partial_x\zeta}_{L^2}+\Norm{\partial_x\zeta}_{L^2}\Norm{\partial_x^2\zeta}_{L^2}+\Norm{\zeta}_{L^2}\Norm{\partial_x^3\zeta}_{L^2}\\
&\lesssim q^\frac{8}{6}+q^\frac{1}{2}\Norm{\partial_x^3\zeta}_{L^2}\\
&\lesssim q^\frac{8}{6}+q^\frac{1}{2}\Norm{\partial_x^2\zeta}_{L^2}^{\frac12}\Norm{\partial_x^4\zeta}_{L^2}^{\frac12}\\
&\lesssim q^\frac{8}{6}+q^\frac{5}{6}\Norm{\partial_x^2\zeta}_{L^2}^{\frac12}\\
&\lesssim q^\frac{8}{6}+q^\frac{1}{3}\Norm{\partial_x^2\zeta}_{L^2},
\end{align*}
in addition to
\[ \Norm{k\mathcal{F}(\E_{\rm rem}^{(1)}(\zeta))}_{L^\infty}+\Norm{k^2\mathcal{F}(\E_{\rm rem}^{(1)}(\zeta))}_{L^2}\lesssim q^{\frac32}.\]

Thus, proceeding as above, we find
\begin{multline*}
\int_{-k_0}^{k_0}\frac{k^4\abs{\mathcal{F}(\dd\E_3(\zeta)+\dd\E_{\rm rem}^{(1)}(\zeta))(k)}^2}{(mq^\frac{2}{3}+\frac13(\gamma(k\F_1(k))^2+\delta^{-1}(k\F_2(k))^2))^2}\ dk\\
\leq \Norm{k\mathcal{F}(\dd\E_3(\zeta)+\dd\E_{\rm rem}^{(1)}(\zeta))(k)}_{L^\infty}^2 \times \int_{-k_0}^{k_0}\frac{k^2}{(mq^\frac{2}{3}+\frac13(\gamma(k\F_1(k))^2+\delta^{-1}(k\F_2(k))^2))^2}\ \dd k\lesssim q^{\frac73}+q^\frac{1}{3}\Norm{\partial_x^2\zeta}_{L^2}^2
\end{multline*}
and
\begin{multline*}
\int_{\RR\setminus[-k_0,k_0]}\frac{k^4\abs{\mathcal{F}(\dd\E_3(\zeta)+\dd\E_{\rm rem}^{(1)}(\zeta))(k)}^2}{(mq^\frac{2}{3}+\frac13(\gamma(k\F_1(k))^2+\delta^{-1}(k\F_2(k))^2))^2}\ \dd k\\
\leq \Norm{k^2\mathcal{F}(\dd\E_3(\zeta)+\dd\E_{\rm rem}^{(1)}(\zeta))}_{L^2}^2\times \sup_{k\in\RR\setminus [-k_0,k_0]} \frac{1}{(mq^\frac{2}{3}+\frac13(\gamma(k\F_1(k))^2+\delta^{-1}(k\F_2(k))^2))^2}\lesssim q^\frac{8}{3} .
\end{multline*}
Plugging these estimates into~\eqref{bound-dx2zeta} and restricting $q\in(0,q_0)$ if necessary yields~\eqref{second-derivative-est}, and the proof is complete.
\end{proof}

\subsection{Convergence results; proof of Theorem~\ref{T.asymptotic-result}}
We are now in position to relate the minimizers of $\E$ in $D_{q,R}$ with the corresponding solution of the {\rm KdV} equation. We first compare 
\[ I_{\rm KdV}=\inf\{\E_{\rm KdV}(\xi)\ :\ \xi\in U^1\} \quad \text{ and } \quad I_q= \inf\{\E(\zeta)\ : \ \zeta\in V_{q,R}\}.\]
\begin{Lemma}\label{L.Iq vs IKdV}
There exists $q_0>0$ such that the quantities $I_q$ and $I_{\rm KdV}$ satisfy
\begin{align}
I_q&=q+\E_{\rm KdV}(\zeta)+\mathcal{O}(q^2),\ \text{uniformly over minimizers of }\E \text{ in }V_{q,R},\label{inf-1}\\
I_q&=q+q^\frac{5}{3}I_{\rm KdV}+\mathcal{O}(q^2)=q-\frac35 \alpha_0 q^\frac{5}{3}+\mathcal{O}(q^2).\label{inf-2}
\end{align}	
uniformly over $q\in(0,q_0)$.
\end{Lemma}
\begin{proof}
Recall that, from Lemma~\ref{L.E-decomp-KdV}, one has for any $\zeta\in H^2$, 
\begin{equation}\label{decomp-KdV-eq}
\E(\zeta)=(\gamma+\delta)\Norm{\zeta}_{L^2}^2+\E_{\rm KdV}(\zeta)+\E_{\rm rem}(\zeta),
\end{equation}
with
\[
\abs{\E_{\rm rem}(\zeta)}\leq C(h_0^{-1},\Norm{\zeta}_{H^1})(\Norm{\zeta}_{L^\infty}^2\Norm{\zeta}_{L^2}^2+\Norm{\zeta}_{L^\infty}\Norm{\partial_x\zeta}_{L^2}^2+\Norm{\partial_x^2\zeta}_{L^2}\Norm{\partial_x\zeta}_{L^2}).
\]
Let $\zeta$ be a minimizer of $\E$  in $V_{q,R}$ and note that $\zeta \in H^2$ by Lemma~\ref{L.ellipticity}. Using Lemma~\ref{L.zeta-Linfty-bound} and~\eqref{decomp-KdV-eq}, we obtain
\[
I_q=\E(\zeta)=q+\E_{\rm KdV}(\zeta)+\mathcal{O}(q^2).
\]
Introducing $\xi=S_{\rm KdV}^{-1}(\zeta)$, we find that $\xi\in U_1$ and
\[
\E_{\rm KdV}(\zeta)=q^\frac{5}{3}\E_{\rm KdV}(\xi)\geq q^\frac{5}{3}I_{\rm KdV}.
\]
Thus we found
\[ I_q\geq q+q^\frac{5}{3}I_{\rm KdV}+\mathcal{O}(q^2).\]
Similarly, notice that $\widetilde\zeta=S_{\rm KdV}(\xi_{\rm KdV})$ satisfies $\widetilde\zeta\in V_{q,R}$ (for $q$ sufficiently small) and, by~\eqref{decomp-KdV-eq}
\[I_q\leq \E(\widetilde\zeta) = q +\E_{\rm KdV}(\widetilde\zeta)+\mathcal{O}(q^2).\]
Since $\E_{\rm KdV}(\widetilde\zeta)=q^\frac{5}{3}\E_{\rm KdV}(\xi_{\rm KdV})=q^\frac{5}{3}I_{\rm KdV}$, we deduce
\[ I_q\leq q+q^\frac{5}{3}I_{\rm KdV}+\mathcal{O}(q^2).\]
We have thus proved~\eqref{inf-2}.
\end{proof}
This next result is the first part of Theorem~\ref{T.asymptotic-result}, which relates the minimizers of $\E$ in $V_{q,R}$ with the minimizers of $\E_{\rm KdV}$ in $U_1$. 
\begin{Theorem}\label{T.convergence-minimizers}
Let $q_0>0$ be such that Theorem~\ref{T.main-result} and Lemma~\ref{L.Iq vs IKdV} hold.
Then for any $q\in(0,q_0)$ and $\zeta\in D_{q,R}$, there exists $x_\zeta\in \RR$ such that
\[
\Norm{q^{-\frac23}\zeta(q^{-\frac13}\cdot)-\xi_{\rm KdV}(\cdot-x_\zeta)}_{H^1}\lesssim q^{\frac16}\ ,
\]
uniformly with respect to $q\in (0,q_0)$ and $\zeta\in D_{q,R}$.
\end{Theorem}
\begin{proof}
Assume that there exists $\epsilon>0$ and a sequence $\zeta_n\in D_{q_n,R}$ with $q_n\searrow 0$ such that
\begin{equation} \label{contradiction}
\forall n\in\NN, \quad \inf_{x_0\in\RR} \Norm{q_n^{-\frac23}\zeta_n(q_n^{-\frac13}\cdot)-\xi_{\rm KdV}(\cdot-x_0)}_{H^1}\geq \epsilon.
\end{equation}
Denote for simplicity $\xi_n(x)=q_n^{-\frac23}\zeta_n(q_n^{-\frac13}x)$. From~\eqref{inf-1} in Lemma~\ref{L.Iq vs IKdV}, we have
\[ I_{q_n}=\E(\zeta_n)=q_n+\E_{\rm KdV}(\zeta_n)+\O(q_n^2)=q_n+q_n^{\frac53}\E_{\rm KdV}(\xi_n)+\O(q_n^2).\]
By~\eqref{inf-2} in Lemma~\ref{L.Iq vs IKdV}, we deduce that 
\[\E_{\rm KdV}(\xi_n)-  I_{\rm KdV} = \O(q_n^{\frac13}) .\]
In particular $\{\xi_n\}_{n\in\NN}$ is a minimizing sequence for $\E_{\rm KdV}$ satisfying the constraint $(\gamma+\delta)\Norm{\xi_n}_{L^2}^2=1$. It follows~\cite{Angulo-Pava09} that there exists a sequence $\{x_n\}_{n\in\NN}$ such that \[\Norm{\xi_n(\cdot-x_n)-\xi_{\rm KdV}}_{H^1}\to 0,\]
which contradicts~\eqref{contradiction}.

The quantitative estimate follows from the argument in~\cite{BonaSouganidisStrauss87}. From the above, we may apply Lemma 4.1 therein and define uniquely $x_\zeta$ such that $\langle\xi,\xi_{\rm KdV}(\cdot-x_\zeta)\rangle=0$, where we denote $\xi=q^{-\frac23}\zeta(q^{-\frac13}\cdot)$. Following the above estimates and~\cite[Lemma 5.2]{BonaSouganidisStrauss87}, we find
\[\Norm{\xi-\xi_{\rm KdV}(\cdot-x_\zeta)}_{H^1}^2\lesssim \E_{\rm KdV}(\xi)-\E_{\rm KdV}(\xi_{\rm KdV})\lesssim q^{\frac13} .\]
This concludes the proof.
\end{proof}
Next we prove the second part of Theorem~\ref{T.asymptotic-result}, which relates the Lagrange multipliers $\alpha$ with the one of the {\rm KdV} equation, $\alpha_0$.
\begin{Theorem}\label{L-multipliers}
The number $\alpha$, defined in Theorem~\ref{T.main-result}, satisfies
\[
\alpha+1=q^\frac{2}{3}\alpha_0+\mathcal{O}(q^\frac{5}{6}),
\]
uniformly over $D_{q,R}$. 
\end{Theorem}
\begin{proof}
By Lemma~\ref{L.E-decomp-KdV}, we have 
 \[\langle \dd\E(\zeta),\zeta\rangle=2(\gamma+\delta)\int_{\RR}\zeta^2\ \dd x+ \langle \dd\E_{\rm KdV}(\zeta),\zeta\rangle +\langle \dd\E_{\rm rem}(\zeta),\zeta\rangle\]
where, using Lemma~\ref{L.zeta-Linfty-bound}, one has
\[
\langle \dd\E_{\rm rem}(\zeta),\zeta\rangle\lesssim q^{\frac73},
\]
uniformly for minimizers of $\E$ in $V_{q,R}$, and so
\[
\langle \dd\E(\zeta),\zeta\rangle=2q+q^\frac{5}{3}\langle \dd\E_{\rm KdV}(S_{\rm KdV}^{-1}(\zeta)),S_{\rm KdV}^{-1}(\zeta)\rangle+\mathcal{O}(q^{\frac73}).
\]
By Theorem~\ref{T.convergence-minimizers} there exists $x_\zeta$ such that $\Norm{S_{\rm KdV}^{-1}(\zeta)-\xi_{\rm KdV}(\cdot-x_\zeta)}_{H^1}=\mathcal{O}(q^\frac16)$ as $q\searrow 0$. This implies that
\[
\langle \dd\E_{\rm KdV}(S_{\rm KdV}^{-1}(\zeta)),S_{\rm KdV}^{-1}\zeta\rangle-\langle \dd\E_{\rm KdV}(\xi_{\rm KdV}),\xi_{\rm KdV}\rangle=\mathcal{O}(q^\frac16)\ \text{as }q\searrow 0
\]
and therefore
\[\langle \dd\E(\zeta),\zeta\rangle=2q+q^\frac{5}{3}\langle \dd\E_{\rm KdV}(\xi_{\rm KdV}),\xi_{\rm KdV}\rangle+\mathcal{O}(q^\frac{11}6).\]
Now recall the Euler-Lagrange equations~\eqref{sol-alpha} and~\eqref{Euler-Lagrange-KdV}, which yield immediately
\begin{align*}
2\alpha(\zeta)\ q&=-\langle \dd\E(\zeta),\zeta\rangle ,\\
2\alpha_0&=-\langle \dd\E_{\rm KdV}(\xi_{\rm KdV}),\xi_{\rm KdV}\rangle,
\end{align*}
and the result follows.
\end{proof}

\section{Numerical study}

In this section, we provide numerical illustrations of our results as well as some numerical experiments for situations which are not covered by our results. We first describe our numerical scheme, before discussing the outcome of these simulations.

\paragraph{Description of the numerical scheme}

Our numerical scheme computes solutions for~\eqref{sol-c} for given value of $c$ (and hence does not not follow the minimization strategy developed in this work). Because we seek smooth localized solutions and our operators involve Fourier multipliers, it is very natural to discretize the problem through spectral methods~\cite{Trefethen}. We are thus left with the problem of finding a root for a nonlinear function defined in a finite (but large) dimensional space. To this aim, we employ the Matlab script \texttt {fsolve} which implements the so-called trust-region dogleg algorithm~\cite{ConnGouldToint00} based on Newton's method. For an efficient and successful outcome of the method, it is important to have a fairly precise initial guess. To this aim, we use the exact solution of the Green-Naghdi model, which is either explicit (in the one-layer situation~\cite{Serre53a}) or obtained as the solution of an ordinary differential equation (in the bi-layer situation~\cite{Miyata87,ChoiCamassa99}) that we solve numerically. Our solutions are compared with the corresponding ones of the full Euler system. To compute the latter, we use the Matlab script developed by Per-Olav Rus\aa{}s and documented in~\cite{GrueJensenRusaasEtAl99} in the bilayer configuration while in the one-layer case, the Matlab script of Clamond and Dutykh~\cite{ClamondDutykh13} offer faster and more accurate results (although limited to relatively small velocities).

\paragraph{Two-layer setting}

\begin{figure}[!htbp]
 \subfigure[Small velocity, $c=1.005$.]{
\includegraphics[width=.5\textwidth]{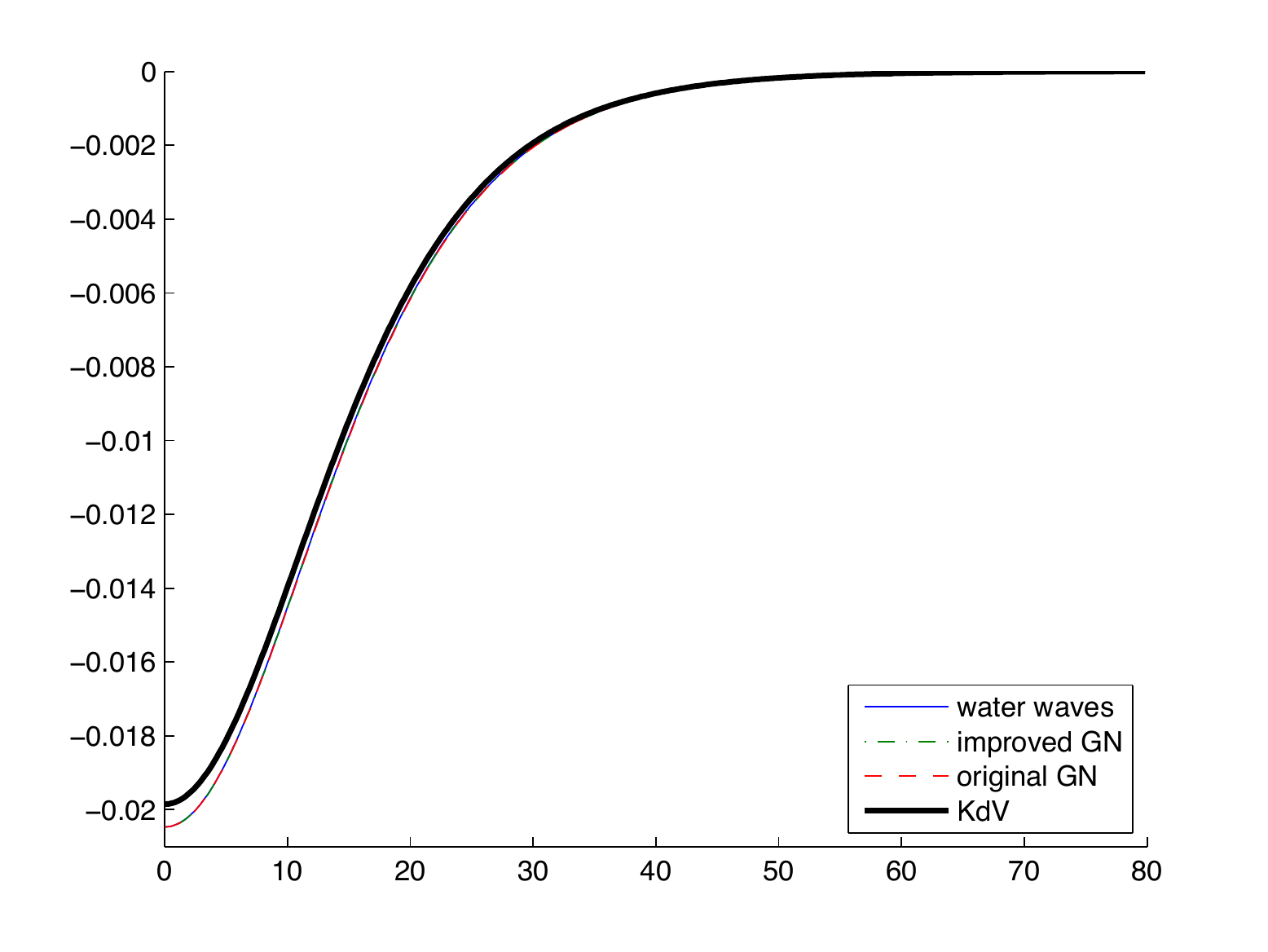}
\label{F.run2-smallq}
}%
\subfigure[Large velocity, $c=1.06065$.]{
\includegraphics[width=.5\textwidth]{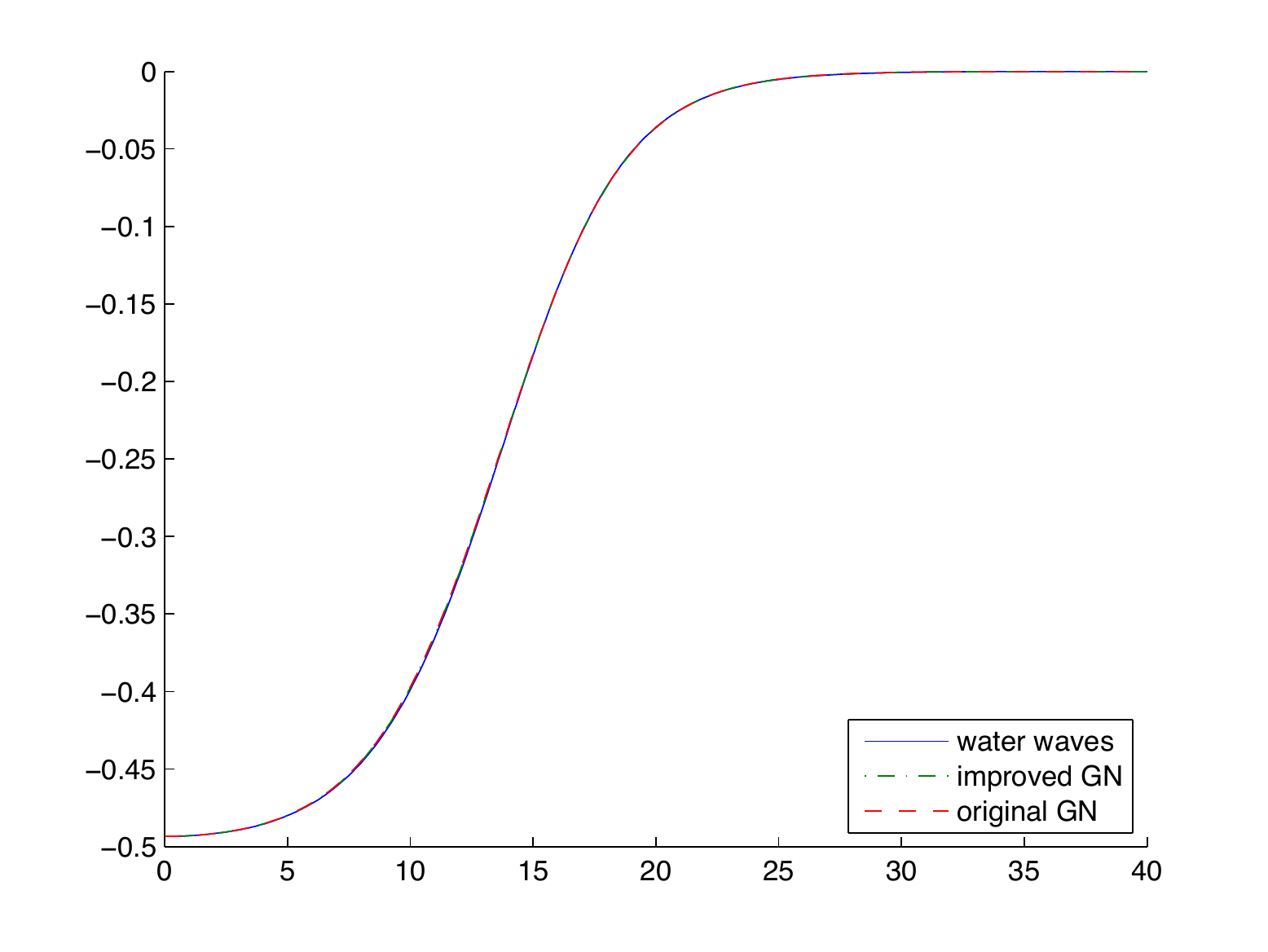}
\label{F.run2-largeq}
}
\caption{Comparison of the bilayer Green-Naghdi models and the water waves system ($\gamma=1,\delta=1/2$).}
\end{figure}

The solitary-wave solutions of the Miyata-Choi-Camassa system have been studied in the original papers of~\cite{Miyata87,ChoiCamassa99}. In particular we know that for a given amplitude, or a given velocity, there exists at most one solitary wave (up to spatial translations). The solitary waves are of elevation if $\delta^2-\gamma>0$, of depression if $\delta^2-\gamma<0$, and do not exist if $\delta^2=\gamma$. Contrarily to the one-layer situation, the bilayer Green-Naghdi model admits solitary waves only for a finite range of velocities (resp. amplitudes), $c\in(1,c_{\rm max}(\gamma,\delta))$ (resp. $|a|\in(0,a_{\rm max}(\gamma,\delta))$). With our choice of parameters (namely $\gamma=1,\delta=1/2$), one has
\[c_{\rm max}=\sqrt{1+1/8}\approx 1.06066 \quad \text{ and } \quad  |a_{\rm max}|=1/2.\]
As the velocity approaches $c_{\rm max}$, the solitary waves broadens and its mass keeps increasing. These type of profiles or often refered to as ``table-top'' profiles, and lead to bore profiles in the limit $c\to c_{\rm max}$.

When the velocity is small the numerically computed solitary wave solutions of the bilayer original ($\F_i=1$) and full dispersion ($\F_i=\F_i^{\rm imp}$) Green-Naghdi systems and the one of the water waves systems (and to a lesser extent the {\rm KdV} model) agree, so that the curves corresponding to the three former models are indistinguishable in see Figure~\ref{F.run2-smallq}. For larger velocities, as in Figure~\ref{F.run2-largeq}, the numerically computed solitary wave solutions of the Green-Naghdi and water waves systems is very different from the $\sech^2$ profile of the solitary wave solution to the Korteweg-de Vries equation. It is interesting to see that both the original and full dispersion Green-Naghdi models offer good approximations, even in this ``large velocity'' limit (the normalized $l^2$ difference of the computed solutions is $\approx 2.10^{-3}$ in both cases). This means that the internal solitary wave keeps a long-wave feature even for large velocities. These observations were already documented and corroborated by laboratory experiments in~\cite{MichalletBarthelemy98,GrueJensenRusaasEtAl99,CamassaChoiMichalletEtAl06}.

\paragraph{One-layer setting} 

\begin{figure}[!b]
 \subfigure[Rescaled solitary waves for $c=1.025,\,1.01,\,1.002$.]{
\includegraphics[width=.5\textwidth]{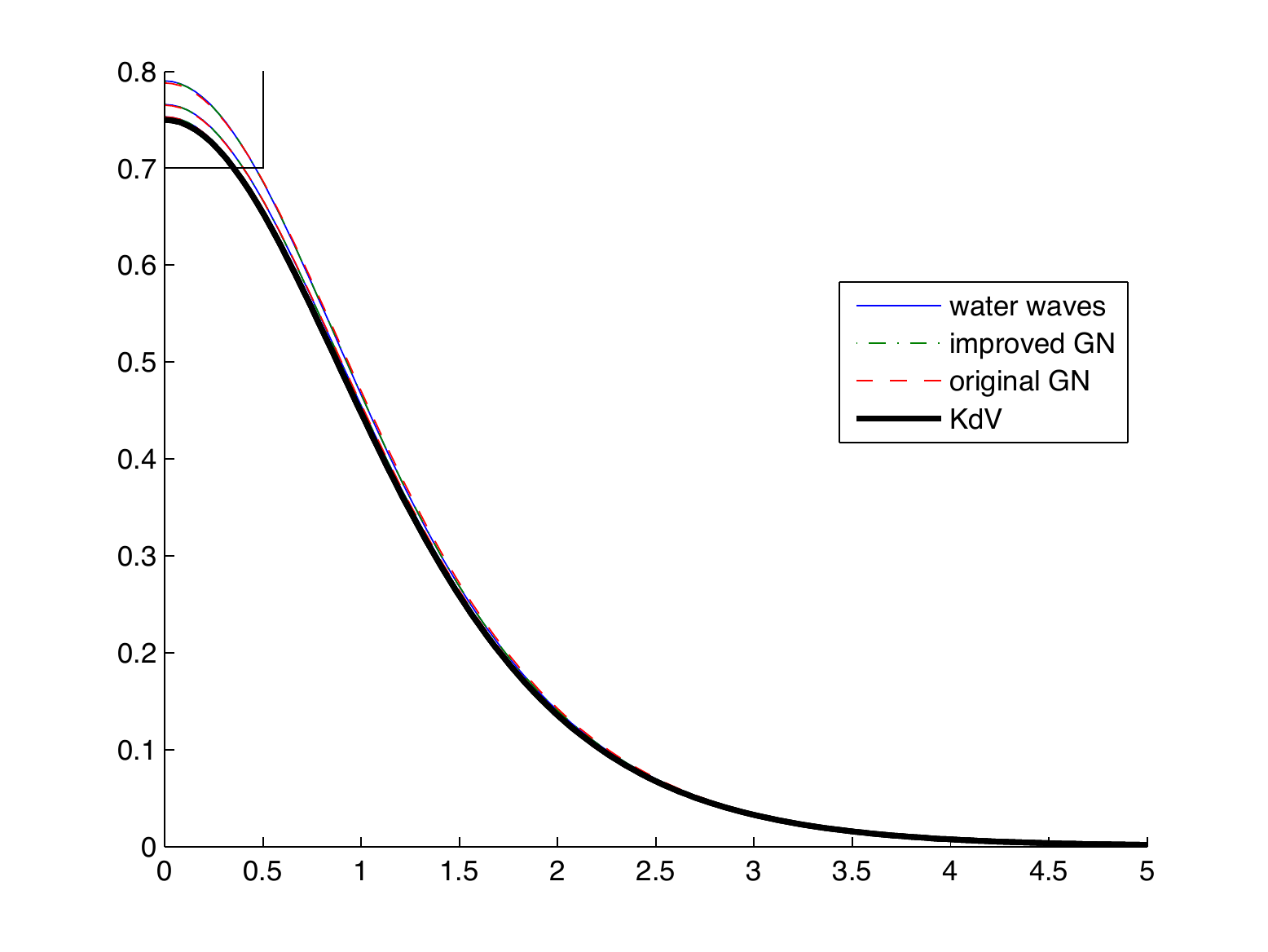}
}%
\subfigure[Close-up.]{
\includegraphics[width=.5\textwidth]{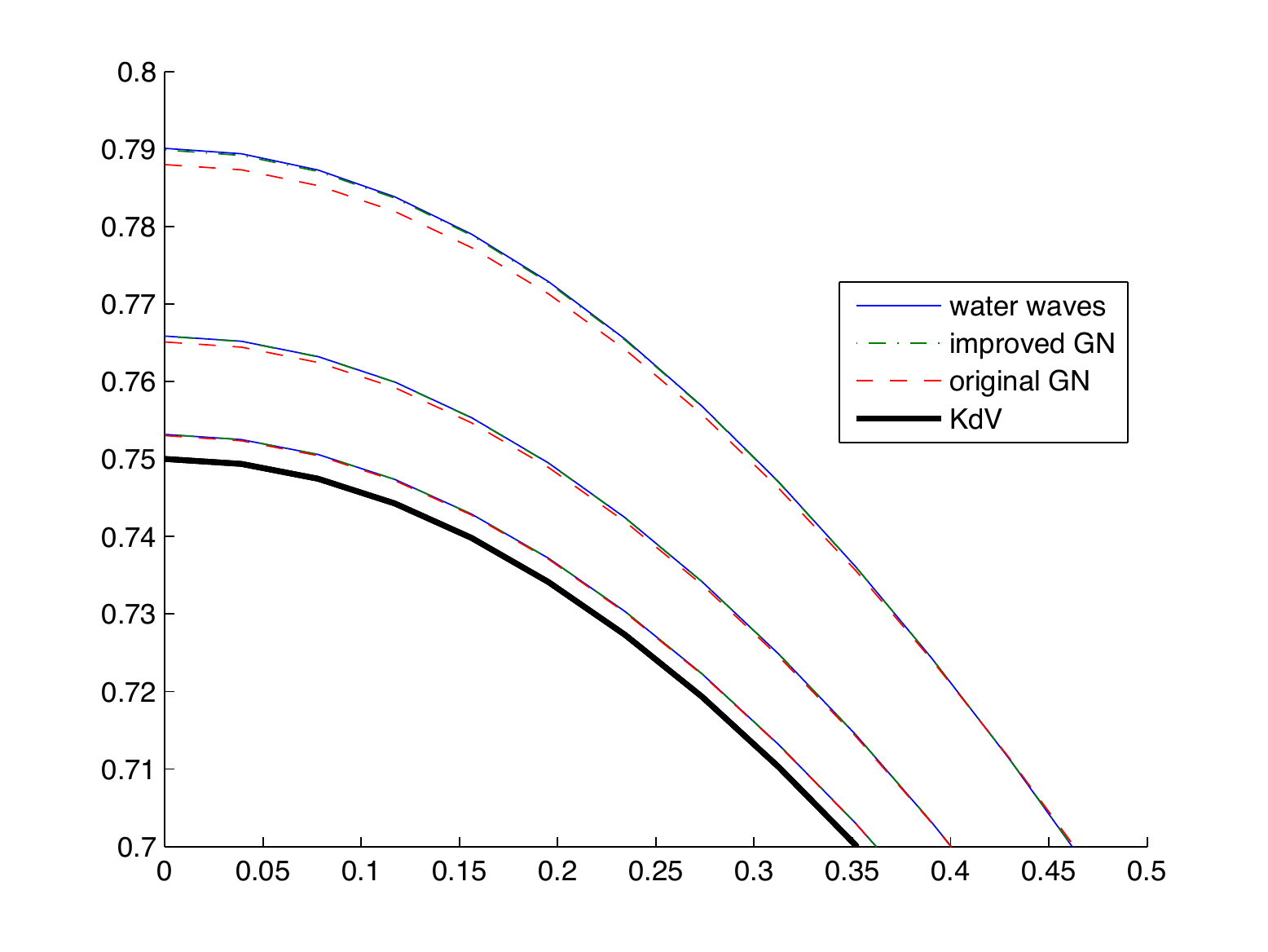}
}
\caption{Comparison of the solutions of the {\rm KdV} and Green-Naghdi models and the water waves system in the one-layer setting ($\gamma=0,\delta=1$).}
\label{F.run1-smallq}
\end{figure}
\begin{figure}[htp]
\centering\includegraphics[width=.5\textwidth]{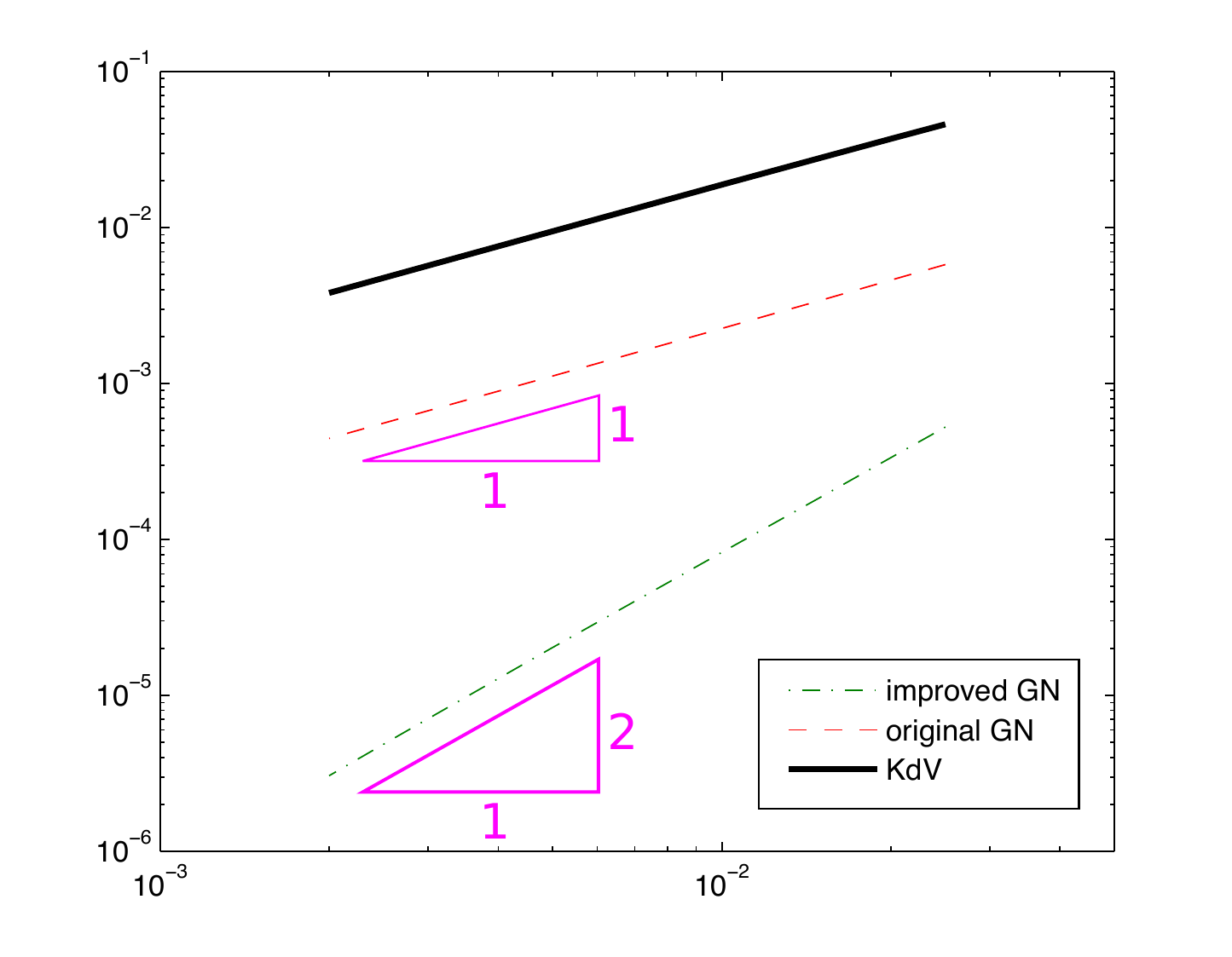}
\caption{Convergence rate. Log-log plot of the normalized $l^2$ norm of the error as a function of $c-1$.}
\label{F.run1-smallq-rate}
\end{figure}

In the one-layer setting, the script by Clamond and Dutykh~\cite{ClamondDutykh13} allows to have a very precise numerical computation of the water waves solitary solution, from which the numerical solutions of the Green-Naghdi models can be compared. In this setting, namely $\gamma=0$ and $\delta=1$, we have an explicit solution for the Green-Naghdi model~\cite{Serre53a}:
\[ \zeta_{\rm GN}(x)=(c^2-1)\sech^2(\frac12\sqrt{3\frac{c^2-1}{c^2}}x) = c^2 \zeta_{\rm KdV}(x).\]
In Figure~\ref{F.run1-smallq}, we compute the solitary waves for our models with different (small) values of the velocity, rescaled by $S_{\rm KdV}^{-1}$. One clearly sees, as predicted by Theorem~\ref{T.asymptotic-result} and the above formula, that the solitary waves converge towards $\xi_{\rm KdV}$ after rescaling, as $c\searrow 1$. One also sees that the water waves solution is closer to the one predicted by the model with full dispersion than the original Green-Naghdi model. Figure~\ref{F.run1-smallq-rate} shows that the convergence rate is indeed quadratic for the full dispersion model whereas it is only linear for the original Green-Naghdi model (and therefore only qualitatively better than the {\rm KdV} model).

\appendix

\section{Paradifferential calculus}
\label{S.para}

The definitions and properties below are collected from~\cite{Metivier08}; see also~\cite{Chemin98,Benzoni-GavageSerre07} for relevant references.

\begin{Definition}[Symbols]\label{D.symbols}
Given $m\in\RR$ and $r\geq 0$, we denote $\Gamma^m_r$ the space of distributions $a(x,\xi)$ on $\RR^2$ such that for almost any $x\in\RR$, $\xi\mapsto a(x,\xi)\in \C^\infty(\RR)$, and
\[\forall \alpha\in\NN, \exists C_\alpha>0 \quad \text{ such that } \quad \forall \xi\in\RR, \quad \Norm{\partial^\alpha_\xi a(\cdot,\xi)}_{W^{r,\infty}}\leq C_\alpha (1+|\xi|)^{m-\alpha},\]
where $W^{r,\infty}$ denote the H\"older space (Lipschitz for integer values).
\end{Definition}

Below, we use an {\em admissible cut-off} function $\psi$ in the sense of~\cite[Definition 5.1.4]{Metivier08} and define paradifferential operators as follows (the constant factor depends on the choice of convention for the Fourier transform).

\begin{Definition}[Paradifferential operators]\label{D.operator}
For $a\in \Gamma^m_0$ and $u\in \mathcal S(\RR)$, we define
\[T_a u(x)\eqdef \frac1{\sqrt{2\pi}}\left\langle \hat u(\cdot),e^{\mathrm{i}x \cdot }\psi(D,\cdot)a(x,\cdot)\right\rangle_{(\mathcal{S}(\mathbb{R}),\mathcal{S}'(\mathbb{R}))},
\]
where $\psi(D,\xi)$ is the Fourier multiplier associated with $\psi(\eta,\xi)$ (here, $\xi$ is a parameter). The operator is defined for $u\in H^s(\RR)$ by density and continuous linear extension. 
\end{Definition}
The following lemma is a direct application of the above definitions~\cite[Theorem 5.1.15]{Metivier08}.
\begin{Lemma}\label{Lemma:action}
For any $r\geq 0$ and  $a\in \Gamma_r^m\subset \Gamma_0^m$, and for all $s\in\RR$, the operator $T_a$ extends in a unique way to a bounded operator from $H^{s+m}$ to $H^s$. 

If $a(\xi)$ is a symbol independent of $x$, then $T_a=a(D)$, the corresponding Fourier multiplier.
\end{Lemma}
The main tool we use is the following composition property~\cite[Theorem 6.1.1]{Metivier08}
\begin{Lemma}
\label{Lemma:composition}
Let $a\in \Gamma_r^m$ and $b\in \Gamma_r^{m'}$ where $0<r\leq 1$. Then $ab\in \Gamma_{r}^{m+m'}$ 
and $T_{a} T_{b}-T_{ab}$ is a bounded operator from $H^{s+m+m'-r}$ to $H^s$, for any $s\in\RR$.
\end{Lemma}

Of particular interest is the case when the symbol $a(x)\in L^\infty$ is independent of $\xi$. The admissible cut-off function can be constructed so that the paraproduct $T_au$ corresponds to a standard Littlewood-Paley decomposition of the product $au$. This allows to show that $au-T_au$ is a smoothing operator provided that $a$ is sufficiently regular.

\begin{Lemma}
\label{Lemma:paraproduct}
Let $v\in H^s$ and $u\in H^t$, and $r\geq 0$. Then $uv-T_v u \in H^{r} $ provided that $s+t\geq 0,\ s\geq r$ and $s+t>r+1/2$.
\end{Lemma}
The definitions of the paraproduct in~\cite{Chemin98} and~\cite{Metivier08} differ slightly but it is not hard to show that~\cite[Theorem 2.4.1]{Chemin98} still holds for the paraproduct as it is defined in~\cite{Metivier08}, and Lemma~\ref{Lemma:paraproduct} follows directly from this theorem. 

We conclude with the following lemma, displayed in~\cite[Theorem 5.2.4]{Metivier08}
\begin{Lemma}
\label{Lemma:paralinearization}
Let $G\in \C^\infty(\mathbb R)$ be such that $G(0)=0$. If $u\in H^s$ with $s>1/2$, then $G(u)-T_{G'(u)} u \in H^{2s-1/2}$.
\end{Lemma}

\noindent\\
{\bf Acknowledgements.} V.~Duch\^{e}ne was partially supported by the Agence Nationale de la Recherche (project ANR-13-BS01- 0003-01 DYFICOLTI). D.~Nilsson and E.~Wahl\'{e}n were supported by the Swedish Research Council (grant no. 621-2012-3753).

\bibliographystyle{siam}
\def\cprime{$'$}

\end{document}